\documentclass[11pt,a4paper,reqno]{amsart}
\usepackage{amsmath,amssymb,amsthm,graphicx,color}
\usepackage[left=1in,right=1in,top=1in,bottom=1in]{geometry}
\usepackage{cite}
\usepackage[british]{babel}
\usepackage{enumerate}
\usepackage{hyperref} 
\usepackage{bm}
\usepackage{comment}

\usepackage{tikz}
\usetikzlibrary{calc,arrows}

\hypersetup{
    colorlinks=false,
    pdfborder={0 0 0},
}

\newtheorem{theorem}{Theorem}[section]
\newtheorem{corollary}[theorem]{Corollary}
\newtheorem{proposition}[theorem]{Proposition}
\newtheorem{lemma}[theorem]{Lemma}

\numberwithin{equation}{section}

\theoremstyle{definition}

\theoremstyle{remark}
\newtheorem{remark}[theorem]{Remark}
\newtheorem{remarks}[theorem]{Remarks}
\newtheorem*{remark*}{Remark}

 \allowdisplaybreaks

\newcommand{\1}[1]{{\mathbf 1}{\{#1\}}}
\newcommand{\2}[1]{{\mathbf 1}{(#1)}}

\newcommand{\R}{{\mathbb R}}
\newcommand{\Q}{{\mathbb Q}}

\newcommand{\N}{{\mathbb N}}
\newcommand{\ZP}{{\mathbb Z}_+}
\newcommand{\RP}{{\mathbb R}_+}
\newcommand{\Sp}[1]{{\mathbb S}^{#1}}
\newcommand{\fA}{{\mathfrak A}}
\newcommand{\X}{{\mathbb X}}

\DeclareMathOperator{\Exp}{\mathbb{E}}
\renewcommand{\Pr}{{\mathbb P}}

\DeclareMathOperator{\sign}{sgn} 
\DeclareMathOperator{\trace}{tr}

\DeclareMathOperator{\divg}{div}
\DeclareMathOperator{\grad}{grad}

\def\diag#1{\mathop{\mathrm{diag}}\left( #1 \right)}

\newcommand{\tra}{{\scalebox{0.6}{$\top$}}}

\newcommand{\eps}{\varepsilon}
\newcommand{\la}{\lambda}
\newcommand{\La}{\Lambda}

\newcommand{\re}{{\mathrm{e}}}

\newcommand{\ud}{{\mathrm d}}

\newcommand{\cC}{{\mathcal C}}
\newcommand{\cD}{{\mathcal D}}
\newcommand{\cE}{{\mathcal E}}
\newcommand{\cF}{{\mathcal F}}
\newcommand{\cG}{{\mathcal G}}

\newcommand{\tX}{\widetilde X}
\renewcommand{\tt}{\varphi}

\newcommand{\as}{\ \text{a.s.}}

\newcommand{\toP}{\overset{\mathrm{P}}{\longrightarrow}}

\newcommand{\bx}{{\mathbf{x}}}
\newcommand{\by}{{\mathbf{y}}}
\newcommand{\bz}{{\mathbf{z}}}
\newcommand{\bu}{{\mathbf{u}}}
\newcommand{\bv}{{\mathbf{v}}}
\newcommand{\be}{{\mathbf{e}}}
\newcommand{\bc}{{\mathbf{c}}}
\newcommand{\bp}{{\mathbf{p}}}
\newcommand{\0}{{\mathbf{0}}}

\newcommand{\bra}{\langle}
\newcommand{\ket}{\rangle}

\newcommand{\Besq}{{\mathrm{BESQ}}}
\newcommand{\Bes}{{\mathrm{BES}}}

\newcommand{\ssym}{\sigma_{\mathrm{sy}}}

\newcommand{\x}{{\mathcal{X}}}
\newcommand{\y}{{\mathcal{Y}}}

\makeatletter
\def\namedlabel#1#2{\begingroup  
    (#2)%
    \def\@currentlabel{#2}%
    \phantomsection\label{#1}\endgroup
}
\makeatother

\begin{document}
\title[Invariance principle for non-homogeneous random walks]{Invariance principle for non-homogeneous 
random walks} 
\author{Nicholas Georgiou}
\address{Department of Mathematical Sciences, Durham University, UK}
\email{nicholas.georgiou@durham.ac.uk}

\author{Aleksandar Mijatovi\'c}
\address{Department of Mathematics, KCL, \& The Alan Turing Institute, UK}
\email{aleksandar.mijatovic@kcl.ac.uk}

\author{Andrew R.\ Wade}
\address{Department of Mathematical Sciences, Durham University, UK}
\email{andrew.wade@durham.ac.uk}




\begin{abstract}
We prove an invariance principle for a class of zero-drift spatially non-homogeneous 
random walks in $\R^d$, which may be recurrent in any dimension.
The limit $\x$ is an elliptic martingale diffusion, which may be point-recurrent at the origin
for any $d\geq2$.
  To characterise $\x$,
we introduce a (non-Euclidean) Riemannian metric on the unit sphere in $\R^d$ and use it 
to express a related spherical diffusion as a Brownian motion with drift. 
This representation 
allows us to establish the skew-product decomposition of the excursions of $\x$ 
and thus develop the excursion theory of $\x$ without appealing to the strong Markov property.
This leads to the uniqueness in law of the stochastic differential equation
for $\x$ in $\R^d$, whose coefficients are discontinuous at the origin.
Using the Riemannian metric we can also detect whether the angular component 
of the excursions of $\x$ is time-reversible. If so, the excursions of $\x$  in $\R^d$ generalise 
the classical Pitman--Yor splitting-at-the-maximum property of Bessel excursions.
\end{abstract}

\maketitle

\section[Introduction]{Introduction}
\label{sec:intro}
A large class of spatially 
non-homogeneous zero-mean random walks on $\R^d$ ($d\geq2$),
which may be  
recurrent for $d\geq3$ and transient for $d=2$,
is introduced and analysed in~\cite{gmmw}.
These walks are martingales 
with  uniformly non-degenerate increments
(see assumptions~\eqref{ass:zero_drift}--\eqref{ass:unif_ellip} below).
It turns out that the information for the transience/recurrence classification is contained in
the limiting covariance structure of their increments, 
described by a matrix-valued function 
$\sigma^2:\Sp{d-1}\to\R^d\otimes \R^d $
on the unit sphere $\Sp{d-1}$ in $\R^d$ (see assumptions \eqref{ass:cov_limit}--\eqref{ass:cov_form} below).

This paper studies \emph{scaling limits} of these
random walks. We prove that under diffusive scaling,
the random walk converges weakly to a diffusion
process $\x = (\x_t, t \in\RP)$ whose law is determined uniquely by 
$\sigma^2$ via the stochastic differential equation (SDE)
\begin{equation}
\label{eqn:x-SDE}
\ud \x_t = \sigma (\hat \x_t) \ud W_t,\qquad \x_0=\bx_0\in\R^d.
\end{equation}
Here
$\hat \bx$ is the radial projection onto $\Sp{d-1}$ of any $\bx\in\R^d$ (with an arbitrary choice $\hat \0\in\Sp{d-1}$
for the origin $\0$),
$(W_t, t \geq 0)$ denotes a standard Brownian motion (BM) on $\R^d$, $\sigma:\Sp{d-1}\to\R^d\otimes\R^d$ 
is a \emph{square root} of $\sigma^2$ 
(i.e., $\sigma(\bu)\sigma^\tra(\bu) = \sigma^2(\bu)$ for all $\bu \in \Sp{d-1}$)
and $\bx_0$ a non-random point.

\begin{theorem}
\label{t:well-posedness}
Let the positive-definite symmetric
matrix-valued function $\sigma^2:\Sp{d-1}\to\R^d\otimes\R^d $ satisfy
\eqref{ass:cov_form}--\eqref{ass:radial-evec} below. 
Then, for any starting point $\x_0=\bx_0$ in $\R^d$, weak existence and uniqueness in law
hold for SDE~\eqref{eqn:x-SDE}
and the strong Markov property is satisfied. 
Moreover, 
the law of $\x$ does not depend on the choices of the
square-root $\sigma$ 
and 
$\hat\0\in\Sp{d-1}$.
\end{theorem}

The process $\x$ possesses certain universal properties, in some aspects resembling those of a BM on $\R^d$. 
The key difference is that, due to the possible recurrence of the random walk in any dimension $d\geq2$, 
the scaling limit $\x$ may visit the origin infinitely often. Since the diffusion coefficient is discontinuous
at $\0$, the proof of the uniqueness in law requires the development of the excursion theory of $\x$ before the
strong Markov property can be established. 
This step constitutes the main technical contribution of the paper (see Section~\ref{sec:excursions}
below) and provides an insight into the structure of the excursion of $\x$. 
It rests on the introduction of a (non-Euclidean) Riemannian metric
on $\Sp{d-1}$ (Section~\ref{subsubsec:RiemannianGeom} below), 
yielding a skew-product decomposition of the excursions of $\x$, which in turn
entails a generalisation of Stroock's representation of the spherical BM~\cite[p.~83]{hsu} (see~\eqref{sde-on-sphere} below).
The new geometry on the sphere also yields a multi-dimensional generalisation of the 
splitting-at-the-maximum property of 
Bessel excursions~\cite{py}. 
Furthermore, the choice of the square root of $\sigma^2$
turns out to be relevant for the pathwise uniqueness of SDE~\eqref{eqn:x-SDE}, which may fail, thus generalising 
to higher dimensions the example of Stroock and Yor~\cite{sy} for the complex BM. 
These and other features of the law of $\x$ are described in more detail in 
Section~\ref{subsec:Limit_comments}  below.
The proof of 
Theorem~\ref{t:well-posedness} is in  
Section~\ref{sec:diff_proc} 
with overview in Section~\ref{sec:overview}. 

Having characterised the scaling limit, we state  
our invariance principle. 
For a discrete-time process 
$X=(X_m,m\in\ZP)$,  
any $n \in \N$ and $t \in \RP$, define $\lfloor nt \rfloor:=\max\{k\in\ZP:k\leq nt\}$ and 
\begin{equation}
\label{eq:scaled_walk}
\tX_n (t) := n^{-1/2} X_{\lfloor nt \rfloor} .\end{equation} 
The paths of $\tX_n=(\tX_n(t),t\in\RP)$ are in  the  Skorohod space
$\cD_d = \cD(\RP ; \R^d)$ of right-continuous functions with left limits,
endowed with the Skorohod metric (see e.g.~\cite[\S 3.5]{ek}).

\begin{theorem}
\label{thm:invariance}
Let \eqref{ass:moments}--\eqref{ass:radial-evec} below hold for the random walk $X$.
Let $\x$ be the unique (weak) solution of~\eqref{eqn:x-SDE}
with $\x_0 = \0$. 
Then, as $n \uparrow \infty$,
the weak convergence 
$\tX_n \Rightarrow  \x$
on $\cD_d$ holds. 
\end{theorem}

The class of random walks satisfying~\eqref{ass:moments}--\eqref{ass:radial-evec} 
consists of $\R^d$-valued Markov chains with an asymptotically stable increment covariance structure.
Thus
Theorem~\ref{thm:invariance} may be viewed as a  multi-dimensional
generalisation of the classical  invariance principle of Lamperti~\cite{lamp2}
for $\RP$-valued Markov chains with asymptotically constant variance of the increments. 
The proof of Theorem~\ref{thm:invariance}  
hinges on the radial invariance principle 
in~\cite{gmw}
and a $d$-dimensional invariance principle for martingale diffusions
with discontinuous coefficients given in Theorem~\ref{thm:conditional_invariance} below.
Invariance principles with continuous coefficients,
such as~\cite[Thm~7.4.1, p.~354]{ek}, do not apply in our setting
(both formally and) because,  
by Corollary~\ref{thm:law-of-excursions} below, the process $\x$
may hit the discontinuity point $\0$ infinitely many times. 
In order to deal with the point-recurrence of $\x$, 
it is necessary to control the amount of time $\x$ spends near $\0$. 
This is achieved via the occupation times formula 
and the analysis of the local time of the 
radial component of $\x$ (see proof of Lemma~\ref{lem:a_integral_convergence} below). 
Note that
neither
the specific form of the law of the radial component 
nor the fact that $\x$  has no drift are crucial for the validity of Theorem~\ref{thm:conditional_invariance}. 
Some consequences of
Theorem~\ref{thm:invariance} for random walks are in Section~\ref{sec:consequences} below.
Its proof 
is in Section~\ref{sec:invariance} below.

\subsection{The diffusion limit}
\label{subsec:Limit_comments}
A natural ellipticity condition for 
$\sigma^2:\Sp{d-1}\to\R^d\otimes\R^d$ in~\cite{gmmw} 
(see~\eqref{ass:cov_form} below)
requires constant total  
$\trace \sigma^2(\bu)  = V$ 
and radial
$\bra \bu ,\sigma^2(\bu) \bu \ket = U$
instantaneous variances for all
$\bu \in \Sp{d-1}$
and some positive reals $U<V$.
Further assumptions on $\sigma^2$ 
in Theorem~\ref{t:well-posedness}
are smoothness~\eqref{ass:sigma_smooth}
and a structural condition
$\sigma^2 (\bu)\bu=U\bu$ for all $ \bu \in \Sp{d-1}$ (\eqref{ass:radial-evec} below),
which ensures the existence of a skew-product decomposition
of excursions of $\x$. 


\subsubsection*{$\x$ is a self-similar Markov process on $\R^d$ (with Brownian scaling).}
The process $\|\x\|/\sqrt{U}$ is Bessel of dimension $V/U > 1$ (see Lemma~\ref{l:radial-bessel} below).
Hence, if $V/U \in(1, 2]$ (resp. $V/U>2$), then  $\liminf_{t \to \infty} \| \x_t\| = 0$ 
(resp.  $\lim_{t \to \infty} \| \x_t \| = \infty$) and
the origin~$\0$ is recurrent for $\x$ if and only if $V/U < 2$.
(The Foster--Lyapunov criteria~\cite[Thm~6.2.1]{pinsky} 
do not apply,
even if Theorem~\ref{t:well-posedness} has been established, 
since $\bx\mapsto\sigma^2(\hat \bx)$ is discontinuous.) 
Let $\Pr_{\bx_0}$ be the law of $\x$ started at 
$\x_0 = \bx_0\in\R^d$. Define $\y=(\y_t,t\geq0)$, $\y_t :=  c\x_{c^{-1/2}t}$, for some constant $c >0$. Then 
the scale invariance of  
$\bx \mapsto \sigma(\hat \bx)$ and $W$ in~\eqref{eqn:x-SDE}
imply that 
$\y$ solves SDE~\eqref{eqn:x-SDE} with $\y_0 = c\bx_0$. 
By Theorem~\ref{t:well-posedness},
the law of $\y$ equals $\Pr_{c\bx_0}$, making $\x$ a globally defined self-similar Markov process on $\R^d$,
which may hit $\0$ infinitely many times. 

\subsubsection*{A stationary diffusion $\psi$ on $\Sp{d-1}$}
Consider the following Stratonovich SDE on $\Sp{d-1}$,
\begin{equation}
\label{sde-on-sphere}
\ud \phi_t = (\ssym ( \phi_t) - \phi_t \phi_t^\tra ) \circ \ud W_t 
-(I - \phi_t \phi_t^\tra ) A_0(\phi_t) \ud t, 
\end{equation}
where $W$ is a standard BM on $\R^d$,
$\ssym$
is the unique positive-definite square root of $\sigma^2$, which is hence smooth by 
Lemma~\ref{lem:s_sym_uniformly_ell} below, and the vector field $A_0$
is a linear combination of the derivatives of the columns of $\ssym$ defined in Section~\ref{sec:Sp-diffusion} 
below. By Lemma~\ref{lem:Rd-not-0-SDE} below, SDE~\eqref{sde-on-sphere} has a unique strong solution 
on $\Sp{d-1}$. In the case  
$\sigma^2=\ssym=I$,
SDE~\eqref{sde-on-sphere}
clearly reduces to Stroock's representation of the BM on $\Sp{d-1}$ with the Riemannian metric induced by
the ambient Euclidean space~\cite[p.~83]{hsu}
($\x$ in this case is a BM on $\R^d$).

The key ingredient of the excursion measure of $\x$
is the stationary distribution $\mu$ on $\Sp{d-1}$ of the solution 
$\phi$ of~\eqref{sde-on-sphere}. 
In order to analyse $\phi$ and characterise
$\mu$, it turns out to be essential to modify the geometry on 
$\Sp{d-1}$
via the Riemannian metric
$g_\bx(v_1,v_2):=\bra \sigma^{-2}(\bx)v_1,v_2\ket$,
where $\bx\in\Sp{d-1}$,  $v_1,v_2\in\R^d$ are in the tangent space of 
$\Sp{d-1}$
at $\bx$
and $\bra\cdot,\cdot\ket$ is the inner product on $\R^d$.
On the Riemannian manifold 
$(\Sp{d-1},g)$,
by 
Lemma~\ref{lem:Rd-not-0-SDE},
$\phi$
is a BM with drift, generated by  
$\cG=(1/2)\Delta_g +V_0$,
where
$\Delta_g$ is the Laplace-Beltrami operator
and $V_0$ is a tangential vector field on $\Sp{d-1}$, explicit in $\sigma^2$ and
its derivatives of order one.
Prop.~\ref{lem:sphere-in-Rd-SDE} states that
the stationary measure $\mu$ is unique. Its proof shows that 
in fact $\mu (\ud \bx ) = p(\bx) \ud_g \bx$,
where
$p:\Sp{d-1}\to\R$ 
is a strictly positive density with respect to the Riemannian volume element 
$\ud_g \bx$ on $(\Sp{d-1}, g)$ (see e.g.~\cite[p.~291]{iw} for definition), 
uniquely determined by the PDE $\cG^*p=0$ 
with $\cG^*$ denoting the  adjoint of $\cG$ on $L^2(\Sp{d-1};\ud_g \bx)$.
Recall that for any vector field $V$ on $\Sp{d-1}$,
$\divg V$
is the trace of the endomorphism of the tangent space given by the 
directional derivatives of $V$ via the 
Levi-Civita connection 
and,
for any smooth $f$ on $\Sp{d-1}$,
we have
$\Delta_g f = \divg(\grad(f))$ (see Sec.~\ref{subsubsec:RiemannianGeom} below). 
Integration by parts
implies that 
$p$ is the unique positive solution of the PDE 
\begin{equation}
\label{eq:mu-density} \frac{1}{2} \Delta_g p - \divg ( p V_0 ) = 0,
\qquad\text{ satisfying $\int_{\Sp{d-1}} p (\bx) \ud_g \bx=1$.}
\end{equation}

We can now define a stationary solution $\psi$ of~\eqref{sde-on-sphere},
indexed by $\R$, with law $\Pr_\Psi$
(see Prop.~\ref{lem:sphere-in-Rd-SDE} below). 
Assuming
$V_0 = \grad F_0$
for a smooth $F_0:\Sp{d-1}\to\R$,
the definition of 
$\grad F_0$
on 
$(\Sp{d-1}, g)$
in
Section~\ref{subsubsec:RiemannianGeom} below
implies that $p := \exp(2F_0)/\int_{\Sp{d-1}}\exp(2F_0(\bx))\ud_g \bx$
is the unique solution of~\eqref{eq:mu-density}.
Moreover,
by~\cite[Thms~4.2 \&~6.1]{kent}, SDE~\eqref{sde-on-sphere}
is time reversible:
for any random time $T\in\R$, independent $\psi$, 
the process $(\psi_{T-t},t\in\RP)$ 
solves~\eqref{sde-on-sphere} started according to the law $\mu$.
In particular,
if 
$F_0\equiv0$, then 
$\psi$ is the standard stationary spherical BM
and the measure $\mu$ is uniform. 

\subsubsection*{Transient case: skew-product decomposition of $\x$}
Suppose that $2 < V/U$. If $\x_0 \neq \0$, 
a Bessel process $r/\sqrt{U}$ of dimension $V/U$ 
(with $r_0=\|\x_0\|$)
is strictly positive and we may define  
$\rho_s  (t) = \int_s^t r_u^{-2} \ud u$ for $t, s \geq 0$. 
Then the process 
$(r_t \phi_{\rho_0 (t)},t\in\RP)$,
where 
the solution $\phi$ of SDE~\eqref{sde-on-sphere},
started at 
$\phi_0 = \hat \x_0$,
and
$r$ are independent,
has the same law as $\x$ (see 
Section~\ref{sec:transient_case} below).

The relevant case for Theorem~\ref{thm:invariance}
is $\x_0 = \0$. 
As $\x$ starts from $\0$ and never returns, a natural description of its law
is via a family of entrance laws at positive times $s$ and the subsequent evolution. 
The latter is given in terms of a Bessel process and a time-changed angular process
solving~\eqref{sde-on-sphere}
as above: 
$(r_t \phi_{\rho_s (t)},t\geq s)$ with $\phi_0:=\hat \x_s$.
The random vector $\hat \x_s$ 
is forced to be 
\emph{independent} of $r_s$
and distributed according to
the \emph{stationary} law $\mu$ of $\phi$, due to the \emph{rapid spinning}
of the process $\x$ as it leaves $\0$: 
$\rho_s (t) \to \infty$ as $s \downarrow 0$ for fixed $t>0$ (see Lemma~\ref{l:independent_angle_transient} below). 
As
$\rho_s(t)=\rho_s(1)+\rho_1(t)$ for any $s,t>0$,
the processes 
$(r_t \psi_{\rho_1 (t)},t>0)$
and
$(\x_t,t>0)$
are equal in law, 
where $\psi$ 
and $r$ are independent. 
The analogy with the classical case of the skew product of BM on $\R^d$ 
in both cases
$\x_0\neq \0$ and $\x_0=\0$ 
(see~\cite[\S IV.35, p.~73]{rw2} and~\cite[p.~276]{itomckean}) is clear.  
Moreover,
in the polar case 
$V/U=2$, the skew product of $\x$ is analogous to the one in the transient case.

\subsubsection*{Point-recurrent case: skew-product decomposition of excursions of $\x$}
Assume $V/U\in(1,2)$ and $\x_0=\0$.
The process $\x$ returns to $\0$ infinitely often since $\|\x\|/\sqrt{U}$ is Bessel 
of dimension $V/U$. 
As the excursions of $\x$ turn out to exhibit the rapid spinning behaviour at each end, 
its excursion measure may be constructed as follows. 
Mark each Bessel excursion 
by an independent draw from the law 
$\Pr_\Psi$ on $\cC ( \R, \Sp{d-1} )$ given in Prop.~\ref{lem:sphere-in-Rd-SDE} below.
Since, due to rapid spinning at the beginning of each excursion of $\x$,
the angular component of the excursion is distributed according to the stationary 
measure $\mu$ of SDE~\eqref{sde-on-sphere} at all times, we need to map the marked Bessel excursion
by time-changing the mark $\psi$ via an additive functional of the Bessel excursion, see Section~\ref{subsub:Marked_Bessel_Ex}
below for details. 
Note that the mapping has to be defined for Bessel excursions lasting longer than $a$ (for any fixed $a>0$), since
the time-change can only be ``anchored'' at a pre-specified time during the life time of the excursion. 
Although this causes some technical difficulties, the mapped Poisson point processes can be interpreted consistently
(for all $a>0$). Its excursion measure turns out to be that of $\x$.

We stress that this construction of the excursion measure 
depends only on $\sigma^2$, which specifies the dimension of the Bessel process and hence its excursion measure
and determines the marks via SDE~\eqref{sde-on-sphere} (the mapping uses only the information
contained in the Bessel excursion). 
Moreover,  
the local time at $\0$ of $\x$
can be defined 
as that of $\|\x\|$ at $0$,
without a reference to the strong Markov property of $\x$.
Hence, once the excursion measure has been constructed (Section~\ref{subsub:Marked_Bessel_Ex} below),
the key step in the proof of Theorem~\ref{t:well-posedness} consists of establishing that (without the strong
Markov property) the point process of excursions of $\x$ is the Poisson point process with the excursion measure 
described above.  The details are in Section~\ref{subsub:Proof_of_weak_uniquenes_rec} below.

In the case $\x_0\neq\0$, 
up to the first hitting time of $\0$, 
the skew product of excursions coincides with
the generalised Lamperti representation for self-similar Markov processes
on $\R^d\setminus\{\0\}$~\cite{acgz}, 
where the L\'evy process is a scalar BM with drift and the angular component equals
the diffusion on $\Sp{d-1}$ in~\eqref{sde-on-sphere} started at $\hat\x_0$. 
Note also that there is a literature (see e.g.~\cite{vuolle1} and the reference therein) 
on the extensions of strong Markov processes on $\R^d\setminus\{\0\}$
with skew-product decomposition beyond the first hitting time of the origin,
of which $\x$ is an example.

\subsubsection*{Splitting excursions at the maximum: a generalised Pitman--Yor representation} 
If the vector field $V_0$ in~\eqref{eq:mu-density} has a potential, the excursions of $\x$
provide a multi-dimensional generalisation of the famous
Pitman--Yor~\cite{py} representation of the Bessel excursions 
with dimension $\delta=V/U \in (1,2)$. 
Let $U=1$
and recall from~\cite{py} that
the unique maximum $M$ of the Bessel excursion $e^r$
is drawn from the $\sigma$-finite density $m\mapsto m^{\delta-3}$ on the interval $(0,\infty)$.
Then,
conditional on $M$, the excursion $e^r$ is obtained by joining back to back 
two independent Bessel processes $\beta$ and $\beta'$ of dimension $4-\delta$, both
started at $0$ and run until the first times ($T_M$ and $T'_M$ respectively)  they hit $M$:
$e^r(t)=\1{t\in(0,T_M]}\beta_t+\1{t\in(T_M,T_M+T'_M)}\beta'_{T_M+T'_M-t}$.
A trivial (but crucial) observation is that when the maximum is reached, 
the process is neither at the beginning nor the end of the excursion. 
Hence, due to rapid spinning,  the angular component $\hat e^\x(T_M)$ of 
the corresponding excursion $e^\x$ of $\x$ at $T_M$ must 
follow the stationary law $\mu$ of SDE~\eqref{sde-on-sphere}.
As SDE~\eqref{sde-on-sphere} is time-reversible (see paragraph 
after~\eqref{eq:mu-density} above), the excursion 
$e^\x$ equals 
\begin{equation}
\label{eq:split_at_max}
e^\x(t)=\1{t\in(0,T_M]}\beta_t\phi_{\rho (T_M-t)}+\1{t\in(T_M,T_M+T'_M)}\beta'_{T_M+T'_M-t}\phi'_{\rho'(t-T_M)},
\end{equation}
where
$\phi,\phi'$
are solutions of SDE~\eqref{sde-on-sphere} with the same initial condition $\phi_0=\phi'_0$,
distributed according to $\mu$, and driven by independent BMs. The time-changes
$\rho (t) = \int_0^t \beta_{T_M-s}^{-2} \ud s$, $t \in (0,T_M]$,
and $\rho' (t) = \int_0^t \beta'^{-2}_{T'_M-s} \ud s$, $t \in [0,T'_M)$,
satisfy
$\lim_{t\downarrow0}\rho (T_M-t)=\lim_{t\uparrow T'_M}\rho' (t)=\infty$. 

In the limit as $U\uparrow V$, which is excluded from our results, 
the angular motion degenerates to a constant as the trace of $\sigma^2$ 
equals the radial eigenvalue. 
The radial part becomes the modulus of the scalar BM,
while rapid spinning and~\eqref{eq:split_at_max} suggest that 
the singular diffusion in the limit changes the ray it lives on every time it hits the origin 
according to a law on $\Sp{d-1}$, which is the limit of the stationary measures 
of SDE~\eqref{sde-on-sphere} as $V/U\downarrow 1$.
It hence appears that the liming singular diffusion is a generalisation of the Walsh BM
(or Brownian spider)~\cite{bpy} to $\R^d$.

\subsubsection*{Smooth square roots and pathwise uniqueness: the Stroock--Yor phenomenon}
SDE~\eqref{eqn:x-SDE} need not (but clearly could) possess pathwise uniqueness even if $\sigma^2$ is the identity
(consider $\sigma(\bu)=\diag{\sign(u_1),\ldots,\sign(u_d)}$
and recall the scalar Tanaka SDE~\cite[\S IX.1,~Ex.(1.19)]{ry}). 
This behaviour persists even for smooth square roots $\sigma$.
Below we give a generalisation of the SDE for complex Brownian 
motion in~\cite[Thm~3.12]{sy},  with the property that 
the failure of pathwise uniqueness occurs precisely when
the solution starts from (or visits) $\0$.

Note first that a simple application of the occupation times formula
and the fact that $\x_0=\0$ if and only if $\|\x_t\|=0$ imply
that 
if $\x$ solves SDE~\eqref{eqn:x-SDE} for a given
choice of $\hat\0$, then it also solves the SDE for any other choice
$\hat\0\in\Sp{d-1}$.
If a square root $\sigma$ satisfies (I)
$P\sigma(\bu)=\sigma(P\bu)$ for all $\bu\in\Sp{d-1}$,
where
$P\in SO(d)\setminus\{I\}$
\footnote{$SO(d)$ is the group of orientation-preserving orthogonal matrices in 
$\R^d\otimes\R^d$ and $I$ is the identity matrix.},
then It\^o's formula and the remark above imply
that for any solution $(\x,W)$ of~\eqref{eqn:x-SDE} started from $\0$,
the process $(\y,W)$, where $\y:=P\x$, is also a solution. 
By Theorem~\ref{t:well-posedness}, $\x$ and $\y$ have the same law but are clearly not equal. 
If, in addition, $\sigma$ satisfies 
(II) $\bu=\sigma(\bu)\bc$ for all $\bu\in\Sp{d-1}$
and some 
$\bc\in\Sp{d-1}$, 
the Brownian motion driving the process $\|\x\|$
equals $\bc^\tra W$ (Lemma~\ref{l:radial-bessel} below),
making $\|\x\|$ 
adapted to $W$.
Moreover, assuming $\x$ never visits $\0$,
the BM driving the angular component via
SDE~\eqref{sde-on-sphere} 
is a time-change of 
$\int_0^\cdot \|\x_s\|^{-1}\ud W_s$
(see~\eqref{eqn:tilde-W} and Proposition~\ref{prop:transient-away-from-0} below).
Hence the skew product $\|\x_t\| \phi_{\rho_0(t)}$, $t\in\RP$, 
where 
$\rho_0  (t) = \int_s^t \|\x_u\|^{-2} \ud u$, 
makes $\x$ a strong solution of~\eqref{eqn:x-SDE}.

It remains to exhibit a smooth $\sigma$ satisfying (I) and (II) above. Note first that~(I)
may only hold in even dimensions. We rely on the Lie group structure of the spheres 
in dimensions $d\in\{2,4\}$ 
for our examples. 
Pick a positive-definite $A\in\R^d\otimes\R^d$
and let $\sigma(\bu)=R(\bu)A$, where $R:\Sp{d-1}\to SO(d)$ is smooth.
For $d=4$, view $\Sp{3}$ as unit quaternions and define $R$ by
$R(\bu)\bv:=\bu\bullet \bv$, 
where 
$\bu\bullet \bv$ denotes the multiplication of quaternions 
$\bv\in\R^4$ 
and $\bu$
(see e.g.~\cite[p.~229]{rw2}).
It is easy to check that 
$R(\bu)\in SO(4)$ 
and
$R(\bu)\be_1=\bu$
for all 
$\bu\in\Sp{3}$,
where 
$\be_1$ is the first standard basis element of $\R^4$, i.e. the real quaternion. 
If in addition $A\be_1=\be_1$, then (II) holds.
Moreover, $\sigma(\bu)$ is a smooth square root of $\sigma^2(\bu)=R(\bu)A^2 R(\bu)^{-1}$.
Pick a unit quaternion $\bp\in\Sp{3}\setminus\{\be_1\}$ and define
$P:=R(p)\in SO(4)$.
The associativity of the product $\bullet$
yields the matrix identity $PR(\bu)=R(P\bu)$ for $\bu\in\Sp{3}$, 
implying (I).
Hence pathwise uniqueness fails when $\x_0=\0$.
Since $\sigma^2(\bu)\bu=\bu$, the process $\x$ hits $\0$ if and only if 
$\trace(\sigma^2(\bu))=\trace(A^2)\in(1,2)$ and we may choose independently a different 
rotation $P$ for each excursion, exhibiting uncountably many solutions of~\eqref{eqn:x-SDE}
for a fixed BM $W$. 
The complex case is analogous:
a BM in~\cite[Thm~3.12]{sy} solves~\eqref{eqn:x-SDE} with 
$\sigma(\bu)=R(\bu)$
a multiplication by $\bu\in\Sp{1}$.

\subsection{Angular convergence and the first exit out of large balls of the random walk.}
\label{sec:consequences}
We now describe the behaviour of the angular component of the random walk $X$ and  
its asymptotic law 
at $\tau_a^n:=\inf\{m\in\ZP:\|X_m\|\geq a\sqrt{n}\}$ its first exit 
out of the ball centred at $\0$ with radius $a\sqrt{n}$ (for some $a>0$).
Both statements are easy consequences of Theorem~\ref{thm:invariance}. 

Let $r$ be a Bessel process of dimension $\delta>1$, $r_0=0$,
and $\tau_a:=\inf\{t\in\RP:r_t=a\}$  (thus $\tau_a<\infty$ a.s). 
Recall that $\Pr[r_1\leq x]=\int_0^{x^2/2}z^{\alpha -1} \re^{-z} \ud z /\Gamma(\delta/2)$ for all $x\in\RP$~\cite[Cor.~XI.1.4]{ry},
where $\Gamma$ denotes the gamma function,  and 
$\Exp[\exp(-\lambda\tau_a)]=(a\sqrt{2\lambda})^\nu/(2^\nu \Gamma(\nu+1)I_\nu(a\sqrt{2\lambda}))$,
for any
$\lambda>0$, where $I_\nu$ denotes the modified Bessel function of the first kind of order $\nu:=(\delta-2)/2$
(see~\cite{kent_eigen} for a series expansion of the density of $\tau_a$ in terms of the zeros of Bessel functions).

\begin{corollary}
\label{cor:marginal}
Let the random walk $X$ satisfy the  assumptions of Theorem~\ref{thm:invariance}
with $U=1$
and 
define $\delta:=V$. 
Let the random vector
$\theta$
with the law $\mu$ on 
$\Sp{d-1}$, whose density satisfies~\eqref{eq:mu-density},
be independent of $r$. 
Then, as $n\to \infty$, the following weak limits hold:
\[ n^{-1/2} X_n \Rightarrow r_1 \theta \quad\text{(and hence $\hat X_n\Rightarrow \theta$)}
\qquad\text{and}\qquad 
(\tau_a^n/n,  n^{-1/2}X_{\tau_a^n})\Rightarrow (\tau_a,a\theta).\]
\end{corollary}

For a  continuous $f:\Sp{d-1}\to\R$, 
Cor.~\ref{cor:marginal} and~\cite[Thm~2.1]{bill}
imply
$\lim_{n \uparrow \infty} \Exp [ f(\hat X_n) ] = \int_{\Sp{d-1}}f\ud\mu$.
However, the ergodic average
$\frac{1}{n} \sum_{k=0}^{n-1} f(\hat X_k )$ cannot in general 
converge in probability to the constant $\int_{\Sp{d-1}}f\ud\mu$,
since
by
Theorem~\ref{thm:invariance}, an analogous argument to the one in the proof of 
Lemma~\ref{lem:a_integral_convergence} below
and~\eqref{eq:scaled_walk}, 
the average converges weakly
to a non-degenerate limit (for a non-constant function $f$): 
$\frac{1}{n} \sum_{k=0}^{n-1} f(\hat X_k )= \int_0^1 f(\hat\tX_n(t))\ud t\Rightarrow \int_0^1 f(\hat \x_t) \ud t$.

\begin{proof}
By~\eqref{eq:scaled_walk} and Theorem~\ref{thm:invariance}
we have 
$n^{-1/2} X_n=\tX_n(1)\Rightarrow \x_1$.
Since $\x_0=\0$, the skew product structure 
(Lem.~\ref{l:independent_angle_transient} (polar case) and Prop.~\ref{prop:time_excurion} (point-recurrent case)) 
yields the first limit. The mapping theorem~\cite[Thm.~5.1]{bill} implies the second
($\bx\mapsto  \hat\bx$ is continuous on $\R^d\setminus\{\0\}$ and $\Pr[\x_1=\0]=0$). 
Note that $\tau_a^n=\tau^a(\tX_n)$ and $\tau_a=\tau^a(r)$, where  $\tau^a(x)$, $x\in \cD_d$, is defined in~\eqref{eq:contact_time_def}.
As $r$ reaches new maxima immediately after $\tau_a$, $\lim_{b\to a}\tau^b(r)=\tau^a(r)$ holds a.s. 
By Lemma~\ref{lem:continuity_of_mapa}, Remark~(a) just after it, Theorem~\ref{thm:invariance} 
and~\cite[Thm.~5.1]{bill} the final limit holds. 
\end{proof}

\section{Assumptions}
\label{sec:rws}
Let $\{\be_1, \ldots, \be_d\}$ be the standard orthonormal basis 
in $\R^d$ ($d\geq2$)
with respect to 
the Euclidean inner product 
$\bra \cdot , \cdot \ket$
on $\R^d$,
and
$\Sp{d-1} := \{ \bu \in \R^d : \| \bu \| = 1\}$ the unit sphere in $\R^d$,
where
$\| \cdot \|$ is the Euclidean norm. 
For $\bx \in \R^d \setminus \{ \0 \}$
and the origin $\0$, 
let $\hat \bx := \bx / \| \bx \|$
and
$\hat \0 := \be_1$, respectively.  

Let $X=(X_n , n \in \ZP)$ be 
a discrete-time, time-homogeneous Markov process on an unbounded Borel subset $\X$ of $\R^d$.
Suppose $X_0$ is a non-random point in $\X$.
Denote the increments of $X$ by 
$\Delta_n := X_{n+1} - X_n$. 
Since the law of $\Delta_n$ depends only on 
$X_n$,
we often take  $n=0$ and write $\Delta$ for $\Delta_0$.
Let $\Pr_\bx [ \, \cdot \, ] = \Pr [ \, \cdot \,  \mid X_0 = \bx]$
and $\Exp_\bx [ \, \cdot \, ] = \Exp [ \, \cdot \,  \mid X_0 = \bx]$
denote the probabilities and expectations when the walk is started from $\bx \in \X$. 
We make the following assumptions.
\begin{description}
\item[\namedlabel{ass:moments}{A0}] Suppose that $\sup_{\bx \in \X} \Exp_\bx [ \| \Delta \|^4  ] < \infty$.
\end{description}
By~\eqref{ass:moments}, 
the mean $\mu(\bx) := \Exp_\bx [ \Delta ]$
and the covariance matrix
$M (\bx) := \Exp_\bx [ \Delta \Delta^{\!\tra} ]$
exist $\forall \bx\in\X$. 
\begin{description}
\item[\namedlabel{ass:zero_drift}{A1}] Suppose that $\mu(\bx) = \0$ for all $\bx \in \X$.
\end{description}
The next assumption ensures that $\Delta$ is \emph{uniformly non-degenerate}. 
\begin{description}
\item[\namedlabel{ass:unif_ellip}{A2}] There exists $v > 0$ such that $\trace M(\bx) = \Exp_\bx[ \| \Delta \|^2 ] \geq v$ for all $\bx \in \X$.
\end{description}
For a matrix $M\in\R^d\otimes \R^d $ define the norm
$\| M \| := \sup_{\bu \in \Sp{d-1}} \| M \bu \|$.
Throughout the paper, let $\sigma^2(\bu)$ be a positive-definite matrix for all  $\bu\in\Sp{d-1}$.
\begin{description}
\item[\namedlabel{ass:cov_limit}{A3}] Suppose that, as $r \to \infty$,
we have
$\eps(r) := \sup_{\bx \in \X : \| \bx \| \geq r} \| M( \bx ) - \sigma^2 ( \hat\bx ) \| \to 0$.
\end{description}
\begin{description}
\item[\namedlabel{ass:cov_form}{A4}] 
Suppose that there exist constants $U, V$ with $0 < U < V < \infty$ such that, for all
$\bu \in \Sp{d-1}$, $\bra \bu ,\sigma^2(\bu) \bu \ket = U$ and $\trace \sigma^2(\bu)  = V$. In the case $2U=V$, suppose in addition 
that $\eps(r)$
as defined in~\eqref{ass:cov_limit} satisfies $\eps(r) = O (r^{-\delta})$ for some $\delta>0$.
\end{description}
Examples of walks satisfying~\eqref{ass:moments}--\eqref{ass:cov_form} are given in~\cite{gmmw},
where it is proved that they are transient if and only if $2U<V$. 
Under~\eqref{ass:moments}--\eqref{ass:cov_form},
an invariance principle for the radial component $\|X\|$ holds~\cite{gmw}.
The full invariance principle requires additional structure on the limiting covariance
matrix $\sigma^2$ to ensure that the angular part is a suitably well-behaved process on the sphere.

\begin{description}
\item[\namedlabel{ass:sigma_smooth}{A5}] 
Suppose that $\sigma^2:\Sp{d-1}\to\R^d\otimes\R^d$ is a $\cC^\infty$-function.
\end{description}
Controlling the dependence between  the radial and angular components requires the following.
\begin{description}
\item[\namedlabel{ass:radial-evec}{A6}] Suppose that $\bu$ is an eigenvector of $\sigma^2(\bu)$ for all $\bu \in \Sp{d-1}$.
\end{description}

\section{The diffusion limit}
\label{sec:diff_proc}

\subsection{Overview}
\label{sec:overview}
Let 
$\ssym:\Sp{d-1} \to\R^d\otimes\R^d $
be the unique positive-definite matrix-valued function satisfying 
$\ssym \ssym^\tra = \sigma^2$, 
i.e. $\ssym$ is the unique \emph{symmetric square root} of $\sigma^2$.
Pick any measurable square root 
$\sigma:\Sp{d-1} \to\R^d\otimes\R^d $
of $\sigma^2$ 
and note that, since $\sigma^2$ and $\ssym$ commute,  
the matrix
$\ssym^{-1}(\bu)\sigma(\bu)$ is orthogonal 
for all $\bu\in\Sp{d-1}$.
By L\'evy's characterisation of Brownian motion, it is 
hence sufficient to prove 
Theorem~\ref{t:well-posedness} 
for the SDE
\begin{equation}
\label{eqn:x-SDE-sym}
\ud \x_t = \ssym (\hat \x_t) \ud W_t,\qquad \x_0=\bx_0\in\R^d.
\end{equation}


The next step is to establish weak existence for SDE~\eqref{eqn:x-SDE-sym}. We start with a simple lemma. 
\begin{lemma}
\label{lem:s_sym_uniformly_ell}
Under~\eqref{ass:cov_form} and~\eqref{ass:sigma_smooth},  
$\ssym$ is uniformly elliptic in the following sense: there exists a constant $\lambda > 0$ such that 
$\bra \bv, \ssym(\bu) \bv\ket \geq \lambda$ for all $\bu,\bv \in \Sp{d-1}$.  
\end{lemma}

\begin{proof}
Since $\sigma^2$ is positive-definite,
by~\eqref{ass:sigma_smooth} and the compactness of $\Sp{d-1}$ there exists
$\varepsilon>0$ such that 
$\det(\sigma^2)>\varepsilon$ on $\Sp{d-1}$. By~\eqref{ass:cov_form} we have 
$\trace \sigma^2(\bu)=V$. 
Hence the smallest eigenvalue 
$\lambda_\text{min}(\bu)$ of $\sigma^2(\bu)$
satisfies 
$\varepsilon<\lambda_\text{min}(\bu) V^{d-1}$ for all
$\bu \in \Sp{d-1}$. 
Since $\ssym$ is symmetric and non-degenerate, its eigenvalues are positive
and the smallest one is equal to $\sqrt{\lambda_\text{min}(\bu)}$. Hence the inequality
in the lemma holds for the constant
$\lambda:=(\varepsilon/ V^{d-1})^{1/2}$.
\end{proof}

Since the function 
$\bx\mapsto\ssym(\hat \bx)$
is bounded 
and uniformly elliptic by Lemma~\ref{lem:s_sym_uniformly_ell},
\cite[\S 2.6, Thm~1]{krylov} implies that weak existence holds for SDE~\eqref{eqn:x-SDE-sym}. 
Once  uniqueness in law for SDE~\eqref{eqn:x-SDE-sym} is established,  
the strong Markov property (and hence Theorem~\ref{t:well-posedness}) follows by~\cite[Thm~6.2.2]{sv}.



The proof of uniqueness in law proceeds as follows. 
Throughout Section~\ref{sec:diff_proc}, assume $U=1$ in~\eqref{ass:cov_form}.
In Section~\ref{sec:radial_process} we prove that the radial component of any
solution of~\eqref{eqn:x-SDE-sym} is  Bessel of dimension $V > 1$.  
Section~\ref{subsubsec:RiemannianGeom} introduces the Riemannian structure on the sphere, 
needed 
in Section~\ref{sec:Sp-diffusion}
to characterise the law of a stationary diffusion on $\Sp{d-1}$ indexed by
$\R$. 
This process is a key ingredient in the description of the projection of the 
path of the solution $\x$ of SDE~\eqref{eqn:x-SDE-sym} (away from $\0$)   onto $\Sp{d-1}$. In
Section~\ref{sec:transient_case} we analyse the case when $0$ is polar for the
radial process ($V \geq2$). We prove that any solution has a skew-product
decomposition constructed using the components from Sections~\ref{sec:radial_process}
and~\ref{sec:Sp-diffusion} that are unique in law.  In Section~\ref{sec:excursions} we
consider the recurrent case ($1< V < 2$). We develop the excursion theory (away from $\0$) of the solution
$\x$ of~\eqref{eqn:x-SDE-sym} without reference to the strong Markov property of $\x$.
We characterise the excursion measure in terms of the excursion measure of the radial part, given in~\cite{py},
and the law of the diffusion on  $\Sp{d-1}$ from Section~\ref{sec:Sp-diffusion}.
This implies the uniqueness in law for SDE~\eqref{eqn:x-SDE-sym}.

\subsection{The radial process}
\label{sec:radial_process}
Let
$r := \| \x \|$
be the radial part 
of a solution $\x$ of SDE~\eqref{eqn:x-SDE-sym}.

\begin{lemma}
\label{l:radial-bessel}
Let~\eqref{ass:cov_form} hold and $\sigma^2:\Sp{d-1}\to\R^d\otimes\R^d $ be measurable. 
For any solution  $(\x,W)$ of SDE~\eqref{eqn:x-SDE-sym},
adapted to a filtration 
$(\cF_t, t \geq 0)$, the process
$y=(y_t , t \geq 0)$,
$y_t:=  \| \x_t \|^2$,
is the unique strong solution of SDE
\begin{equation}
 \label{eqn:y-SDE-global}
y_t = \|\x_0\|^2 + 2\int_0^t \sqrt{y_s} \ud Z_s + V t, \quad t \geq 0,
\end{equation}
where $(Z_t, t \geq 0)$ is an $(\cF_t)$-Brownian motion given by
$Z_t := \int_0^t \hat \x_s^\tra \ssym(\hat \x_s) \ud W_s$.
In particular, the law of $r=\sqrt{y}$ is $\Bes^V\!\big(\| \x_0 \|\big)$.
\end{lemma}

\begin{remark}
A solution $\x$ of SDE~\eqref{eqn:x-SDE-sym} is continuous and hence predictable (see~\cite[Sec.~IV.5]{ry}). Since 
$\bx\mapsto \ssym(\hat\bx)\hat\bx$ is measurable on $\R^d$ (recall that we defined $\hat\0:=\be_1$),
the integrand in the definition of $Z$ is a bounded predictable process. Hence the stochastic integral $Z$ is well defined,
even though (due to rapid spinning, see Section~\ref{sec:excursions} below) its integrand is far from continuous. Moreover,
the integrand does not in general have paths in $\cD_d$ (defined in Section~\ref{sec:invariance_conditional} below). 
\end{remark}

\begin{remark}
Assuming~\eqref{ass:radial-evec}, the Brownian motion $Z$  in Lemma~\ref{l:radial-bessel} can be expressed as
\begin{equation}\label{eqn:Z-def-rad}
Z_t = \int_0^t \hat \x_u^\tra \ud W_u.
\end{equation}
\end{remark}

\begin{proof}[Proof of Lemma~\ref{l:radial-bessel}]
For any solution $(\x,W)$ of \eqref{eqn:x-SDE-sym}, the processes $y$ and $Z$ 
defined in the lemma are $(\cF_t)$-adapted. It\^o's formula and the
assumption~\eqref{ass:cov_form} imply that 
equation~\eqref{eqn:y-SDE-global} holds.  The process $Z$ is a
Brownian motion by L\'evy's characterisation,~\eqref{ass:cov_form} and assumption $U=1$. 
Since SDE~\eqref{eqn:y-SDE-global} has weak existence and pathwise uniqueness, 
the law of $y$ is
$\Besq^V\!\big(\|\x_0\|^2\big)$. 
\end{proof}

\subsection{A Riemannian structure on $\Sp{d-1}$}
\label{subsubsec:RiemannianGeom}
This section introduces a Riemannian metric $g$ on $\Sp{d-1}$, gives an explicit description of its inverse tensor 
in local coordinates and relates it to the Laplace-Beltrami operator corresponding to $g$ (see~\cite{jost} as reference
on Riemannian geometry).

Identify the tangent space $T_\bx\Sp{d-1}$  at
$\bx\in\Sp{d-1}$ with the $(d-1)$-dimensional linear subspace 
$\{v\in\R^d:\bra v,\bx\ket=0\}$ of $\R^d$
and let  the cotangent space $T^*_\bx\Sp{d-1}$ 
be the vector space dual of 
$T_\bx\Sp{d-1}$.
Denote by
$T\Sp{d-1}$
and 
$T^*\Sp{d-1}$
the tangent and cotangent~\cite[Def.~2.1.9]{jost} bundles over 
$\Sp{d-1}$,
respectively. 
Any smooth section of the vector bundle 
$T^*\Sp{d-1}\otimes T^*\Sp{d-1}$, defined in~\cite[Def.~2.1.10]{jost},
is known as a $(0,2)$-tensor field. 
Let
\begin{equation}
\label{eq:Metric}
g_\bx(v_1,v_2):=\bra \sigma^{-2}(\bx)v_1,v_2\ket\qquad\text{for any $\bx\in\Sp{d-1}$ and $v_1,v_2\in T_\bx\Sp{d-1}$.}
\end{equation}
By~\eqref{ass:sigma_smooth},
$g$ is a symmetric positive-definite $(0,2)$-tensor field, 
i.e., a Riemmanian metric on the smooth manifold $\Sp{d-1}$. 
The metric $g$ provides a \textit{canonical} way of  identifying tangent and cotangent vectors:
the map
$\tilde g : T\Sp{d-1}\to T^*\Sp{d-1}$ given by
$\tilde g_\bx(v):T_\bx\Sp{d-1}\to\R$, where 
$\tilde g_\bx(v)(u):=g_\bx(v,u)$ for any $\bx\in\Sp{d-1}$, $v,u\in T_\bx\Sp{d-1}$,
is a bundle isomorphism~\cite[Def.~2.1.6]{jost}.
For any $f\in\cC^\infty(\Sp{d-1},\R)$, 
there exists a unique smooth section 
$df$
of the cotangent bundle
$T^*\Sp{d-1}$, representing the action of the derivative of $f$ on each tangent space~\cite[Sec.~1.2]{jost}.
A vector field on the sphere is an element in the module 
$\Gamma(T\Sp{d-1})$
(over the ring $\cC^\infty(\Sp{d-1},\R)$)
of smooth sections of 
$T\Sp{d-1}$ 
\cite[Def~2.1.3]{jost}.
Let the \textit{gradient} of $f$ 
be
$\grad f:=\tilde g^{-1} (df)$.
Hence
$\grad f$ is the unique vector field
satisfying the identity
$g(\grad f,X)=dfX$ 
for all $X\in\Gamma(T\Sp{d-1})$.
Moreover,
the operator
$\grad:\cC^\infty(\Sp{d-1},\R) \to \Gamma(T\Sp{d-1})$
is defined in a coordinate free fashion. 

There exists a unique connection (the Levi-Civita connection)~\cite[Def.~4.1.1]{jost} 
$\nabla:T\Sp{d-1}\times \Gamma(T\Sp{d-1})\to T\Sp{d-1}$
on $(\Sp{d-1},g)$,
which is metric and torsion-free~\cite[Thm~4.3.1]{jost}. 
In short, the connection $\nabla$ allows us to compare tangent vectors in near-by tangent spaces
in a way that is compatible with the geometry induced by the metric $g$, cf.~\cite[Secs~4.1 \& 4.2]{jost}.
In particular, 
a vector field $X\in\Gamma(T\Sp{d-1})$ gives rise to a linear endomorphism
$(\nabla X)_\bx:T_\bx\Sp{d-1}\to T_\bx\Sp{d-1}$
for any $\bx \in\Sp{d-1}$
\cite[Def.~4.1.1]{jost}. 
Put differently, $\nabla_v X$ is the derivative of the vector field $X$
at $\bx$ in the direction $v\in T_\bx\Sp{d-1}$.
Define the \textit{divergence} of the vector field
$X$ to be the trace of this linear endomorphism, 
$(\divg X)(\bx):=\trace (\nabla X)_\bx$. 
This yields 
a coordinate free definition of the
divergence operator
$\divg:\Gamma(T\Sp{d-1})\to \cC^\infty(\Sp{d-1},\R)$.
The \textit{Laplace-Beltrami operator} $\Delta_g:\cC^\infty(\Sp{d-1},\R)\to \cC^\infty(\Sp{d-1},\R)$ 
on the Riemannian manifold
$(\Sp{d-1},g)$
can now also be defined in a coordinate-free way as 
$\Delta_g f := \divg(\grad f)$
for any 
$f\in\cC^\infty(\Sp{d-1},\R)$.

We now introduce local coordinates on $\Sp{d-1}$
in order to identify the bundle isomorphism $\tilde g^{-1}:T^*\Sp{d-1}\to T\Sp{d-1}$. 
For each 
$q \in \{1,\ldots,d\}$,
define
$[q]:=\{1,\ldots,d\}\setminus\{q\}$
and,  
throughout this section,  
identify $\R^{d-1}$ with the 
linear subspace of $\R^d$ spanned by $\{\be_i;i\in[q]\}$.
Consider an atlas
of charts
$\bz_{q}:H_q^\pm\to B^{d-1}$
on $\Sp{d-1}$, 
where 
$\pm$ is either $+$ or $-$,  
$H_q^\pm:=\{ \bx=(x_1,\ldots,x_d)^\tra \in \Sp{d-1} : \pm x_q > 0\}$
is a hemisphere, $B^{d-1}$ is the open unit ball in $\R^{d-1}$
and 
$\bz_{q}(\bx):=\sum_{i\in[q]}x_i\be_i$.
The derivative of the smooth
inverse  
$\bz_{q}^{-1}:B^{d-1}\to H_q^\pm$
induces a linear isomorphism
$d \bz_{q}^{-1}(z):T_zB^{d-1}\to T_{\bz_q^{-1}(z)} H_q^\pm$
for each $z\in B^{d-1}$.
Using the canonical identification 
$T_zB^{d-1}\equiv \R^{d-1}$ for all $z\in B^{d-1}$,
at each $\bx\in H_q^\pm$
we obtain the basis
$\mathcal{B}_\bx:=\{E_i:=d \bz_{q}^{-1}(\bz_{q}(\bx))\be_i;i\in[q]\}$ 
of $T_\bx \Sp{d-1}$ and dual basis $\mathcal{B}^*_\bx := \{E^*_i; i\in[q]\}$ 
of $T^*_\bx \Sp{d-1}$, defined by $E^*_i(E_j) = \delta_{ij}$ for $i,j \in [q]$, where
$\delta_{ij}$ is the Kronecker delta.
We interpret the tangent vector $E_i$ as a linear map
$E_i:\cC^\infty(H_q^\pm,\R)\to\cC^\infty(H_q^\pm,\R)$ satisfying the Leibniz rule,
$E_i(f): \bx \mapsto \partial_i(f\circ \bz_q^{-1})(\bz_q(\bx))$,
where $\partial_i$ is the partial derivative in the $i$-th component~\cite[p.~247]{iw}.

\begin{lemma}
\label{lem:inverse_metric}
Assume \eqref{ass:cov_form}--\eqref{ass:radial-evec}.
For $\bx\in H_q^\pm$, the matrix 
$(g^{ij}(\bx))_{i,j\in[q]}$
corresponding to the linear isomorphism
$\tilde g_\bx^{-1}:T^*_\bx\Sp{d-1}\to T_\bx\Sp{d-1}$
in terms of the bases
$\mathcal{B}^*_\bx$
and
$\mathcal{B}_\bx$,
equals
$g^{ij}(\bx) = \sigma^2_{ij}(\bx) - x_ix_j$ for any $i,j \in [q]$.
The inverse matrix
$(g_{ij}(\bx))_{i,j\in[q]}$, 
corresponding to the isomorphism 
$\tilde g_\bx:T_\bx\Sp{d-1}\to T^*_\bx\Sp{d-1}$, is given by
$g_{ij}(\bx) = \sigma^{-2}_{ij}(\bx) +
\sigma^{-2}_{qq}(\bx)x_ix_j/\bra \bx,\be_q\ket^2-(\sigma^{-2}_{qi}(\bx)x_j+\sigma^{-2}_{qj}(\bx)x_i)/\bra
\bx,\be_q\ket$,
for any $i,j \in [q]$.
Moreover, 
in the coordinates on $H_q^\pm$, 
$\Delta_g$ 
equals
$$\Delta_g f  = \sum_{i,j\in [q]} g^{ij} \big(E_i (E_j(f))
-\sum_{k\in[q]}\Gamma_{ij}^k E_k( f)\big),\qquad 
\text{for any $f\in\cC^\infty(H_q^\pm,\R)$,}
$$
where $\Gamma_{ij}^k:=\frac{1}{2}\sum_{\ell\in[q]}g^{k\ell}(E_i(g_{j\ell})+ E_j( g_{i\ell}) -E_\ell( g_{ij}))$ for $i,j,k\in[q]$.
\end{lemma}

\begin{proof}
Recall that
$B^{d-1}\subset\R^{d-1}\equiv\mathrm{Lin}\{\be_i;i\in[q]\}\subset\R^d$.
For any point $z\in B^{d-1}$ and tangent vector $u\in\R^{d-1}$
we have $d \bz_q^{-1}(z)u=u-\be_q \bra z, u\ket/\bra \bz_q^{-1}(z),\be_q\ket$.  Since
 $g_{ij}(\bx)= g_\bx(d \bz_{q}^{-1}(\bz_{q}(\bx))\be_i, d \bz_{q}^{-1}(\bz_{q}(\bx))\be_j)$
for any 
$i,j\in[q]$, 
the formula for $g_{ij}(\bx)$ follows by~\eqref{eq:Metric}. 

We now prove that 
$(g^{ij}(\bx))_{i,j\in[q]}$, defined in the lemma, is the inverse of
$(g_{ij}(\bx))_{i,j\in[q]}$.
Define $(d-1)$-dimensional square matrices $S^-$ and $S$ as follows:
$S^-_{ij}:=\sigma^{-2}_{ij}(\bx)$ 
and 
$S_{ij}:=\sigma^2_{ij}(\bx)$
for any $i,j\in[q]$.
Define $(d-1)$-dimensional vectors $S^-_q, S_q$
by
$S^-_{q,i}:=\sigma^{-2}_{qi}(\bx)$ 
and
$S_{q,i}:=\sigma^2_{qi}(\bx)$ 
for $i\in[q]$.
Let
$s:=\sigma^2_{qq}(\bx)$
and 
$s^-:=\sigma^{-2}_{qq}(\bx)$.
Since 
$\sigma^{-2}(\bx)\sigma^2(\bx)$
is the identity on $\R^d$,
we have
\begin{equation}
\label{eq:inverse_equalities}
S^-S + S^-_q S^\tra_q = I,\qquad S^- S_q = -sS_q^-,\qquad
S S^-_q = - s^-S_q,
\end{equation}
where $I$
denotes the identity matrix on $\R^{d-1}$. 
Denote
$z:= \bz_{q}(\bx)$, and $D:=\pm \sqrt{1-\|z\|^2}$.
Since 
$\bx= z  + D \be_q\in\Sp{d-1}$, 
the assumption in~\eqref{ass:radial-evec} 
implies 
$\sigma^{-2}(\bx) (z+ D \be_q) = z+ D \be_q$
(recall $U=1$).
Hence the following identities hold,
\begin{equation}
\label{eq:A6_consequances}
S^{-}z_q= z_q - D S^{-}_q,\qquad z_q^\tra S^{-}_q = (1-s^{-}) D,\qquad Sz_q=z_q- D S_q,
\end{equation}
where $z_q$ denotes the $(d-1)$-tuple of coordinates of $z$ 
expressed in the basis 
$\{\be_i;i\in[q]\}$ of 
$\R^{d-1}$.
Define $(d-1)$-dimensional square matrices 
$G,G^-$ as follows:
$$
G^-:=S-z_qz_q^\tra,\qquad G:=S^-+s^- z_qz_q^\tra/D^2 -(z_qS^{-\tra}_q+S^-_qz_q^\tra)/D.
$$
A direct calculation, using identities in~\eqref{eq:inverse_equalities}--\eqref{eq:A6_consequances} and the fact
that $S=S^\tra$ and $S^-=S^{-\tra}$, yields 
$GG^-=I$. 
It remains to note that $G^-_{ij}=g^{ij}(\bx)$ and $G_{ij}=g_{ij}(\bx)$ for all $i,j\in[q]$.

The Laplace-Beltrami operator $\Delta_g$ on any Riemannian manifold
can be expresses in local coordinates  in terms of the Christoffel symbols $\Gamma^k_{ij}$ 
as above, cf.~\cite[Ch.~V, Eqs. (4.19) and (4.32)]{iw}.
This formula is key in the proof 
of Lemma~\ref{lem:Rd-not-0-SDE}\eqref{Rd-not-0-SDE:V0}
below and hence of
Theorem~\ref{t:well-posedness}. 
We could not find a reference for it so we prove it in
Appendix~\ref{sec:manifolds} below
(see Lemma~\ref{lem:BL-G-in-lc}). 
\end{proof}

\subsection{A stationary diffusion on $\Sp{d-1}$}
\label{sec:Sp-diffusion}
Define 
$A:\R^d\setminus\{\0\}\to\R^d\otimes\R^d$ by $A(\by) := \ssym(\hat \by)$, $\by\in\R^d\setminus\{\0\}$,
and note that it is an extension of  
$\ssym:\Sp{d-1}\to\R^d\otimes\R^d$. 
For any $j \in\{1,\dots,d\}$, define $A_j: \R^d\setminus\{\0\} \to \R^d$ by
$A_j(\by)=A(\by)\be_j$ 
and note that 
its derivative $D A_j(\by)$ 
at
$\by\in\R^d\setminus\{\0\}$  
(i.e.
a linear endomorphism of $\R^d$ satisfying
$(A_j(\by+\mathbf{h}) -A_j(\by) - D A_j(\by)\mathbf{h})/\|\mathbf{h}\|\to \0$ as
$\|\mathbf{h}\|\to 0$)  
exists since, by Lemma~\ref{lem:s_sym_uniformly_ell},
$\ssym$ can be expressed as an absolutely convergent power series in $\sigma^2$, which is
smooth by~\eqref{ass:sigma_smooth}.  
Let $A_0: \R^d \setminus \{\0\} \to \R^d$ be given by $A_0(\by) := \frac{1}{2}\sum_{j=1}^d D A_j(\by)A_j(\by)$
for any 
$\by\in\R^d\setminus\{\0\}$.  

Let $S_0,S_j:\Sp{d-1}\to \R^d$ be $S_0(\bx) := -(I - \bx\bx^\tra)A_0(\bx)$ and
$S_j(\bx) := (\ssym(\bx) - \bx \bx^\tra)\be_j$ for any $\bx\in\Sp{d-1}$
and $j \in\{1,\dots, d\}$. 
Let
$\cC(\RP,\Sp{d-1})$
be equipped with 
the Borel $\sigma$-algebra 
generated by
the compact-open topology~\cite[Sec.~XII.1]{dugundji}, 
which coincides with the $\sigma$-algebra generated by the projections at any time $t\in\R$, cf.~\cite[p.~57]{bill}.

\begin{lemma}
\label{lem:Rd-not-0-SDE}
Assume \eqref{ass:cov_form}--\eqref{ass:radial-evec}. Then the following statements hold.
\begin{enumerate}[(a)]
\item 
\label{Rd-not-0-SDE:(a)}
$S_0(\bx),\ldots,S_d(\bx)\in T_\bx\Sp{d-1}$ 
for all $\bx\in\Sp{d-1}$
and  the vector fields $S_0,\ldots,S_d$ are
in $\Gamma(T\Sp{d-1})$. 
\item
\label{Rd-not-0-SDE:(b)}
Let $W$ be a standard Brownian motion on $\R^d$.
The Stratonovich SDE on $\Sp{d-1}$, given by 
\begin{equation}
\label{eq:Strat_SDE}
\ud X_t = S_0(X_t)\ud t+\sum_{j=1}^d S_j(X_t)\circ \ud W^j_t, \qquad X_0=\bx\in\Sp{d-1},
\end{equation}
has a unique strong solution in the sense of~\cite[Ch.~V, Def~1.1 \& Thm~1.1]{iw}.
\item
\label{Rd-not-0-SDE:(c)}
Let $\Pr_\bx$ denote the law of the solution of~\eqref{eq:Strat_SDE} on $\cC(\RP,\Sp{d-1})$. 
Then 
$\{\Pr_\bx,\bx\in\Sp{d-1}\}$
is a strongly Markovian system~\cite[p.~204]{iw}, 
determined uniquely by its generator 
$\cG$, 
\begin{equation*}
\cG f := S_0(f) + \frac{1}{2}\sum_{i=1}^d S_i(S_i(f))\qquad \text{for any $f\in\cC^\infty(\Sp{d-1},\R)$,}
\end{equation*}
where the vector fields $S_i$, $i\in\{0,\ldots,d\}$, are viewed as linear (over $\R$) maps 
$\cC^\infty(\Sp{d-1},\R)\to \cC^\infty(\Sp{d-1},\R)$ 
satisfying the Leibniz rule. 
\item\label{Rd-not-0-SDE:V0}
$V_0:=\cG  - \frac{1}{2} \Delta_g $ is a vector field in $\Gamma(T\Sp{d-1})$, making the solution of~\eqref{eq:Strat_SDE} 
a Brownian motion with drift on the Riemannian manifold $(\Sp{d-1},g)$
with generator
$ \frac{1}{2} \Delta_g + V_0$.
\item
\label{Rd-not-0-SDE:Ito}
Any solution $(X,W)$ 
of the It\^o SDE 
\begin{align}
\label{eqn:sphere-SDE-two-drivers}
\ud X_t & = ( \ssym(\hat X_t) - \hat X_t \hat X_t^\tra ) \ud W_t  -
             \frac{V-1}{2} \frac{ \hat X_t }{\|X_t\|} \ud t,  \quad X_0 = \bx \in\Sp{d-1} 
\end{align}
satisfies $\|X_t\|=1$ for all $t\in\RP$ and is a solution of SDE~\eqref{eq:Strat_SDE}. 
\end{enumerate}
\end{lemma}

\begin{proof}
The vector fields $S_j$,
$j\in\{0,\dots,d\}$, 
are tangential to $\Sp{d-1}$ by~\eqref{ass:radial-evec}
and smooth by~\eqref{ass:sigma_smooth}.
Hence~\eqref{Rd-not-0-SDE:(a)} holds.
Moreover,
we may interpret $S_j$ as a linear map on $\cC^\infty(\Sp{d-1},\R)$ satisfying the Leibniz rule~\cite[p.~248]{iw}
(see e.g.~\eqref{eq:Tan_to_der} below).
Hence part~\eqref{Rd-not-0-SDE:(b)} of the lemma follows from~\cite[Ch.~V, Thm~1.1]{iw}.
The family of laws 
$\{\Pr_\bx,\bx\in\Sp{d-1}\}$
is a strongly Markovian system generated by the second order differential operator $\cG$
by~\cite[Ch.~V, Thm~1.2]{iw}, which establishes part \eqref{Rd-not-0-SDE:(c)}.

To establish part \eqref{Rd-not-0-SDE:V0}, consider a chart 
$\bz_q:H_q^\pm\to B^{d-1}$ (for some $q\in\{1,\ldots,d\}$) and the corresponding frame field $\{E_i,i\in[q]\}$,
defined in the paragraph preceding Lemma~\ref{lem:inverse_metric}.
Then we can express the vector field $S_j$ on $H_q^\pm$ as a linear mapping from 
$\cC^\infty(H_q^\pm,\R)\to\cC^\infty(H_q^\pm,\R)$, satisfying the Leibniz rule, as
follows: for any $\bx\in H_q^{\pm}$ and $j \in [q]$ we have
\begin{equation}
\label{eq:Tan_to_der}
S_j(f)(\bx) =(D\bz_q(\bx)S_j(\bx))^\tra\sum_{i\in[q]} E_i(f)(\bx) \be_i=\sum_{i\in[q]} S_j^i(\bx) E_i(f)(\bx),
\end{equation}
where the second equality holds by $D\bz_q = \bz_q$, and where $S^i_j(\bx) = \bra S_j(\bx),
\be_i \ket$.
This implies $S_j(S_j(f))=\sum_{i,k\in[q]}S_j^i S_j^k E_i(E_k(f)) + \sum_{k\in[q]} \bar V_{k,j}
E_k(f)$ for some functions $\bar V_{k,j}\in\cC^\infty(H_q^\pm,\R)$, $k,j \in[q]$, and 
all $f\in\cC^\infty(H_q^\pm,\R)$. The definition of $S_j$ above, \eqref{ass:cov_form}, \eqref{ass:radial-evec} and Lemma~\ref{lem:inverse_metric}
imply 
$
\sum_{j=1}^d S_j^i(\bx) S_j^k(\bx) = g^{ik}(\bx)
$
for all $\bx\in H_q^\pm$ and $i,k\in[q]$.  Hence, by the definition of $\cG$ in the lemma and the expression for
$\Delta_g$ in the local coordinates on $H_q^\pm$ in Lemma~\ref{lem:inverse_metric},
the equality
$V_0(f) = \sum_{i\in[q]} V_{0,i} E_i(f)$
holds
for some functions $V_{0,i}\in\cC^\infty(H_q^\pm,\R)$,  $i\in[q]$.
Since such an equality holds for every 
$q\in\{1,\ldots,d\}$
and choice of $\pm$ (i.e. for every chart in our atlas),
$V_0$ satisfies the Leibniz rule and is hence an element of $\Gamma(T\Sp{d-1})$, implying~\eqref{Rd-not-0-SDE:V0}.

Extend the vector fields $S_0,S_1,\dots,S_d$ to $\R^d \setminus \{\0\}$ by defining
$\bar S_0(\by) := -(I - \hat \by \hat \by^\tra)A_0(\by)$ and $\bar S_j(\by) := (A(\by) - \hat \by \hat \by^\tra)\be_j$, $j\in\{1,\dots,d\}$,
for any
$\by \in\R^d \setminus \{\0\}$.  
Define a function
$R: \R^d \setminus \{\0\} \to \R^d$ by $R(\by) := \frac12 \sum_{j=1}^d D\bar S_j(\by) \bar S_j(\by)$.
To prove~\eqref{Rd-not-0-SDE:Ito}, we establish the following formula
\begin{equation}\label{eqn:R-T}
R(\by) = (I - \hat \by \hat \by^\tra) A_0(\by) -\frac{V-1}{2}\frac{\hat\by}{\|\by\|}
\qquad\text{for all $\by \in \R^d \setminus \{\0\}$.}
\end{equation}

Let $G(\by):=\hat \by$ for any
$\by\in\R^d\setminus\{\0\}$
and note that
$A=A \circ G$ and $DG(\by)=(I-\hat \by\hat \by^\tra)/\|\by\|$,
implying 
$DG(\by)\by= \0$,
$DG(\by)^\tra=DG(\by)$  
and
$DA_j(\by)\by = DA_j(\hat \by) DG(\by)\by= \0$ for all  
$j\in\{1,\dots,d\}$.
Since $\bar S_j(\by) = A_j(\by) - \hat \by \bra \hat \by,\be_j\ket$,
we get
$D\bar S_j(\by) = DA_j(\by) - (\hat\by^\tra \be_j I + \hat\by \be_j^\tra)
 DG(\by)$
by the product rule,
where $I$ is the identity matrix on $\R^d$.  
Hence,  using the fact that $A(\by)\by =\by$,
we get  $D\bar S_j(\by)\bar S_j(\by) = DA_j(\by) A_j(\by)
- (\hat \by^\tra \be_j I + \hat\by\be_j^\tra)(A(\by) -
\hat\by\hat\by^\tra)\be_j/\|\by\|$.   
Summing over $j\in\{1,\dots,d\}$ yields the identity
$2R(\by) = 2A_0(\by) - \trace(A(\by) - \hat\by\hat\by^\tra)\hat\by/\|\by\|$.  
Differentiating the identity $A(\by)\by = \by$ (in $\by$) yields
$I = A(\by) + \sum_{j=1}^d\bra \by,\be_j\ket DA_j(\by)$,
and hence
$A(\by) = A^2(\by) + \sum_{j=1}^d\bra \by,\be_j\ket DA_j(\by) A(\by)$, 
for all $\by\in\R^d\setminus\{\0\}$. 
Since
$A$ is symmetric 
we have
$DA_j(\by)^\tra \be_i=DA_i(\by)^\tra \be_j$
for all
$i,j\in\{1,\ldots,d\}$.
Hence we have
$2 \bra A_0(\by),\be_j \ket = \sum_{i=1}^d \bra A_i(\by), DA_i(\by)^\tra \be_j\ket  = \trace( DA_j(\by)A(\by))$.
Together with~\eqref{ass:cov_form}, this implies
$\trace A(\by) = V + 2\bra A_0(\by), \by \ket$ and \eqref{eqn:R-T} follows.

Let $(X,W)$ be a solution of~\eqref{eqn:sphere-SDE-two-drivers}.
A simple application of It\^o's formula yields $\ud \|X_t\|^2 =0$, 
implying the first statement in~\eqref{Rd-not-0-SDE:Ito}.
By~\eqref{eqn:R-T} it follows that $X$ in fact satisfies the SDE
$\ud X_t = (\bar S_0(X_t)+R(X_t))\ud t+\sum_{j=1}^d \bar S_j(X_t) \ud W^j_t$,
where $\bar S_j$, $j\in\{1,\ldots,d\}$, are defined above~\eqref{eqn:R-T}.
By the definition of the Stratonovich integral on $\R^d$~\cite[Ch.~III, Sec.~1, Eq.~(1.10)]{iw},
it follows that 
$\ud X_t = \bar S_0(X_t)\ud t+\sum_{j=1}^d \bar S_j(X_t) \circ \ud W^j_t$.
Since $S_j=\bar S_j$, $j\in\{0,\ldots,d\}$, on $\Sp{d-1}$ 
and $X$ stays on the sphere for all time,
SDE~\eqref{eq:Strat_SDE} holds for $X$ (see~\cite[Ch.~V, Rem.~1.1]{iw}).
\end{proof}

By Lemma~\ref{lem:Rd-not-0-SDE}\eqref{Rd-not-0-SDE:(c)}, 
the map $\bx\mapsto \Pr_\bx[A]$ on $\Sp{d-1}$ is Borel measurable for any 
Borel measurable set $A$ in $\cC(\RP,\Sp{d-1})$.
We can hence define 
a transition function on $\Sp{d-1}$, 
$P_t(\bx,\cdot) := \Pr_\bx[\phi_t \in \cdot]$,
where $(t,\bx)\in\RP\times\Sp{d-1}$ and
$(\phi_u,u \in \RP)$
is the coordinate process on 
$\cC(\RP,\Sp{d-1})$.
In particular, 
the law $\Pr$ of
the solution of~\eqref{eq:Strat_SDE}, started according to a probability measure 
$\nu$ on $\Sp{d-1}$, equals
$\Pr[\cdot]=\int_{\Sp{d-1}}\nu(\ud \bx) \Pr_\bx[\cdot]$.

\begin{proposition}
\label{lem:sphere-in-Rd-SDE}
Let~\eqref{ass:cov_form}--\eqref{ass:radial-evec} hold.  
There exists a unique probability measure $\mu$ on $\Sp{d-1}$ with full support,
such that
$\mu(\cdot) =\int_{\Sp{d-1}}\mu(\ud \bx) P_t(\bx,\cdot)$ for all $t\in\RP$
and the transition function 
$P_t(\bx,\cdot)$ converges to its stationary measure $\mu$ in the following sense:\footnote{Recall that 
$\|\nu_1(\cdot)-\nu_2(\cdot)\|_{\mathrm{TV}}:=\sup_{\fA\subset\Sp{d-1}}|\mu_1(\fA)-\nu_2(\fA)|$
for 
probability measures 
$\nu_1$ and $\nu_2$ on $\Sp{d-1}$.}
\begin{equation}\label{eqn:unif-ergod}
\lim_{t \to \infty} \sup_{\bx \in \Sp{d-1}} 
\| P_t(\bx,\cdot) - \mu(\cdot) \|_{\mathrm{TV}} = 0.
\end{equation}
Furthermore, 
there exists a unique law $\Pr_\Psi[\cdot]$ on the Borel sets
of $\cC(\R,\Sp{d-1})$ with compact-open topology, 
satisfying 
$\Pr_\Psi[\psi_s\in\cdot]=\mu(\cdot)$
and
$\Pr_\Psi[\psi_{s+t} \in \cdot \mid \psi_s ]=P_t(\psi_s,\cdot)$ for all $(s,t) \in \R\times\RP$, 
where $(\psi_u,u \in \R)$ denotes the coordinate process on $\cC(\R,\Sp{d-1})$.  
\end{proposition}

\begin{remarks}\label{rem:stat-meas}
\noindent (a) The unique stationary measure $\mu$ exists and has full support
essentially because the vector fields $S_1,\ldots,S_d$
in Lemma~\ref{lem:Rd-not-0-SDE}\eqref{Rd-not-0-SDE:(a)} 
span $T_\bx\Sp{d-1}$ at every 
$\bx\in \Sp{d-1}$. 
The proof uses the representation 
in Lemma~\ref{lem:Rd-not-0-SDE}\eqref{Rd-not-0-SDE:V0}
of the process as a Brownian motion with drift 
and applies the
well-known results for the stability of elliptic diffusions on compact Riemannian manifolds~\cite{pinsky}.

\noindent (b) The  geometry introduced in 
Section~\ref{subsubsec:RiemannianGeom}
allows us to characterise the time-reversibility  of the diffusion $X$ satisfying SDE~\eqref{eq:Strat_SDE}. 
This leads to an explicit description, 
given in~\eqref{eq:split_at_max} of Section~\ref{subsec:Limit_comments} above,
of the excursions of the process 
$\x$ appearing in Theorem~\ref{t:well-posedness}.  

\noindent (c) 
Kolmogorov's extension theorem~\cite[\S~III.1,~Thm~(1.5)]{ry}
and the first statement in Prop.~\ref{lem:sphere-in-Rd-SDE} 
imply that $\Pr_\Psi[\cdot]$  exists and is unique: 
for $t_1<\dots<t_k$  in $\R$ and measurable sets $\fA_i\subset \Sp{d-1}$, $i=1,\ldots,k$,
the fdd is  
$\int_{\fA_1}\mu(\ud \bx_1)\int_{\fA_2}P_{t_2-t_1}(\bx_1,\ud \bx_2) \ldots\int_{\fA_k}P_{t_k-t_{k-1}}(\bx_{k-1},\ud \bx_k)$,
cf.~\cite[\S~XII.4]{ry}.
\end{remarks}

\begin{proof}
By Lemma~\ref{lem:Rd-not-0-SDE}\eqref{Rd-not-0-SDE:V0},
the generator of the strong Markov process satisfying SDE~\eqref{eq:Strat_SDE}
takes the form $\cG = \frac12 \Delta_g + V_0$. 
The volume element 
$\ud_g \bx$ on the Riemannian manifold $(\Sp{d-1},g)$  
is a $(d-1)$-dimensional form, 
given in local coordinates 
on $H_q^\pm$ by $\sqrt{\det{G}}\prod_{i \in [q]}\ud x_i$, where $G = (g_{ij}(\bx))_{i,j \in [q]}$
(see~\cite[p.~291]{iw} and Lemma~\ref{lem:inverse_metric} above).
Let 
$\cG^\star$ be the adjoint 
of $\cG$
with respect to the measure $\ud_g \bx$. 
Assumptions of~\cite[Ch.~4, Thm~11.1]{pinsky} are satisfied for the generator $\cG$
since its second order term is the Laplace-Beltrami operator 
and the vector field $V_0$ is smooth by~\eqref{ass:sigma_smooth}. 
Hence by~\cite[Ch.~4, Thm~11.1]{pinsky},
all harmonic functions for $\cG$ are constant and
there exists a unique positive function 
$h\in\cC^2(\Sp{d-1},\R)$ 
satisfying
$\cG^\star h=0$ and 
$\int_{\Sp{d-1}} h(\bx) \ud_g \bx = 1$.
Moreover, by~\cite[Ch.~4, Thm~11.1(ix)]{pinsky},
the assumptions of~\cite[Ch.~4, Thm~8.6]{pinsky} 
for the Riemannian manifold $(\Sp{d-1},g)$
and the operator $\cG$
are satisfied, implying that 
$\mu(\ud \bx) = h(\bx) \ud_g \bx$ is the unique stationary probability measure for the
transition function $P_t(\bx,\ud \by)$. 
Again, by~\cite[Ch.~4, Thm~11.1(ix)]{pinsky},
the assumptions of~\cite[Ch.~4, Thm~9.9]{pinsky} for 
$(\Sp{d-1},g)$
and $\cG$ are satisfied.  Hence, as  $\Sp{d-1}$ is compact, \cite[Ch.~4, Thm~9.9]{pinsky} 
implies the convergence in total variation in~\eqref{eqn:unif-ergod}.
\end{proof}

\subsection{Proof of Theorem~\ref{t:well-posedness} when 0 is polar for the radial process }
\label{sec:transient_case}
Assume throughout this section that $V \geq 2$ (and $U=1$)
and let 
$(\x,W)$ be any solution  to \eqref{eqn:x-SDE-sym}, adapted to $(\cF_t, t \geq 0)$,
on a probability space that supports a one-dimensional $(\cF_t)$-Brownian motion, independent of $(\x,W)$.
By Lemma~\ref{l:radial-bessel}, $0$ is polar for $r=\|\x\|$. 

\begin{lemma}
\label{l:transient-time-change} 
Let~\eqref{ass:cov_form} hold. 
If  either (i) $s >0$; or (ii) $\x_0 \neq \0$ and $s=0$, define
\begin{equation}
\label{eqn:time-change-started-at-s}
\rho_s(t) := \int_s^t r^{-2}_u \ud u , \quad t \geq s.
\end{equation}
Then, almost surely,  $\rho_s : [s, \infty) \to \RP$  is  continuously increasing and  $\lim_{t \uparrow \infty} \rho_s(t)=\infty$. 
Its continuous inverse
$c_s: \RP \to [s, \infty)$ is 
$c_s(t) := \inf\{ u \geq s : \rho_s(u) = t \}$. In particular, $c_s(0)=s$.
\end{lemma}

Lemma~\ref{l:transient-time-change} is a direct consequence of the next lemma. 


\begin{lemma}
\label{lem:angular_clock}
Pick $x,m\in\RP$ and $\delta\geq2$. Let $\beta = (\beta_t, t \geq 0)$ be $\Bes^\delta(x)$, 
$\tau_m:=\inf\{t\geq0:\beta_t=m\}$ (with $\inf\emptyset=\infty$) and $f_m(y):=(m-y)^{-2}$. 
If $m>x$ or $x>0=m$, then $\int_0^{\tau_m}f_m(\beta_u)\ud u=\infty$ a.s.
If $x=m=0$, then for any $t>0$ it holds that $\int_0^tf_0(\beta_u)\ud u=\infty$ a.s.
\end{lemma}


\begin{proof}
Note that $\tau_m<\infty$ a.s. for all $x,m\in\RP$ and $\delta\geq2$
and $y\mapsto |y-m|f_m(y)$ is not integrable at $m$. 
Hence Lemma~\ref{lem:angular_clock} follows from~\cite[Thm~2.2, Eq.~(2.5)]{cherny_BF} in all cases
except when $x=m=0$.
Assume $x=m=0$
and time-reverse $\beta$  killed at $\tau_a$ (for some large $a>0$) at the last time 
the process visits some $b\in(0,a)$ (this is a co-optional time, see~\cite[Ch.~VII.4]{ry} for details on time reversals). 
The time reversal is a diffusion on $(0,a)$ with the same volatility function as $\beta$ and the scale function given by $\bar s = 1/(s(a)-s):(0,a)\to\R$,
where 
$s(y)=-y^{2-\delta}$ (resp. $\log(y)$) if $\delta>2$ (resp. $\delta=2$).
Note that $\lim_{y\downarrow0}\bar s(y)=0$, $\lim_{y\uparrow a}\bar s(y)=\infty$ and $\bar sf_0/\bar s'=(s(a)-s)f_0/s'$ is not integrable at $0$.
Hence the lemma follows by~\cite[Thm.~2.11(ii)]{mu}.
\end{proof}

\begin{proposition}
\label{prop:transient-away-from-0}
Suppose that \eqref{ass:cov_form}, \eqref{ass:sigma_smooth} and \eqref{ass:radial-evec}
hold. 
Assume either (i) $s >0$; or (ii) $\x_0 \neq \0$ and $s=0$ hold. 
Let a standard one-dimensional Brownian motion $Z$ be given by \eqref{eqn:Z-def-rad}
and let $c_s$ be as in Lemma~\ref{l:transient-time-change}. 
The process
$\tt=(\tt_t, t \geq 0)$ on $\Sp{d-1}$, defined by
$\tt_t:=\hat\x_{c_s(t)}$,
is a strong solution of SDE~\eqref{eqn:sphere-SDE-two-drivers}
started at
$\tt_0=\hat \x_s$ and 
driven by   a $d$-dimensional  Brownian motion $(B_t, t \geq 0)$ adapted to the filtration $(\cF_{c_s(t)}, t \geq 0)$, 
independent of $(Z_t, t \geq 0)$. 
\end{proposition}

\begin{proof}
By assumption we have $r_s > 0$ a.s. Since $0$ is polar
for $\Besq^V(r^2_s)$,
$(r^{-2}_t;t\geq s)$ is a continuous semimartingale. 
Hence 
$\ud( r^{-1}_t ) = -r^{-2}_t \ud Z_t - (V-3)/(2r^3_t) \ud t$
by It\^o's formula and~\eqref{eqn:y-SDE-global}. 
By~\eqref{ass:radial-evec}, the covariation equals 
$\ud [\x , r^{-1}]_t = \ssym (\hat \x_t) \ud [W,-W^\tra]_t \ssym (\hat \x_t) \hat \x_t/r_t^2 = 
-\hat \x_t/r_t^2\ud t$,
and It\^o's product rule implies 
\begin{equation}
\label{eqn:theta-SDE}
\ud \hat \x_t = f( \hat \x_t ) r^{-2}_t \ud t + g(\hat \x_t) r^{-1}_t \ud W_t, \quad t \geq s,
\end{equation}
where we have used the notation 
\begin{equation}
\label{eqn:theta-SDE-coeffs}
f ( x) := -\frac{V-1}{2} \frac{\hat x}{\|x\|} ~~~\text{and}~~~ g( x ) := \ssym (\hat x) - \hat x \hat x^\tra,
\qquad\text{for any $x \in \R^d$.}
\end{equation}

Define continuous local martingales 
$A=(A_t; t\geq 0)$
and
$\zeta=(\zeta_t; t\geq 0)$
by
\begin{equation}
\label{eqn:tilde-W}
A_t := \int_{s}^{c_s(t)}  r_u^{-1} \ud W_u \qquad \text{and}\qquad
\zeta_t := \int_s^{c_s(t)} r_u^{-1} \ud Z_u,
\end{equation}
where $Z$ is given in~\eqref{eqn:Z-def-rad}.
Both $A$ and $\zeta$ are adapted to 
$(\cF_{c_s(t)}, t \geq 0)$.
By~\cite[Prop.~V.1.4--5]{ry} and Lemma~\ref{l:transient-time-change}
it holds that
$
[A,A^\tra]_t = I \int_s^{c_s(t)} \frac{\ud u}{r^2_u} =  I t$,
where
$I$ is the identity matrix on $\R^d$,
and 
$[\zeta,\zeta]_t=t$.
Hence, by L\'evy's characterisation theorem,
both 
$A$ and
$\zeta$
are 
$(\cF_{c_s(t)})$-Brownian motions. 
Furthermore, by~\eqref{eqn:Z-def-rad} and~\cite[Prop.~V.1.4--5]{ry}, we have that 
$\zeta_t = \int_{s}^{c_s(t)} \hat\x_u^\tra r_u^{-1}  \ud W_u  = \int_0^t\tt_u^\tra \ud A_u$ for all $t\geq0$.
Let $(\gamma'_t, t \geq 0)$ be a one-dimensional $(\cF_t)$-Brownian motion, independent of $(\x,W)$. 
Define 
$(\cF_{c_s(t)})$-Brownian motion 
$\gamma=(\gamma_t, t \geq 0)$ by $\gamma_t:=\int_s^{c_s(t)} r_u^{-1} \ud \gamma'_u$
and
note that
$[\zeta,\gamma]\equiv0$.
Define $B=(B_t, t \geq 0)$ by
$B_t := A_t - \int_0^t \tt_u \ud \zeta_u + \int_0^t \tt_u \ud \gamma_u$
and observe
$\ud [B,B^\tra]_t = (I - \tt_t\tt_t^\tra)^2\, \ud t + \tt_t\tt_t^\tra \ud t = I \ud t$ and 
$\ud [B,\zeta]_t = (I - \tt_t\tt_t^\tra) \ud[A,A^\tra]_t\tt_t + \tt_t \ud[\gamma,\zeta]_t = 0$.
In particular,
$B$
is a $d$-dimensional $(\cF_{c_s(t)})$-Brownian motion, independent of $\zeta$.

We now show $B$ is independent of $Z$. By the Markov property,
$B_t$ depends on $\cF_s = \cF_{c_s(0)}$ only
via $B_0 = \0$, so $B$ is independent of $\cF_s$. Hence $B$ is independent of $(Z_t, t \in [0,s])$.  
It remains to prove that $B$
is independent of $(Z_t-Z_s, t \geq s)$. 
Note that by~\eqref{eqn:tilde-W} and Lemma~\ref{l:transient-time-change}
it holds that 
$Z_{c_s(t)}-Z_s=\int_s^{c_s(t)} r_u r_u^{-1}\ud Z_u  = \int_0^{t} r_{c_s(v)} \ud \zeta_v$
for all $t\geq0$. Hence the covariation of 
$\cF_{c_s(t)}$-local martingales 
$M:=Z_{c_s(\cdot)}-Z_s$ and $B$ is identically equal to zero. 
Since the inverse of 
the quadratic variation 
$[M]_u=c_s(u)-s$
equals $\rho_s(s+u)$,
by Knight's theorem~\cite[Theorem~V.1.9]{ry}, the processes
$M_{\rho_s(s+\cdot)}$
and 
$B$
are independent Brownian motions. 
It only remains to note that
$M_{\rho_s(s+u)}=Z_{s+u}-Z_s$ for any $u\geq0$.

By definition we have 
$\tt_t = \hat \x_s + \int_s^{c_s (t)} \ud \hat \x_u$.
Hence the change of variable formulas for Stieltjes~\cite[Prop.~0.4.1]{ry} and stochastic~\cite[Prop.~V.1.4]{ry} integrals
and~\eqref{eqn:theta-SDE} 
imply 
\begin{equation}\label{eqn:tildetheta-SDE}
\tt_t =\tt_0+ \int_0^t (\ssym(\tt_u) - \tt_u\tt^\tra_u) \ud A_u - \frac{V-1}{2} \tt_u \ud u, \qquad t \geq 0. 
\end{equation}
Since
$(\ssym(\tt_t) - \tt_t\tt_t^\tra) \ud B_t 
 = (\ssym(\tt_t) - \tt_t\tt_t^\tra)\bigl( (I-\tt_t\tt_t^\tra)\ud A_t + \tt_t \ud \gamma_t \bigr)
= (\ssym(\tt_t) - \tt_t\tt_t^\tra) \ud A_t $,
the process $\tt$ satisfies 
SDE~\eqref{eqn:sphere-SDE-two-drivers} driven by $(B_t,t \geq 0)$ as required.  
\end{proof}

\begin{proof}[Proof of Theorem~\ref{t:well-posedness} in the transient case with $\x_0 \neq \0$.]
By Proposition~\ref{prop:transient-away-from-0} (enlarge the probability space if needed), 
the law of any solution $\x$ of SDE~\eqref{eqn:x-SDE-sym},
satisfying $\x_0 \neq \0$, 
is equal to that of 
$(r_t \tt_{\rho_0(t)}, t\geq0)$,  
where $r\sim\Bes^{V}(\|\x_0\|)$,
$\rho_0(\cdot)$ is given in~\eqref{eqn:time-change-started-at-s}
and $\tt$ is the unique solution of~\eqref{eqn:sphere-SDE-two-drivers}  with 
$\tt_0=\hat\x_0$, independent of $r$.
\end{proof}

In order to characterise the law of $\x$ in 
the case $V \geq 2$ with   $\x_0 = \0$, we need to understand 
the law of the $\hat \x_s$ (for any fixed $s>0$) and its dependence on the path of the radial process $r$.
Define
$\cF^r_\infty := \sigma ( r_t, t \geq 0)$. 
Since $r\sim\Bes^{V}(0)$ is non-negative and $r^2$
is a strong solution of SDE~\eqref{eqn:y-SDE-global}, 
we have $\cF^r_\infty = \sigma ( r_t^2, t \geq 0)=\sigma (Z_t ,t \geq 0)$.
Recall that by
Prop.~\ref{lem:sphere-in-Rd-SDE},  the process $\tt$ defined in Proposition~\ref{prop:transient-away-from-0}
has a unique stationary measure $\mu$. 

\begin{lemma}
\label{l:independent_angle_transient}
Suppose that \eqref{ass:cov_form}, \eqref{ass:sigma_smooth} and \eqref{ass:radial-evec}
hold. 
Then for any $t >0$, $\hat \x_t$ has the law $\mu$ and is independent
of  $\cF^r_\infty$.  Put differently, the conditional law takes the form
\[ \Pr [ \hat \x_t \in \cdot \mid \cF^r_\infty ] 
= \mu ( \cdot) , \as , \text{ for any  } t > 0. \]
\end{lemma}

\begin{proof}
Fix $t >0$ and let $s \in (0,t)$. 
By
Prop.~\ref{prop:transient-away-from-0}
and
Lemma~\ref{l:transient-time-change}
we have
$\hat \x_t = \tt_{\rho_s (t)}$, 
where $\varphi$ satisfies SDE~\eqref{eqn:sphere-SDE-two-drivers}.
By \eqref{Rd-not-0-SDE:Ito}, \eqref{Rd-not-0-SDE:(b)} and \eqref{Rd-not-0-SDE:(c)} of Lemma~\ref{lem:Rd-not-0-SDE}
and Prop.~\ref{lem:sphere-in-Rd-SDE},
$\tt$ is strong Markov 
with the transition function
$P_u(\bx, \cdot)$ 
that does not depend on $s$.
Hence, for $\fA\subseteq\Sp{d-1}$, we find
\begin{align}
\label{eqn:earlier_conditioning}
 \Pr [ \hat\x_t \in \fA \mid \cF^r_\infty ] 
& = \Exp [  \Pr [ \hat\x_t \in \fA \mid \sigma(\hat\x_s)\vee\cF^r_\infty ] \mid \cF^r_\infty ] 
 = \Exp [ P_{\rho_s(t)}(\hat\x_s,\fA)   \mid \cF^r_\infty ],
\end{align}
as $\tt_{\rho_s(t)}$ depends on $\cF^r_\infty$ only through $\rho_s(t)$ and $\tt_0 = \hat\x_s$.
Crucially, \eqref{eqn:earlier_conditioning} holds for any fixed time
$s\in(0,t)$, and also for any random time $s=S\in(0,t)$ if $S$ is $\cF^r_\infty$-measurable.

By Lemma~\ref{lem:angular_clock} we have
$\lim_{s\downarrow0}\rho_s (t)=\infty$.
Hence, for sufficiently small $s$, 
an arbitrarily large time interval separates $\tt_0 = \hat\x_s$ and $\tt_{\rho_s(t)}$,
and so stationarity must be attained at the latter, regardless of $\hat\x_s$.
Formally, we apply the uniform ergodicity of $\tt$ in~\eqref{eqn:unif-ergod}.  
Lemmas~\ref{l:transient-time-change} and~\ref{lem:angular_clock}
imply that for any $u >0$,  there is an $\cF^r_\infty$-measurable random variable $S = S(t,u)$ with $S \in (0,t)$ a.s.~such that 
$\rho_S (t) \geq u$.
By \eqref{eqn:unif-ergod}, for any $\eps>0$ there exists $u>0$ 
such that 
$ |P_{\rho_S(t)}(\tt_0,\fA)  - \mu ( \fA) | \leq \eps , \as$ 
Hence, by \eqref{eqn:earlier_conditioning} applied at the random time $S$,
we have 
$ | \Pr [ \hat\x_t \in \fA \mid \cF^r_\infty ]  - \mu ( \fA ) | \leq \eps, \as$ 
Since $\eps>0$ was arbitrary, the result follows.
\end{proof}

\begin{proof}[Proof of Theorem~\ref{t:well-posedness} in the transient case with $\x_0 = \0$.]
For any $k\in\N$
and
open set $U\subset \R^k$,
define a measurable function $F_U:(0,\infty)^k\to[0,1]$,
$F_U(t_1,\ldots,t_k):= \Pr_\Psi[(\Psi_{t_1},\ldots,\Psi_{t_k})\in U]$,
where the law
$\Pr_\Psi[\cdot]$ is defined in 
Prop.~\ref{lem:sphere-in-Rd-SDE}.
By 
Lemma~\ref{l:transient-time-change},
Proposition~\ref{prop:transient-away-from-0}
and
Lemma~\ref{l:independent_angle_transient}
we have 
$\Pr[(\hat\x_{t_1},\ldots,\hat\x_{t_k})\in U |\cF^r_\infty ]=F_U(\rho_s(t_1),\ldots,\rho_s(t_k))$ a.s. 
for $0<s<t_1<\cdots<t_k$. 
Hence
$\Pr[(\hat\x_{t_1},\ldots,\hat\x_{t_k})\in U]=\Exp F_U(\rho_s(t_1),\ldots,\rho_s(t_k))$. 
Therefore the finite-dimensional distributions of  $(\hat \x_t, t> 0)$ 
are uniquely determined by $\Pr_\Psi[\cdot]$
and the law of $r$.
Moreover,
by Lemma~\ref{l:radial-bessel}, the law of $(\|\x\|,\hat{\x})$, and hence 
of $\x$, is uniquely determined by~$\Bes^V(0)$ and $\Pr_\Psi[\cdot]$.
The uniqueness in law of~\eqref{eqn:x-SDE-sym} implies that $\x$ is strong Markov and 
Thm~\ref{t:well-posedness} follows in the transient case. 
\end{proof}

\subsection{Proof of Theorem~\ref{t:well-posedness} in the recurrent case: rapid spinning of $\hat \x$}
\label{sec:excursions}
In this section we assume $V\in(1,2)$ and $U=1$. 
Hence, by Lemma~\ref{l:radial-bessel},  $r=\|\x\|$ is $\Bes^V(0)$ where $\x$ is a solution of SDE~\eqref{eqn:x-SDE-sym}.
We recall briefly the necessary elements of excursion theory 
(see~\cite[Ch.~XII]{py},~\cite[Ch.~IV]{bert} as a general reference).
Since $0$ is regular and instantaneous for $r$, there exists Markov local time 
$L=(L_t, t \geq 0)$  
at $0$.
By~\cite[Prop.~XI.1.1]{ry}, up to a constant factor, $L$ 
can be expressed as a time-change of the Brownian local time at 0,
where the time-change is a constant multiple of $(\int_0^tr^{-2(V-1)}_u\ud u; t\geq0)$.
Hence, by~\cite[Thm~2.4]{cherny_BF}, $\lim_{t\uparrow \infty}L_t=\infty$ $\Pr$-a.s.
Let $L_\la^{-1} := \inf\{t\geq0 : L_t > \la\}$ (for $\la\geq0$)
be the  right-continuous inverse of $L$ and 
$L_{\la^-}^{-1}:=\lim_{\kappa \uparrow \la} L_\kappa^{-1}$ 
(for $\la > 0$),
$L_{0^-}^{-1} := 0$.  
The process $(L^{-1}_\la, \la \geq 0)$ is a subordinator (i.e. a L\'evy process with non-decreasing 
paths). 
Furthermore, as $L$ tends to infinity, $L^{-1}$ is not killed: $\Pr[L_\la^{-1}\in\RP \forall \la\in\RP]=1$. 
Define the (countable) set of jump times by
$\La^r := \{ \la \geq 0 : L_{\la^-}^{-1} < L_\la^{-1}\}$,
set $\tau_\la^r := L_\la^{-1} - L_{\la^-}^{-1}$ and note that both
$L_\la^{-1}$ and $L_{\la^-}^{-1}$ are stopping times for any $\la\in\RP$.  
For any $w \in \cC_d=\cC(\R_+,\R^d)$, 
let $\tau_{\0}(w) := \inf\{t>0: w(t) = \0\}$ ($\inf \emptyset =\infty$)
and define $\cE_d:=\{w\in\cC_d:0 < \tau_{\0}(w) < \infty \text{ and $w(t) = \0$ for all $t \notin (0,\tau_{\0}(w))$}\}$ 
with the topology induced by 
the compact-open topology~\cite[Sec.~XII.1]{dugundji} on $\cC_d$. 
Let $\delta_d$ be the zero function in $\cC_d$. 
Since $0$ is recurrent for the strong Markov process $r$,
by~\cite[Ch.~IV, Thm.~10(i)]{bert}, 
the point process 
$e^r = (e^r_\la, \la \geq 0)$ with values in $\cE_1 \cup \{ \delta_1 \}$, 
defined by 
$e^r_\la(t):= r_{L_{\la^-}^{-1}  + t}  \1{ t \leq \tau^r_\la }$
(resp.  $e^r_\la = \delta_1$)
if $\la \in \La^r$ (resp. $\la \notin \La^r$), 
is a Poisson
point process (PPP)
with 
excursion measure $\mu_r$ on $\cE_1$.

\subsubsection{Marked Bessel excursions}
\label{subsub:Marked_Bessel_Ex}
Pick $a\in(0,\infty)$ 
and let 
$t\wedge a:=\min(t, a), t\vee a:=\max(t,a)$
for any $t\in\R$.
For any $w \in \cE_1$ 
satisfying $\tau_0(w)>a$,  
define 
$\varrho_w^a:(0, \tau_0(w))\to\R$ by the formula
\begin{equation}\label{eqn:bessel-excursion-time-change}
\varrho_w^a(t) :=\sign(t-a) \int^{t\vee a}_{t\wedge a} w(u)^{-2} \ud u, \qquad t \in (0, \tau_0(w)).
\end{equation}
Let
$\cE_1^{(a)}:=\{w\in\cE_1:w\geq0, \tau_0(w)>a \text{ and } \lim_{t \uparrow \tau_0(w)} \varrho_w^a(t) = - \lim_{t \downarrow 0} \varrho_w^a(t) = \infty\}$
and, for $d\in\N\setminus\{1\}$, define the set 
$\cE_d^{(a)}:=\{w\in\cE_d:\|w\|\in\cE_1^{(a)}\}$
and the map 
$\Phi_a:\cE_1^{(a)}\times \cC(\R,\Sp{d-1})\to \cE_d^{(a)}$, 
\[
\Phi_a(w,\theta)(t):= \begin{cases} w(t) \cdot \theta\circ \varrho_w^a (t) & t \in (0,\tau_0(w)), \\
\0 & t \in \RP\setminus(0,\tau_0(w)).
\end{cases}
\]
The topology on 
$\cE_d^{(a)}$
is induced by the compact-open topology on $\cC_d$~\cite[Sec.~XII.1]{dugundji}.
Hence the Borel $\sigma$-algebra on $\cE_d^{(a)}$ is generated by 
$\pi_t:\cE_d^{(a)}\to\R^d$,
$\pi_t(w):= w(t)$, for any $t\in\RP$~\cite[p.~57]{bill}.

\begin{lemma}\label{l:rho-2way-inf}
The following statements hold for any fixed $a\in(0,\infty)$. 
\begin{enumerate}[(i)]
\item 
\label{l:rho-2way-inf(i)}
For $w \in \cE_1^{(a)}$, 
$\varrho_w^a: (0,\tau_0(w)) \to \R$ is continuous, increasing and 
$c_w^a: \R\to (0,\tau_0(w))$, given by
$c_w^a(u):=\inf\{t\in(0,\tau_0(w)):\varrho_w^a(t)\geq u\}$,
is continuous, increasing and $c_w^a(0)=a$. 

\item 
\label{l:rho-2way-inf(iv)}
Pick $b\in(0,a)$, $w\in\cE_1^{(a)}$ and let $I_b^a(w):=\varrho_w^b(t)-\varrho_w^a(t)$,
$t\in(0,\tau_0(w))$.
Then $I_b^a(w)>0$ does not depend on $t$, satisfies
$c_w^a(u)=c_w^b(u+I_b^a(w))$ for all $u\in\R$ and $\lim_{b\to0}I_b^a(w)=\infty$.

\item 
\label{l:rho-2way-inf(iii)}
$\Phi_a:\cE_1^{(a)}\times \cC(\R,\Sp{d-1})\to \cE_d^{(a)}$ 
is a Borel isomorphism, i.e. $\Phi_a$ is a bijection with inverse 
given by 
$\Phi_a^{-1}(w)=(\|w\|, w \circ c_{\|w\|}^a/\|w \circ c_{\|w\|}^a\|)$, $w\in\cE_d^{(a)}$, 
and 
both $\Phi_a$ and $\Phi_a^{-1}$ are Borel measurable. 
Moreover, for any $s\in\R$, the map
$\cE_d^{(a)}\to\RP$, $w\mapsto c_{\|w\|}^a(s)$, is continuous.

\item 
\label{l:rho-2way-inf(v)}
Define the set 
$\Upsilon^{(a)}_d:=\{(b,w)\in (a,\infty)\times \cE^{(a)}_d: w\in\cE^{(b)}_d\}$ for any $d\in\N$.
Then the map
$Q_a:
\Upsilon^{(a)}_1\times \cC(\R,\Sp{d-1})
\to
\cE^{(a)}_1\times \cC(\R,\Sp{d-1})$,
$Q_a(b,w,\theta):=(w,\theta(\cdot+I_a^b(w))$,
is continuous and the equality 
$\Phi_b^{-1}(w)=Q_a(b,\Phi_a^{-1}(w))$
holds for any
$(b,w)\in\Upsilon^{(a)}_d$.

\item
\label{l:rho-2way-inf(vi)}
The map
$\{(b,b',w)\in (0,\infty)^2\times \cE_1: w\in\cE^{(b\vee b')}_1\}\to \R$,
$(b,b',w)\mapsto \varrho_w^{b'}(b)$, is continuous. 
\end{enumerate}
\end{lemma}

\begin{remark}
\noindent (a) The maps $\Phi_a$ and $\Phi_a^{-1}$ in Lemma~\ref{l:rho-2way-inf}\eqref{l:rho-2way-inf(iii)}
are homeomorphisms. The proof of this fact is more complicated than that of 
Lemma~\ref{l:rho-2way-inf}\eqref{l:rho-2way-inf(iii)} and is omitted  as 
it is not used.  \\
\noindent (b) The topology on 
$\Upsilon^{(a)}_d$ is induced by 
$(a,\infty)\times \cE^{(a)}_d$.
Parts~\eqref{l:rho-2way-inf(iii)} and~\eqref{l:rho-2way-inf(v)} of Lemma~\ref{l:rho-2way-inf} imply that the map $(b,w)\mapsto\Phi_b^{-1}(w)$,
defined 
on
$\Upsilon^{(a)}_d$,
is measurable. The map in~\eqref{l:rho-2way-inf(vi)} is measurable.
\end{remark}

\begin{proof}
Since  $w(u)>0$ for all 
$u \in (0,\tau_0(w))$, \eqref{l:rho-2way-inf(i)} holds.  
Note that $\cE_1^{(a)}\subset \cE_1^{(b)}$
and $I_b^a(w)=\int_b^a 1/w(u)^2\ud u$. Part~\eqref{l:rho-2way-inf(iv)} follows by the representation of $c_w^a$ from~\eqref{l:rho-2way-inf(i)}
and the definition of $\cE_1^{(a)}$.

For part~\eqref{l:rho-2way-inf(iii)},
note that 
$\tau_0(w)=\tau_\0(\Phi_a(w,\theta))$
for all
$w\in\cE_1^{(a)}$ and $\theta\in\cC(\R,\Sp{d-1})$.
Since 
$\theta$ is bounded and $w$ is continuous and equals $0$ on $\RP\setminus(0,\tau_0(w))$,
both $\Phi_a$ 
and its inverse are well-defined.
Since the $\sigma$-algebra on $\cE_d^{(a)}$ is generated by the projections,
the map $\Phi_a$ is Borel measurable 
if and only if $\pi_t\circ \Phi_a$ is a measurable map into $\R^d$
for every $t\in\RP$.
Since, 
for any measurable set $A$ in $\R^d$, 
$(\pi_0\circ \Phi_a)^{-1}(A)$ is either empty or the whole space 
we may assume
$t>0$.
Then, 
$(\pi_t\circ \Phi_a)^{-1}(\{\0\})=(\cE_1^{(a)}\setminus \{w\in \cE_1^{(a)}:w(t)>0\})\times \cC(\R,\Sp{d-1})$
is clearly measurable. 
It is therefore sufficient to prove that 
$(\pi_t\circ \Phi_a)^{-1}(B)$ is open for any ball $B$ centred at $b\in\R^d$ of radius $\eps'\in(0,\|b\|)$.
Pick $(w,\theta)\in (\pi_t\circ \Phi_a)^{-1}(B)$ and set
$\eps:=(\eps'-\|\Phi_a(w,\theta)(t)-b\|)/2>0$.
Then 
$I_w:=\inf_{s\in[t\wedge a, t\vee a]}w(s)>0$.
In particular,
$[t\wedge a,t\vee a]\subset(0,\tau_0(w))$.
Define 
$S_w:=\sup_{s\in[t\wedge a, t\vee a]}w(s)$.
There exists $\delta_0\in(0,1)$ such that
if $|\varrho_w^a(t)-s|<\delta_0$ then 
$\|\theta(\varrho_w^a(t))-\theta(s)\|<\eps/(3S_w+3)$.
Assume now that 
$t\neq a$
and pick
$\delta\in (0,1)$
smaller than
$\min\{\eps/3, I_w/2,\delta_0 I_w^4(4(2S_w+1)|a-t|)^{-1}\}$.
Define the compact $K_1:=[t\wedge a,t\vee a]\subset \RP$ 
(resp. $K_2:=[\varrho_w^a(t)-1,\varrho_w^a(t)+1]\subset\R$),
$\eps_1:=\delta$ (resp. 
$\eps_2:=\eps/(3S_w+3)$)
and the neighbourhood 
$N_{\eps_1}(K_1):=\{u\in\cE_1^{(a)}:\sup_{s\in K_1}|w(s)-u(s)|<\eps_1\}$ 
(resp. $N_{\eps_2}(K_2):=\{\phi\in\cC(\R,\Sp{d-1}):\sup_{s\in K_2}\|\theta(s)-\phi(s)\|<\eps_2\}$)
of $w$ (resp. $\theta$) in 
$\cE_1^{(a)}$ (resp. $\cC(\R,\Sp{d-1})$).
Pick
$(u,\phi)\in N_{\eps_1}(K_1)\times N_{\eps_2}(K_2)$
and note that 
$u(s)> I_w - \delta>I_w/2$ for all $s\in K_1$.
Hence, by~\eqref{eqn:bessel-excursion-time-change}, 
we have
$|\varrho_w^a(t)-\varrho_u^a(t)|\leq 4(2S_w+1)|a-t|I_w^{-4}\delta <\delta_0<1$,
implying 
$u(t)\|\theta(\varrho_w^a(t))-\theta(\varrho_u^a(t))\|<\eps/3$ 
and 
$\varrho_u^a(t)\in K_2$.
Hence
$u(t)\|\theta(\varrho_u^a(t))-\phi(\varrho_u^a(t))\|<\eps/3$
and the following inequalities hold
$$
\|\Phi_a(w,\theta)(t)-\Phi_a(u,\phi)(t)\|\leq
|w(t)-u(t)| + 
u(t)(\|\theta(\varrho_w^a(t))-\theta(\varrho_u^a(t))\|+
\|\theta(\varrho_u^a(t))-\phi(\varrho_u^a(t))\|)<\eps.
$$
Thus 
$
\|\Phi_a(u,\phi)(t)-b\|\leq\eps+
\|\Phi_a(w,\theta)(t)-b\|< \eps'$,
implying
$N_{\eps_1}(K_1)\times N_{\eps_2}(K_2)\subset (\pi_t\circ \Phi_a)^{-1}(B)$
and hence that $\pi_t\circ\Phi_a$ is measurable for $t\neq a$. 
If
$t=a$,
we have
$\varrho_u^a(t)=0$
for all 
$u\in\cE_1^{(a)}$.
Hence
$(u,\phi)\in\cE_1^{(a)}\times \cC(\R,\Sp{d-1})$,
such that 
$|w(t)-u(t)|<(w(t)\wedge\eps)/2$
and
$\|\theta(0)-\phi(0)\|<2\eps/w(t)$,
satisfies 
$\Phi_a(u,\phi)(t)\in B$ (where $(w,\theta), B,\eps$ are as above)  and the 
measurability of 
$\pi_t\circ\Phi_a$ follows. 

Due to the product structure of the image, the map $\Phi_a^{-1}$
is measurable if 
$\cE_d^{(a)}\to\cC(\R,\R^d\setminus\{\0\})$,
$w\mapsto  w \circ c_{\|w\|}^a$,
is measurable, 
which is equivalent to 
$g_s:\cE_d^{(a)}\to\R^d\setminus\{\0\}$,
$g_s(w):=w(c_{\|w\|}^a(s))$,
being measurable for every $s\in\R$.
The map $g_s$ is in fact continuous. If $s=0$, then $g_s(w)=w(a)$ is an 
evaluation at $a$,  
which is continuous in the compact-open topology. 
If $s\neq0$, let $B$ denote an open ball centred at 
$b\in\R^d\setminus\{\0\}$
of radius $\eps'\in(0,\|b\|)$,
pick
$w\in g_s^{-1}(B)$
and let 
$\eps:=(\eps'-\|g_s(w)-b\|)/2$.
Define
$t:=c_{\|w\|}^a(s)\neq a$ 
and let
$S_{\|w\|}:=\sup_{p\in[t\wedge a, t\vee a]}\|w(p)\|$, $I_{\|w\|}:=\inf_{p\in[t\wedge a, t\vee a]}\|w(p)\|$,
$K_1:=[0,\tau_\0(w)]$
and 
$\bar S_{\|w\|}:=\sup_{p\in K_1}\|w(p)\|$.
There exists
$\delta_0\in(0,1)$
such that 
$[t-\delta_0,t+\delta_0]\subset (0,\tau_\0(w))$
and
$\forall x\in[t-\delta_0,t+\delta_0]$
we have
$\|w(x)-w(t)\|<\eps/2$.
Choose 
$\delta\in(0,1)$ 
smaller than
$\min\{\eps/2, I_{\|w\|}/2,\delta_0 I_{\|w\|}^4(4(2S_{\|w\|}+1)|a-t|(\bar S_{\|w\|}+1)^2)^{-1}\}$,
define 
$\eps_1:=\delta \wedge (\eps/2)$ 
and pick arbitrary 
$u$ in 
$N_{\eps_1}(K_1):=\{u\in\cE_d^{(a)}:\sup_{p\in K_1}\|w(p)-u(p)\|<\eps_1\}$. 
Then 
$|\varrho_{\|w\|}^a(t)-\varrho_{\|u\|}^a(t)| <\delta_0/(\bar S_{\|w\|}+1)^2$
and hence
$\varrho_{\|u\|}^a(t)\in K_2:=[\varrho_{\|w\|}^a(t)-1,\varrho_{\|w\|}^a(t)+1]$.
As
$s=\varrho_{\|w\|}^a(t)$,
$c_{\|w\|}^a(s)=c_{\|u\|}^a(\varrho_{\|u\|}^a(t))$
and
$\sup\{\|u(c_{\|u\|}^a(q))\|^2:q\in K_2\} \leq (\bar S_{\|w\|}+1)^2$,
we have
\begin{equation}
\label{eq:c_is_cont}
|c_{\|w\|}^a(s) - c_{\|u\|}^a(s)| 
 \leq 
|\varrho_{\|w\|}^a(t)-\varrho_{\|u\|}^a(t)| (\bar S_{\|w\|}+1)^2
<\delta_0.
\end{equation}
Hence,
$ \|g_s(w)-g_s(u)\|\leq 
\|w(c_{\|w\|}^a(s)) - w(c_{\|u\|}^a(s))\|+
\|w(c_{\|u\|}^a(s)) - u(c_{\|u\|}^a(s))\|\leq \eps/2 + \eps/2 =\eps$
and the inclusion
$N_{\eps_1}(K_1)\subset g_s^{-1}(B)$,
implying the continuity of $g_s$,
follows.
Since $\delta_0$ could be arbitrarily small,
the bound in~\eqref{eq:c_is_cont} 
also implies the continuity of $w\mapsto c_{\|w\|}^a(s)$.

The equality in part~\eqref{l:rho-2way-inf(v)} follows from~\eqref{l:rho-2way-inf(iv)} and \eqref{l:rho-2way-inf(iii)}.
What remains to be proved is that 
$(b,w,\theta)\mapsto \theta(\cdot+I_a^b(w))$ is continuous 
at an arbitrary point 
$(b_0,w_0,\theta_0)\in\Upsilon^{(a)}_1\times \cC(\R,\Sp{d-1})$.
Since for any $t\in\R$  we have 
$\|\theta_0(t+I_a^{b_0}(w_0))-\theta(t+I_a^b(w))\| 
\leq 
\|\theta_0(t+I_a^{b_0}(w_0))-\theta_0(t+I_a^{b}(w_0))\|+ 
\|\theta_0(t+I_a^{b}(w_0))-\theta_0(t+I_a^{b}(w))\|+ 
\|\theta_0(t+I_a^{b}(w))-\theta(t+I_a^{b}(w))\|$,
the uniform continuity of $\theta_0$
on any compact, together with the proximity of 
$(b_0,w_0)$
and
$(b,w)$,
yields  a uniform control on compacts of the first two terms. 
The third term is controlled by the proximity of $\theta_0$ and $\theta$
in $\cC(\R,\Sp{d-1})$.
The estimates, analogous to the ones in the proof of~\eqref{l:rho-2way-inf(iii)}, are omitted. 

Pick $(b_0,b_0',w_0)$ in the domain of the map in~\eqref{l:rho-2way-inf(vi)}
and let  $(b,b',w)$ be an arbitrary element close to it. 
If $b_0=b_0'$, then $\varrho^{b_0'}_w(b_0)=0$ and $w_0(b_0)>0$. Then
$b$ and $b'$ must be very close to $b_0$ (and hence each other) and
$w$ must be positive in the neighbourhood of $b_0$. Hence the continuity 
of the map in~\eqref{l:rho-2way-inf(vi)} follows. If $b_0<b_0'$, then 
$-\varrho^{b_0'}_w(b_0)=\int_{b_0}^{b_0'}\ud u/w_0^2(u)$ and 
$w_0$ is bounded away from zero on compact interval 
$K\supset [b_0,b_0']$. 
Moreover, we may assume that 
$b<b'$,
$K\supset [b,b']$ 
and that 
$w$ is uniformly close to $w_0$
on $K$.
Hence 
$|\varrho^{b_0'}_w(b_0)- \varrho^{b'}_w(b)|$
is arbitrarily small and the continuity follows. The remaining case $b_0'<b_0$
is analogous.
\end{proof}

\begin{remark}
The continuity of the functions $g_s$, $s\in\R$, in the proof of Lemma~\ref{l:rho-2way-inf}\eqref{l:rho-2way-inf(iii)}
above does not imply the continuity of the map $\Phi_a^{-1}$.
\end{remark}

Define
$\cE_d^+:=\cup_{a>0}\cE_d^{(a)}\subset \cE_d$ (for $d\in\N$)
with the topology induced by that of  $\cC_d$.

\begin{proposition}
\label{prop:beta_Psi}
The excursion measure of $r$ satisfies $\mu_r(\cE_1\setminus \cE^+_1)=0$. 
Let $\Pr_\Psi$ be the law on $\cC(\R,\Sp{d-1})$ 
from Prop.~\ref{lem:sphere-in-Rd-SDE}.
Then there exists a unique $\sigma$-finite atomless Borel measure 
$\nu$
on 
$\cE^+_d$,
satisfying 
$\nu(A\cap \cE^{(a)}_d)= \mu_r\otimes\Pr_\Psi[\Phi_a^{-1}(A\cap\cE^{(a)}_d)]$
for all $a>0$ and Borel measurable $A\subseteq \cE_d^+$.
\end{proposition}

\begin{remark}
By Prop.~\ref{prop:beta_Psi}, 
$e^r$
is a PPP on 
$\cE_1^+\cup\{\delta_1\}$
and
$\nu$ 
induces a PPP on
$\cE_d^+\cup\{\delta_d\}$.
\end{remark}

\begin{proof}
In order to establish
$\mu_r(\cE_1\setminus \cE^+_1)=0$, note that 
by~\cite{py}, 
the excursion measure $\mu_r$ has the following representation:
any excursion 
$e^r_\la$ 
has a finite maximum  and this maximum is attained at a unique time.  
Furthermore, conditional on the
maximum being at some level $M>0$, the excursion has the same law as the path
formed by taking two independent $\Bes^{4-\delta}(0)$ processes, both run up
until their first hitting time of the level $M$, and placing them end-to-end.
Since $2<4-\delta<3$, by Lemma~\ref{lem:angular_clock}, any excursion in the support of 
$\mu_r$ is in $\cE_1^+$. 

Let $\Psi = (\Psi^\la,\la \geq 0)$ be a family of independent stationary diffusions $\Psi^\la=(\Psi^\la_t,t\in \R)$
with the law $\Pr_\Psi$ from Prop.~\ref{lem:sphere-in-Rd-SDE}.  
Assume that $r$ is independent of $\Psi$.
By the Marking and Mapping theorems of~\cite{kingman_PPP}
(the latter applies since $\Phi_a$ is measurable and bijective by Lemma~\ref{l:rho-2way-inf}\eqref{l:rho-2way-inf(iii)}),
the point process $e^{r,\Psi,a} = (e^{r,\Psi,a}_\la, \la \geq 0)$, defined by
$e^{r,\Psi,a}_\la := \delta_d$,
if $\tau_\la^r\leq a$, and 
$e^{r,\Psi,a}_\la := \Phi_a(e^r_\la,\Psi^\la)$,
if $\tau_\la^r> a$, 
is a PPP
in $\cE_d^{(a)} \cup \{\delta_d\}$ 
with excursion measure 
$\mu_r\otimes \Pr_\Psi[\Phi_a^{-1}(\cdot)]$
on $\cE_d^{(a)}$ of finite total mass 
$\mu_r\otimes \Pr_\Psi[\Phi_a^{-1}(\cE_d^{(a)})]=\mu_r(\cE_1^{(a)})<\infty$.
Moreover, 
by~\cite[p.~13]{kingman_PPP},
$\mu_r\otimes \Pr_\Psi[\Phi_a^{-1}(\cdot)]$
is atomless. 
Hence any measure
$\nu$ 
satisfying the identity in the proposition for all $a\in(0,\infty)$
is also atomless, $\sigma$-finite and unique. 
The next claim implies the proposition.

\noindent \textbf{Claim.} 
$\mu_r\otimes \Pr_\Psi[\Phi_a^{-1}(A)]=\mu_r\otimes \Pr_\Psi[\Phi_b^{-1}(A)]$
for any $0<b<a$ and measurable $A\subseteq \cE_d^{(a)}$.

Consider $Q:\cE_1^{(a)}\times\cC(\R,\Sp{d-1})\to\cE_1^{(a)}\times\cC(\R,\Sp{d-1})$,
$Q(w,\theta):=Q_b(a,w,\theta)$, where $Q_b$ is defined in Lemma~\ref{l:rho-2way-inf}\eqref{l:rho-2way-inf(v)}.
Hence
$Q=\Phi_a^{-1}\circ \Phi_b|_{\cE_1^{(a)}\times\cC(\R,\Sp{d-1})}$ 
is a Borel isomorphism. 
It suffices to show that 
$Q$ is measure preserving, i.e. 
$\mu_r\otimes \Pr_\Psi[B]=\mu_r\otimes \Pr_\Psi[Q(B)]$ for any 
measurable $B\subseteq\cE_1^{(a)}\times\cC(\R,\Sp{d-1})$.
The measure
$(\mu_r/\mu_r(\cE_1^{(b)}))\otimes \Pr_\Psi$,
restricted to $\cE_1^{(b)}\times\cC(\R,\Sp{d-1})$,
is the probability law of the random element $(X,Y):=(e^r_{\la_b},\Psi^{\la_b})$, 
where
$\la_b$ is the time of the first jump 
of size greater than $b$
of the subordinator $L^{-1}$. 
In particular, we need to show 
$\Pr[(X,Y)\in B] =\Pr[Q^{-1}(X,Y)\in B]$. 
Since 
$Q^{-1}(w,\theta)=(w,\theta(\cdot-I_b^a(w)))$,
$I_b^a(w)$ depends only on $w$ by Lemma~\ref{l:rho-2way-inf}\eqref{l:rho-2way-inf(iv)}
and, by Prop.~\ref{lem:sphere-in-Rd-SDE},
the process $Y$  is stationary,
it holds that 
$\Pr[(X,Y)\in B|\sigma(X)]= \Pr[Q^{-1}(X,Y)\in B|\sigma(X)]$,
implying the claim. 
\end{proof}

\subsubsection{Proof of Theorem~\ref{t:well-posedness}}
\label{subsub:Proof_of_weak_uniquenes_rec}
Let $(\x,W)$ be a solution of SDE~\eqref{eqn:x-SDE-sym} with $\x_0=\0$, adapted to $(\cF_t, t \geq 0)$.
Since we are only interested in the law of the solution, we may assume that we are in the canonical setting,
i.e. the probability space is $\Omega=\cC(\RP,\R^n)$ (for some $n\in\N$) and the filtration satisfies the usual
conditions with respect to the probability measure $\Pr$ on $\Omega$.
Define the point process
$e^\x=(e^\x_\ell,\ell\geq0)$
of excursions of  $\x$ away from $\0$ by
$e^\x_\ell:=\delta_d$ if $\ell\in\RP\setminus\La^r$, and 	
$e^\x_\ell:\RP\to\R^d$, where 
\begin{equation}
\label{eq:Def_PPP_e_x}
e^\x_\ell(u):= \begin{cases} \x_{L_{\ell-}^{-1} + u}  & u \in (0,\tau^r_\ell), \\
\0 & u \in \RP\setminus(0,\tau_\ell^r),
\end{cases}
\end{equation}
if $\ell \in \La^r$ (the notation introduced earlier in Section~\ref{sec:excursions} will be used throughout Section~\ref{subsub:Proof_of_weak_uniquenes_rec}).
The point process
$\|e^\x\| = (\|e^\x_\ell\|,\ell \geq 0)$
with excursions 
$\|e^\x_\ell(u)\|= r_{L^{-1}_{\ell-}+u} \1{u \leq \tau^r_\ell}$, $u \in \RP$,
for any $\ell\in\La^r$, 
is clearly equal to the PPP
$e^r$
defined above. 
Since $\x_t=\0$ if and only if $r_t=0$, $e^\x$ takes values in 
$\cE_d^+ \cup \{\delta_d\}$.
The key step in the proof of Theorem~\ref{t:well-posedness} 
is to show that $e^\x$ is indeed a PPP with excursion measure from Proposition~\ref{prop:beta_Psi}. 

For the rest of the section, 
fix an arbitrary 
$(\cF_t)$-stopping time $\tau$ with $\Pr[\tau<\infty]=1$.
Then $L^{-1}_{L_\tau}$ is an $(\cF_t)$-stopping time.
Define 
$\tilde r=(\tilde r_u,u\geq0)$
by
$\tilde r_u:=r_{L^{-1}_{L_\tau}+u}$. 
By the strong Markov property of $r$, 
the process $\tilde r$ is strong Markov with respect to the filtration
$(\cF_{L^{-1}_{L_\tau}+u},u\geq0)$, has the same law as $r$ and is independent of
$\cF_{L^{-1}_{L_\tau}}$.
The (Markov) local time $(\tilde L_u,u\geq0)$ 
of 
$\tilde r$ at $0$ satisfies
$\tilde L_u=L_{L^{-1}_{L_\tau}+u} - L_\tau$.
The inverse local time $\tilde L^{-1}=(\tilde L^{-1}_\mu,\mu\geq0)$ is a subordinator
satisfying
$\tilde L^{-1}_\mu=L^{-1}_{L_\tau+\mu} -L^{-1}_{L_\tau}$,
independent of $\cF_{L^{-1}_{L_\tau}}$.
Pick $a>0$ and define recursively the stopping times: 
$\mu_a^{0}:=0$ and $\mu_a^{n}:=\inf\{t>\mu_a^{n-1}:\tau^r_{t+{L_\tau}}>a\}$ for 
any $n\in\N$. Here $\tau^r_{t+L_\tau}=\tau^{\tilde r}_{t}:=\tilde L^{-1}_t-\tilde L^{-1}_{t-}$ is the
jump of the subordinator $\tilde L^{-1}$ and $\mu_a^{n}$ is the epoch of local time corresponding to 
the $n$-th excursion of $\tilde r$, lasting longer than $a$. 
For any $u\in\RP$, the equality
$e^r_{L_{u+L^{-1}_{L_\tau}}} = e^{\tilde r}_{\tilde L_u}$
holds,
where 
$(e^{\tilde r}_\mu,\mu\geq0)$ 
is given by 
$e^{\tilde r}_\mu:=\tilde r_{\tilde L^{-1}_{\mu-}+u} \1{u \leq \tau^{\tilde r}_\mu}$, $u \in \RP$.
Finally, for any $b\in(0,a)$, let 
$N_b(t):=\sup\{m\in\N: \tilde L^{-1}_{\mu_b^m-}<t \}$ (with convention $\sup\emptyset :=0$)
be the number of excursions of $\tilde r$ started before time $t\in\RP$ with length at least $b$.
Note that all the random elements defined in this paragraph 
depend on the choice of the stopping time $\tau$.

\begin{theorem}
\label{thm:master_identity_excursion}
Suppose that \eqref{ass:cov_form}, \eqref{ass:sigma_smooth} and \eqref{ass:radial-evec}
hold, with $U=1$ and $V \in (1,2)$.  
For any $a>0$, $n\in\N$ 
and finite $(\cF_t)$-stopping time $\tau$,
the regular conditional distribution 
of the random element $e^\x_{L_\tau+\mu^n_a}$ (defined in~\eqref{eq:Def_PPP_e_x} with $\ell=L_\tau+\mu^n_a$) in $\cE^{(a)}_d$,
given $\cF_{L^{-1}_{L_\tau}}$,
takes the form
$$\Pr[e^\x_{L_\tau+\mu^n_a}\in \cdot|\cF_{L^{-1}_{L_\tau}}] = \mu_r\otimes\Pr_\Psi[\Phi_a^{-1}(\cdot)]/\mu_r(\cE^{(a)}_1)\qquad\text{a.s.}$$
Here the law $\Pr_\Psi$ on $\cC(\R,\Sp{d-1})$ is defined in Prop.~\ref{lem:sphere-in-Rd-SDE} and $\mu_r$
is the excursion measure of the PPP $e^r$.
In particular,
the excursion 
$e^\x_{L_\tau+\mu^n_a}$
is independent of 
$\cF_{L^{-1}_{L_\tau}}$
and its law on $\cE^{(a)}_d$,
$\mu_r\otimes\Pr_\Psi[\Phi_a^{-1}(\cdot)]/\mu_r(\cE^{(a)}_1)$,
depends neither on $n\in\N$ nor on the stopping time $\tau$. 
\end{theorem}

\begin{remark}
Theorem~\ref{thm:master_identity_excursion} would follow trivially if we knew that
$\x$ was strong Markov. However, this cannot be assumed \textit{a priori}.
Once the uniqueness in law of SDE~\eqref{eqn:x-SDE-sym} has been established,
the strong Markov property of $\x$ follows. 
\end{remark}

As
$e^\x_{L_\tau+\mu_a^n}\in\cE_d^{(a)}$, 
we can define the process
$\theta^{a,n}$ 
with paths in $\cC(\R,\Sp{d-1})$
by
$(e^r_{L_\tau+\mu_a^n},\theta^{a,n}):=\Phi_a^{-1}(e^\x_{L_\tau+\mu_a^n})$.
The key step in the proof of Theorem~\ref{thm:master_identity_excursion}
is given by the following lemma.

\begin{lemma}
\label{prop:excursion-angular}
Under assumptions (and notation) of 
Theorem~\ref{thm:master_identity_excursion},
the regular conditional distribution of $\theta^{a,n}$ takes the form
$\Pr[\theta^{a,n}\in \cdot|\cF_{L^{-1}_{L_\tau}}\vee \cF^r_\infty]=\Pr_\Psi[\cdot]$ a.s.
(recall $\cF^r_\infty=\sigma(r_t,t\geq0)$).
\end{lemma}

\begin{proof}
Since $\cC(\R,\Sp{d-1})$ is  Polish, the regular conditional distribution  $\Pr[\theta^{a,n}\in \cdot|\cF_{L^{-1}_{L_\tau}}\vee \cF^r_\infty]$ exists. 
Moreover, as every trajectory of 
$\theta^{a,n}$ is continuous,
it is sufficient to prove that 
$\Pr$-a.s. the finite-dimensional distributions at rational times coincide with those of
$\Pr_\Psi$.
Since the set of all finite subsets of the rationals is countable
and the Borel $\sigma$-algebra on $\Sp{d-1}$ is generated by a countable family of open balls, 
by a diagonalisatoin argument
it suffices to prove that the finite-dimensional distributions 
at a given set of (rational) times (evaluated on the products of 
the finite intersections of generating sets) coincide $\Pr$-a.s. 
We establish this in two steps.
First, we show that the process
$(\theta^{a,n}_t,t\geq0)$
solves SDE~\eqref{eqn:sphere-SDE-two-drivers}, started at $\theta^{a,n}_0=\hat \x_{a+L^{-1}_{(L_\tau+\mu_a^n)-}}$
and driven by a Brownian motion $B$ independent of $\cF^r_\infty$.
Second, we use this to prove the equality of the finite-dimensional marginals of the two measures.

Since, 
for $s\in\RP$, 
the map 
$w\mapsto c^a_w(s)$ 
on $\cE^{(a)}_d$ is continuous (and hence measurable) 
by Lemma~\ref{l:rho-2way-inf}\eqref{l:rho-2way-inf(iii)},
we may define a non-negative random variable 
$\eta_a(s):=c^a_{e^r_{L_\tau+\mu_a^n}}(s)+L^{-1}_{(L_\tau+\mu_a^n)-}$.
Since 
$\eta_a(0)-L^{-1}_{L_\tau}$
is the first time an excursion of $\tilde r$
lasts longer than $a$,
after $n-1$ such excursions 
have occurred,
$\eta_a(0)$
is a finite $(\cF_t)$-stopping time.  
The definition of $c^a_w$
implies that 
$\eta_a(s)=\eta_a(0)+\inf\{t\in(0,\infty):\int_{\eta_a(0)}^{\eta_a(0)+t}r_u^{-2}\ud u \geq s\}$
is also an $(\cF_t)$-stopping time for any $s>0$.  In fact for 
$0\leq s\leq u$ 
it holds that 
$\eta_a(s)\leq \eta_a(u)<L^{-1}_{L_\tau+\mu_a^n}$.
Put differently, $(\eta_a(s),s\geq0)$ 
is a stochastic time-change and
we can define the filtration 
$(\cG_s, s \geq 0)$ by 
$\cG_s:=\cF_{\eta_a(s)}$.

Since 
$r^{-1}_{\eta_a(0) +\cdot}$  is continuous and $(\cF_{\eta_a(0)+t})$-adapted
on the stochastic interval $(0, L^{-1}_{L_\tau+\mu_a^n}-\eta_a(0))$, 
we can define continuous local martingales 
$A=(A_s; s\geq 0)$
and
$\zeta=(\zeta_s; s\geq 0)$
by
\begin{equation*}
A_s := \int_{\eta_a(0)}^{\eta_a(s)}   r_u^{-1} \ud W_u \qquad \text{and}\qquad
\zeta_s := \int_{\eta_a(0)}^{\eta_a(s)}   r_u^{-1} \ud Z_u,
\end{equation*}
where $Z$ is given in~\eqref{eqn:Z-def-rad}.
Both $A$ and $\zeta$ are adapted to 
$(\cG_s, s \geq 0)$.
As in the proof of Proposition~\ref{prop:transient-away-from-0},
it follows that
$A$ and $\zeta$
are 
$(\cG_s)$-Brownian motions.
Apply~\cite[Prop.~V.1.4]{ry}
and~\eqref{eqn:Z-def-rad}
to 
$\zeta$
to obtain
$\zeta_s = \int_0^s(\hat \x_{\eta_a(u)})^\tra r_{\eta_a(u)}^{-1}\ud W_{\eta_a(u)}$. 
Similarly we get
$A_s = \int_0^s   r_{\eta_a(u)}^{-1}\ud W_{\eta_a(u)}$.
Since by definition
$\hat \x_{\eta_a(u)}=\theta^{a,n}_u$ for all $u\in\RP$,
we find
$\zeta_s = \int_0^s(\theta^{a,n}_u)^\tra \ud A_u$ for all $s\geq0$.
Without loss of generality there exists 
a one-dimensional $(\cF_t)$-Brownian motion,
$\bar \gamma=(\bar \gamma_t, t \geq 0)$, independent of $(\x,W)$. 
Define a
$(\cG_s)$-Brownian motion
$\gamma=(\gamma_t, t \geq 0)$ by $\gamma_s:=\int_{\eta_a(0)}^{\eta_a(s)}   r_u^{-1} \ud \bar \gamma_u$.
Then, 
as in the proof of Proposition~\ref{prop:transient-away-from-0},
the process 
$B=(B_t, t \geq 0)$, 
$B_s := A_s - \int_0^s \theta^{a,n}_u \ud \zeta_u + \int_0^s \theta^{a,n}_u \ud \gamma_u$,
is a $d$-dimensional 
$(\cG_s)$-Brownian motion, independent of $\zeta$.

\noindent \textbf{Claim.} 	
$B$ is independent of $Z$ and hence (by Lemma~\ref{l:radial-bessel}) of $r$.\\
\noindent \textit{Proof of Claim.}
Recall that 
$\eta_a(0)$ and $L^{-1}_{L_\tau+\mu_a^n}$ are $(\cF_t)$-stopping times. 
Since 
$B_0=\0$,
$B$
is independent of 
$\cG_0=\cF_{\eta_a(0)}$
and hence of 
$(Z_s,0\leq s \leq \eta_a(0))$.
$B$ is  measurable with respect to 
$\bigvee_{s\in\RP}\cG_s\subseteq  \cF_{L^{-1}_{L_\tau+\mu_a^n}}$
and hence independent  of the Brownian motion 
$(Z_{u+L^{-1}_{L_\tau+\mu_a^n}}-Z_{L^{-1}_{L_\tau+\mu_a^n}},u\geq0)$.
We now prove that 
$B$
is independent of the stopped Brownian motion
$(\bar Z_s, s\geq0)$,
$\bar Z_s:=Z_{(s+\eta_a(0))\wedge L^{-1}_{L_\tau+\mu_a^n}}-Z_{\eta_a(0)}$.
Define the $\cG_s$-local martingale 
$M=(M_u, u \geq 0 )$,
$M_u:=Z_{\eta_a(u)}-Z_{\eta_a(0)}$,
and note that 
$M_u = \int_0^u r_{\eta_a(v)} (\theta^{a,n}_v)^\tra\ud A_v =\int_0^u r_{\eta_a(v)}\ud \zeta_v$.
Hence the covariation of $M$ and $B$ is identically equal to zero. 
Furthermore, the quadratic variation 
$[M]_u
= c^a_{e^r_{L_\tau+\mu_a^n}}(u) - a$ 
of $M$  
converges, i.e. $[M]_\infty:=\lim_{u\uparrow \infty} [M]_u=L^{-1}_{L_\tau+\mu_a^n}-\eta_a(0)$,
with inverse given by $v\mapsto \varrho^a_{e^r_{L_\tau+\mu_a^n}}(a+v)$, $v\in[0,[M]_\infty)$.
Since the limit
$M_\infty:=\lim_{u\uparrow\infty} M_u=Z_{L^{-1}_{L_\tau+\mu_a^n}}-Z_{\eta_a(0)}$ exists, 
we can define the processes $(M_{\varrho^a_{e^r_{L_\tau+\mu_a^n}}(a+t)},0\leq t \leq [M]_\infty)$, which is independent 
of $B£$ by~\cite[Thm~V.1.9]{ry}. The claim follows by noting that 
$M_{\varrho^a_{e^r_{L_\tau+\mu_a^n}}(a+t)}=\bar Z_t$ for any $t\in[0,[M]_\infty]$.

By Lemma~\ref{l:radial-bessel},
the process
$r^{-2}_{\eta_a(0) +\cdot}$ is a continuous semimartingale 
on the stochastic interval $(0,\tau^r_{L_\tau+\mu_a^n}-a)$.
In particular, an analogous calculation to the one that established~\eqref{eqn:theta-SDE} implies
\begin{equation*}
\hat \x_{\eta_a(0) +t} = \hat \x_{\eta_a(0) } + \int_{\eta_a(0)}^{\eta_a(0) +t} f( \hat \x_u ) r^{-2}_u  \ud u 
+\int_{\eta_a(0) }^{\eta_a(0)+t}  g(\hat \x_u) r^{-1}_u \ud W_u, \quad t \in (0,\tau_{L_\tau+\mu_a^n}^r-a),
\end{equation*}
with $f,g$ in~\eqref{eqn:theta-SDE-coeffs}.
Applying the stochastic time-change
$(c^a_{e^r_{L_\tau+\mu_a^n}}(u) - a ,u\geq0)$ with~\cite[Prop.~V.1.4]{ry}
and noting that 
$\eta_a(u)=\eta_a(0)+c^a_{e^r_{L_\tau+\mu_a^n}}(u) - a$
and
$\hat \x_{\eta_a(u)}=\theta^{a,n}_u$ 
for all $u\in\RP$,
implies that
$(\theta^{a,n}_u,u\geq0)$
satisfies the SDE in~\eqref{eqn:sphere-SDE-two-drivers}, started at $\theta^{a,n}_0=\hat \x_{a+L^{-1}_{(L_\tau+\mu_a^n)-}}$
driven by the Brownian motion $A$ defined above. 
It is easy to see from the definition of the Brownian motion $B$ above that
$\int_0^t (\ssym(\theta^{a,n}_u) - \theta^{a,n}_u(\theta^{a,n}_u)^\tra) \ud B_u 
= \int_0^t (\ssym(\theta^{a,n}_u) - \theta^{a,n}_u(\theta^{a,n}_u)^\tra) \ud A_u$
for all $t\geq0$.
Hence 
$(\theta^{a,n}_u,u\geq0)$
satisfies SDE~\eqref{eqn:sphere-SDE-two-drivers} driven by $B$. 
By the Claim, $r$ and $\theta^{a,n}$ are independent. 

The second step in the proof of the lemma analyses the conditional law of
$\theta^{a,n}$.
The number of excursions longer than $b$ started before the start of the $n$-the excursion of $\tilde r$ of 
length at least $a$, i.e. 
$N_b(\tilde L^{-1}_{\mu_a^n-})$,
is 
$\cF^r_\infty$ measurable. 
Fix $t\in\R$
and note that by Lemma~\ref{l:rho-2way-inf}\eqref{l:rho-2way-inf(iv)}
we have
$\lim_{b\downarrow 0}t+I_b^a(e^r_{L_\tau+\mu_a^n})=\infty$.
On the event 
$\{N_b(\tilde L^{-1}_{\mu_a^n-})=k-1\}$,
by Lemma~\ref{l:rho-2way-inf}\eqref{l:rho-2way-inf(iv)}--\eqref{l:rho-2way-inf(iii)}, it holds that 
$\theta^{a,n}_t=\theta^{b,k}_{t+I_b^a(e^r_{L_\tau+\mu_a^n})}$. 
Pick an arbitrary measurable subset $\fA\subseteq \Sp{d-1}$. 
Then it holds that 
\begin{equation*}
\Pr[\theta^{a,n}_t\in \fA|\cF_{L^{-1}_{L_\tau}}\vee \cF^r_\infty]=
\sum_{k\in\N}\1{N_b(\tilde L^{-1}_{\mu_a^n-})=k-1} \Pr[\theta^{b,k}_{t+I_b^a(e^r_{L_\tau+\mu_a^n})}\in \fA |\cF_{L^{-1}_{L_\tau}}\vee \cF^r_\infty].
\end{equation*}
For all $b\in(0,a)$ such that $I_b^a(e^r_{L_\tau+\mu_a^n}) > - t$, 
the first step of the proof implies 
\begin{equation}
\label{eq:Bound_Law}
|\Pr[\theta^{a,n}_t\in \fA|\cF_{L^{-1}_{L_\tau}}\vee \cF^r_\infty]-\mu(\fA)|\leq 
\int_{\Sp{d-1}}|P_{t+I_b^a(e^r_{L_\tau+\mu_a^n})}(x,\fA)-\mu(\fA)| \Pr_b[\ud x],
\end{equation}
where  
$\Pr_b[\ud x]:=\sum_{k\in\N}\1{N_b(\tilde L^{-1}_{\mu_a^n-})=k-1} \Pr[\theta^{b,k}_0\in \ud x |\cF_{L^{-1}_{L_\tau}}\vee \cF^r_\infty]$
is a probability measure  on $\Sp{d-1}$, 
$P$ 
is the transition function from Prop.~\ref{lem:sphere-in-Rd-SDE}
and $\mu$ denotes its stationary measure.  
By~\eqref{eqn:unif-ergod} in Prop.~\ref{lem:sphere-in-Rd-SDE}, Lemma~\ref{l:rho-2way-inf}\eqref{l:rho-2way-inf(iv)} and~\eqref{eq:Bound_Law}, for any $\epsilon>0$
there exists $b\in(0,a)$ such that 
$|\Pr[\theta^{a,n}_t\in \fA|\cF_{L^{-1}_{L_\tau}}\vee \cF^r_\infty]-\mu(\fA)|\leq \epsilon$.
Hence we must have
$\Pr[\theta^{a,n}_t\in \fA |\cF_{L^{-1}_{L_\tau}}\vee \cF^r_\infty]=\mu(\fA)=\Pr_\Psi[\{f\in\cC(\R,\Sp{d-1}):f(t)\in \fA\}]$.
An analogous argument shows that finite-dimensional distributions of $\Pr_\Psi[\cdot]$
and $\Pr[\theta^{a,n}_t\in \cdot |\cF_{L^{-1}_{L_\tau}}\vee \cF^r_\infty]$ coincide. 
This proves the lemma.
\end{proof}


\begin{proof}[Proof of Theorem~\ref{thm:master_identity_excursion}] 
Pick an arbitrary measurable set $B$ in $\cE^{(a)}_d$ and define a subset $A:=\Phi_a^{-1}(B)$
of $\cE^{(a)}_1\times\cC(\R,\Sp{d-1})$. 
A standard argument, based on the Monotone-Class Theorem, implies that the function
$F_A:\cE^{(a)}_1\to[0,1]$, given by $F_A(\epsilon):=\int_{\cC(\R,\Sp{d-1})}\1{A}(\epsilon,f)\Pr_\Psi[\ud f]$,
is measurable. 
Hence Lemma~\ref{prop:excursion-angular}, the tower property and the definition of the map $\Phi_a^{-1}$ imply 
$\Pr[e^\x_{L_\tau+\mu^n_a}\in B|\cF_{L^{-1}_{L_\tau}}]= 
\Pr[(e^r_{L_\tau+\mu^n_a},\theta^{a,n})\in A|\cF_{L^{-1}_{L_\tau}}]= \Exp[ F_A(e^r_{L_\tau+\mu^n_a})|\cF_{L^{-1}_{L_\tau}}]$.
Since $r$ is strong Markov, we get 
$\Pr[e^\x_{L_\tau+\mu^n_a}\in B|\cF_{L^{-1}_{L_\tau}}]= \Exp[ F_A(e^r_{L_\tau+\mu^n_a})]$.
Since the law of the excursion $e^r_{L_\tau+\mu^n_a}$ is given by 
$\mu_r(\cdot)/\mu_r(\cE^{(a)}_1)$, the theorem follows. 
\end{proof}

Pick $v\in (0,\infty)$ and a measurable $B\subseteq \R^d$.
Let $B_v:=\{\hat \by: \by\in B\setminus\{\0\}, \|\by\|=v\}$ be the intersection
$B\cap(v\Sp{d-1})$ projected onto the unit sphere. For any $b\in\R$, define the measurable set
$\fA_v^b(B):=\{f\in\cC(\R,\Sp{d-1}): f(b)\in B_v\}$.

\begin{proposition}
\label{prop:time_excurion}
Pick $k\in\N$ and indices $0=:i_0<i_1<i_2<\cdots<i_{k-1}<i_k$.
Define $n:=i_k$ and 
choose measurable sets $B_1,\ldots,B_n\subseteq\R^d$
and
times $0<u_1<u_2<\cdots<u_n$. 
For $0\leq i<j\leq n$, let 
$F_{i,j}:(\RP\times(0,\infty))^{j-i}\to[0,1]$
be $F_{i,j}(b_p,v_p;i+1\leq p\leq j):=\Pr_\Psi[\cap_{p=i+1}^j\fA_{v_p}^{b_p}(B_p)]$.
Define 
$a_j:=u_j-\tilde L^{-1}_{\tilde L_{u_j}-}$ for any $j\in\{1,\ldots,n\}$
(recall that 
$\tilde L$ depends on $\tau$).
Then, on the event 
$E_k:=\{\tilde L_{ u_{i_0+1}}=\tilde L_{ u_{i_1}}
<\tilde L_{ u_{i_1+1}}=\tilde L_{ u_{i_2}}
<\tilde L_{ u_{i_2+1}}=\tilde L_{ u_{i_3}}
<\cdots
<\tilde L_{ u_{i_{k-1}+1}}=\tilde L_{ u_{i_k}}\}$,
it holds that 
\begin{multline}
\label{eq:time_master_formula}
\Pr\left[e^{\x}_{L_\tau+\tilde L_{u_j}} (a_j)\in B_j  \quad\text{for $j\in\{1,\ldots,n\}$}\left\vert \cF_{L^{-1}_{L_\tau}}\vee\cF_\infty^r\right.\right] \\
= \prod_{l=0}^{k-1} F_{i_l,i_{l+1}}\left(\varrho^{a_{i_l+1}}_{e^{\tilde r}_{\tilde L_{u_p}}}(a_p),e^{\tilde r}_{\tilde L_{u_p}}(a_p); i_l+1\leq p\leq  i_{l+1}\right).
\end{multline}
\end{proposition}

\begin{remark}
\label{rem:length_excursion}
In~\eqref{eq:time_master_formula}, 
for any 
$p\in\{i_l+1,\ldots, i_{l+1}\}$, 
it holds that
$\tilde L_{u_p}=\tilde L_{u_{i_l+1}}$
and hence $e^{\tilde r}_{\tilde L_{u_p}}$ refers to a single excursion. 
Note also that $E_k$ depends on the sequence $i_1<\cdots<i_k$ and not just on the index $k$. 
This information is suppressed from the notation for brevity. 
\end{remark}

\begin{proof}
A moment's reflection reveals that $F_{i,j}$, defined in the proposition, is measurable  and
$E_k\in\cF_\infty^r$. 
Note that 
$a_j$ 
is $\cF^r_\infty$-measurable
and
$a_j>0$ $\Pr$-a.s. 
for any $j\in\{1,\ldots,n\}$. 
Moreover, on $E_k$,
by Remark~\ref{rem:length_excursion}
the triplet 
$(a_{i_l+1}, a_p,e^{\tilde r}_{\tilde L_{u_p}})$
is in the domain of the map in 
Lemma~\ref{l:rho-2way-inf}\eqref{l:rho-2way-inf(vi)}
for all  
$l\in\{0,\ldots,k-1\}$ and $p\in\{i_l+1,\ldots, i_{l+1}\}$.
Hence
we may 
define $\cF_\infty^r$-measurable random variables
$t_l^p:=\varrho^{a_{i_l+1}}_{e^{\tilde r}_{\tilde L_{u_p}}}(a_p)$
and
$v_l^p:=e^{\tilde r}_{\tilde L_{u_p}}(a_p)$.
In fact, on $E_k$, $v_l^p>0$ and $t_l^p\geq 0$ $\Pr$-a.s.
Hence the right-hand side of~\eqref{eq:time_master_formula}
is well-defined on $E_k$ and $\cF_\infty^r$-measurable. 

Assume first that $k=1$, i.e. $i_1=n$, $E_1=\{\tilde L_{u_1} = \tilde L_{u_n}\}$
and
$a_j=u_j-\tilde L^{-1}_{\tilde L_{u_1}-}$ for $j\in\{1,\ldots,n\}$.
Pick $b>0$ and let
$E_1^b:=E_1\cap\{a_1>b\}$.
By~\eqref{l:rho-2way-inf(iii)} and~\eqref{l:rho-2way-inf(v)}
of
Lemma~\ref{l:rho-2way-inf},
the map 
$Q_b:
\Upsilon^{(b)}_1\times \cC(\R,\Sp{d-1})
\to
\cE^{(b)}_1\times \cC(\R,\Sp{d-1})$
is measurable.
Hence,
on 
$E_1^b$,
we may define a random element 
$Q_b(a_1,\Phi_b^{-1}(e^{\x}_{L_{u_j+L^{-1}_{L_\tau}}}))= \Phi_{a_1}^{-1}(e^{\x}_{L_{u_j+L^{-1}_{L_\tau}}})$.
Recall that $N_{a_1}(\tilde L^{-1}_{u_1})$
is the number of excursions or $\tilde r$ that started prior to 
$\tilde L^{-1}_{u_1}$
with length of at least $a_1$. 
Clearly, 
$N_{a_1}(\tilde L^{-1}_{u_1})$
is $\cF_\infty^r$-measurable. 
Hence,
conditional on 
$\cF_{L^{-1}_{L_\tau}}\vee\cF_\infty^r$,
the law of 
$\theta^{a_1, N_{a_1}(\tilde L^{-1}_{u_1})}$
equals 
$\Pr_\Psi[\cdot]$
by
Lemma~\ref{prop:excursion-angular},
where
$\Phi_{a_1}^{-1}(e^{\x}_{L_{u_j+L^{-1}_{L_\tau}}})=(e^{\tilde r}_{\tilde L_{u_j}},\theta^{a_1, N_{a_1}(\tilde L^{-1}_{u_1})})$.
On $E_1^b$, the left-hand side of~\eqref{eq:time_master_formula} is 
\begin{multline*}
\Pr\left[\theta^{a_1, N_{a_1}(\tilde L^{-1}_{u_1})}\in \fA_{v_0^j}^{t_0^j}(B_j) \text{ for $j\in\{1,\ldots,n\}$} \left\vert \cF_{L^{-1}_{L_\tau}}\vee\cF_\infty^r\right.\right]
= F_{0,n}(t_0^p, v_0^p; 1\leq p\leq n).
\end{multline*}
Since this identity is independent of $b$ and 
$E_1^b\nearrow E_1$ as $b\downarrow0$, 
the proposition holds for $k=1$ and any $i_1=n\in\N$.

We proceed by induction:  
assume that~\eqref{eq:time_master_formula} holds for some $k\in\N$
and  any increasing sequence of indices of length at most $k$.
Pick an event $E_{k+1}$. Put differently, choose 
a sequence of indices 
$0=i_0<i_1<\cdots<i_k<i_{k+1}=n$.
The $(\cF_t)$-stopping time
$\rho:=L^{-1}_{L_\tau}+u_{i_k}$
satisfies 
$L^{-1}_{L_\tau}<\rho\leq L^{-1}_{L_\rho}$.
Since 
$L^{-1}_{L_\rho}$
is an $(\cF_t)$-stopping time,
the $\sigma$-algebra 
$\cF_{L^{-1}_{L_\rho}}$
is well-defined and
contains   
$\cF_{L^{-1}_{L_\tau}}$.
For the sequence
$0< i_1<\cdots<i_k$, define the 
event $E_k$ as in the statement of the proposition.
Note that 
$E_{k+1}=E_k\cap E'_{k+1}$, where $E'_{k+1}:=\{\tilde L_{u_{i_k}}<\tilde L_{u_{i_k+1}}=\tilde L_{u_{i_{k+1}}}\}$,
and $E_{k+1},E_k, E'_{k+1}\in\cF_\infty^r$.
Define a $\Bes^V(0)$ process
$r'=(r'_u,u\geq0)$
by
$r'_u:=r_{L^{-1}_{L_\rho}+u}$.
Then its Markov (resp. inverse) local time 
$L'=(L'_u,u\geq0)$ 
(resp. $L'^{-1}=(L'^{-1}_\mu,\mu\geq0)$)
equals
$L'_u=L_{L^{-1}_{L_\rho}+u}-L_\rho$ 
(resp.
$L'^{-1}_\mu=L^{-1}_{L_\rho+\mu} -L^{-1}_{L_\rho}$)
and
$L'^{-1}$
is a subordinator
independent of 
$\cF_{L^{-1}_{L_\rho}}$.

Pick $j\in\{i_k+1,\ldots,i_{k+1}\}$.
On $E'_{k+1}$ the inequality
$u_j+L^{-1}_{L_\tau}>L^{-1}_{L_\rho}$
holds.
Hence we can define positive times
$u'_j:=u_j+L^{-1}_{L_\tau}-L^{-1}_{L_\rho}$
that clearly satisfy 
$r'_{u'_j}=\tilde r_{u_j}$.
Furthermore, we have
$$
L'_{u'_j}=L_{u_j+L^{-1}_{L_\tau}}-L_\rho
\quad\text{ and }\quad
L'^{-1}_{L'_{u'_j}-}=L^{-1}_{(L_{u_j+L^{-1}_{L_\tau}})-}-  L^{-1}_{L_\rho}.
$$
Hence we find 
$
a_j=u_j+L^{-1}_{L_\tau}  -  L^{-1}_{(L_{u_j+L^{-1}_{L_\tau}})-}  = u'_j- L'^{-1}_{L'_{u'_j}-}
\text{ for all $j\in\{i_k+1,\ldots,i_{k+1}\}$.}
$
Let 
$e^{r'}=(e^{r'}_\mu,\mu\geq0)$ 
be the PPP
given by 
$e^{r'}_\mu(u):=r'_{L'^{-1}_{\mu-}+u} \1{u \leq \tau^{r'}_\mu}$, $u \in \RP$,
where 
$\tau^{r'}_\mu:=L'^{-1}_\mu-L'^{-1}_{\mu-}$
is the size of the jump of the subordinator $L'^{-1}$ at the moment of local time $\mu$. 
It holds that 
$e^{\tilde r}_{\tilde L_{u_j}}=e^r_{L_{u_j+L^{-1}_{L_\tau}}} =e^r_{L_{u'_j+L^{-1}_{L_\rho}}} =e^{r'}_{L'_{u'_j}}$,
and  hence
$t^j_k=\varrho^{a_{i_k+1}}_{e^{r'}_{L'_{u'_j}}}(a_j)$,
$v^j_k = e^{r'}_{L'_{u'_j}}(a_j)$,
for all $j\in\{i_k+1,\ldots,i_{k+1}\}$.
Trivially it holds that 
$e^{\x}_{L_{u_j+L^{-1}_{L_\tau}}}=e^{\x}_{L_{u'_j+L^{-1}_{L_\rho}}}$,
so me may apply the basis of the induction (i.e. $k=1$) to the stopping time 
$\rho$
on the event $E'_{k+1}$
as follows:
\begin{multline*}
 \Pr\left[e^{\x}_{L_{u'_j+L^{-1}_{L_\rho}}}(a_j)\in B_j,  \quad j\in\{i_k+1,\ldots,i_{k+1}\}\left\vert \cF_{L^{-1}_{L_\rho}}\vee\cF_\infty^r\right.\right] \\
=
F_{i_k,i_{k+1}}\left(\varrho^{a_{i_k+1}}_{e^{r'}_{L'_{u'_j}}}(a_j),e^{r'}_{L'_{u'_j}}(a_j); i_k+1\leq j\leq  i_{k+1}\right).
\end{multline*}
Hence
$\Pr[e^{\x}_{L_{u_j+L^{-1}_{L_\tau}}}(a_j)\in B_j,\>  j\in\{i_k+1,\ldots,i_{k+1}\}\vert \cF_{L^{-1}_{L_\rho}}\vee\cF_\infty^r] 
=
F_{i_k,i_{k+1}}(t^j_k,v^j_k; i_k+1\leq j\leq  i_{k+1})$
on $E'_{k+1}$.
Define the event $D_k:=\cap_{j=1}^{i_k}\{e^{\x}_{L_{u_j+L^{-1}_{L_\tau}}}(a_j)\in B_j\}\cap E_k\in\cF_{L^{-1}_{L_\rho}}$.
On the event $E_{k+1}$, 
\begin{multline*}
\Exp\left[
\1{D_k}
\Pr\left[e^{\x}_{L_{u_j+L^{-1}_{L_\tau}}}(a_j)\in B_j,  \quad j\in\{i_k+1,\ldots,i_{k+1}\}\left\vert \cF_{L^{-1}_{L_\rho}}\vee\cF_\infty^r\right.\right] 
\left\vert \cF_{L^{-1}_{L_\tau}}\vee\cF_\infty^r\right.\right] \\
= 
\Pr\left[ D_k \left\vert \cF_{L^{-1}_{L_\tau}}\vee\cF_\infty^r\right.\right] 
F_{i_k,i_{k+1}}(t^j_k,v^j_k; i_k+1\leq j\leq  i_{k+1})
 \end{multline*}
equals the left-hand side in~\eqref{eq:time_master_formula}.
The proposition follows 
by
the induction hypothesis.
\end{proof}

\begin{corollary}
\label{cor:almost_strong_Markov}
Let $\x$ be a solution of SDE~\eqref{eqn:x-SDE-sym} started at $\0$ and adapted to $(\cF_t,t\geq0)$.
\begin{enumerate}
\item[(a)] Let $\tau$ be a finite $(\cF_t)$-stopping time. Then the process 
$\tilde \x=(\tilde \x_t, t\geq0)$, defined by 
$\tilde \x_t:=\x_{L^{-1}_{L_\tau}+t}$, 
is independent of $\cF_{L^{-1}_{L_\tau}}$
and has the same law as $\x$.
\item[(b)] Let $\y$ be a solution of SDE~\eqref{eqn:x-SDE-sym} started at $\0$. Then the laws on $\cC_d$ of $\x$ and $\y$ coincide. 
\end{enumerate}
\end{corollary}

\begin{proof}
(a) 
If we prove that
for any 
$0<u_1<u_2<\cdots<u_n$
and
measurable sets $B_1,\ldots,B_n\subseteq\R^d$,
the equality
$\Pr[\tilde \x_{u_1}\in B_1,\ldots, \tilde \x_{u_n}\in B_n\vert \cF_{L^{-1}_{L_\tau}}]= 
\Pr[\x_{u_1}\in B_1,\ldots, \x_{u_n}\in B_n]$ 
holds $\Pr$-a.s., 
part~(a) follows
by a diagonalisation argument (cf. first paragraph in the proof of Lemma~\ref{prop:excursion-angular}),
since $\tilde \x_0=\x_0$
and all the trajectories of $\tilde \x$ are continuous.
Recall that $L_{L^{-1}_{L_\tau}+u}=L_\tau+\tilde L_u$.
Hence,
for all $u\geq0$,
$\tilde \x_u  = e^\x_{L_\tau+\tilde L_u} ( u - \tilde L^{-1}_{\tilde L_u-})$
and in particular (take $\tau\equiv0$)
$\x_u  = e^\x_{L_u} ( u - L^{-1}_{L_u-})$.
Note that the set 
$E_k$ in Proposition~\ref{prop:time_excurion} 
is determined by $k\in\{1,\ldots,n\}$ and 
the indices $i_1<\ldots<i_{k-1}$ (with $i_0=0$
and
$i_k=n$) and should be denoted by
$E_k^{i_1,\ldots,i_{k-1}}$. 
Furthermore,
$E_k^{i_1,\ldots,i_{k-1}}\cap E_{k'}^{i'_1,\ldots,i'_{k'-1}}\neq\emptyset$ 
if and only if $k=k', i_1=i'_1,\ldots, i_{k-1}=i'_{k'-1}$, in which case the two sets
clearly coincide. Put differently, this finite family of sets is  pairwise disjoint. 
Since the union of $E_k^{i_1,\ldots,i_{k-1}}$
equals the entire probability space,
we can define a path functional 
$$
F(\tilde \x):=\sum_{k,i_1<\cdots<i_{k-1}} \1{E_k^{i_1,\ldots,i_{k-1}}} 
\prod_{l=0}^{k-1} F_{i_l,i_{l+1}}\left(\varrho^{a_{i_l+1}}_{e^{\tilde r}_{\tilde L_{u_p}}}(a_p),e^{\tilde r}_{\tilde L_{u_p}}(a_p); i_l+1\leq p\leq  i_{l+1}\right).
$$
Note that $F$ 
is defined $\Pr$-a.s. on $\Omega$ and is measurable.
Furthermore, 
$F$ 
is a function only of the radial component $\tilde r= \|\tilde \x\|$ of 
$\tilde \x$. 
By Proposition~\ref{prop:time_excurion}, we get
$\Pr[\tilde \x_{u_1}\in B_1,\ldots, \tilde \x_{u_n}\in B_n\vert \cF_{L^{-1}_{L_\tau}}\vee \cF_\infty^r]=F(\tilde \x)$.
An identical argument applied to $\x$  (with $\tau\equiv0$)
yields
$\Pr[\x_{u_1}\in B_1,\ldots, \x_{u_n}\in B_n\vert\cF_0\vee \cF_\infty^r]=F(\x)$. 
By the strong Markov property of $r$, 
the process 
$\tilde r$, and therefore
$F(\tilde \x)$, is independent of 
$\cF_{L^{-1}_{L_\tau}}$.
Hence 
$\Pr[\tilde \x_{u_1}\in B_1,\ldots, \tilde \x_{u_n}\in B_n\vert \cF_{L^{-1}_{L_\tau}}]= \Exp[F(\tilde \x)]$
a.s. Since the laws of $r$ and $\tilde r$ coincide, we have
$\Exp[F(\tilde \x)]= \Exp[F(\x)]=\Pr[\x_{u_1}\in B_1,\ldots, \x_{u_n}\in B_n]$.
This concludes the proof of (a). 

\noindent (b) As before it is sufficient to show
$\Pr[\x_{u_1}\in B_1,\ldots, \x_{u_n}\in B_n]= \Pr'[\y_{u_1}\in B_1,\ldots, \y_{u_n}\in B_n]$ 
for any 
$0<u_1<u_2<\cdots<u_n$
and
measurable sets $B_1,\ldots,B_n\subseteq\R^d$,
where 
$\Pr'[\cdot]$ is
the probability measure on the space where $\y$ is defined. 
Proposition~\ref{prop:time_excurion} implies this statement, using the same argument as in part (a) 
as the processes $\|\x\|$ and $\|\y\|$ have the same law.
\end{proof}

\begin{corollary}\label{thm:law-of-excursions}
Let $\x$ be a solution of SDE~\eqref{eqn:x-SDE-sym} started at $\0$. 
The point process $e^\x$ on $\cE_d^+\cup\{\delta_d\}$, defined in~\eqref{eq:Def_PPP_e_x},  
is a PPP with excursion measure characterised in Proposition~\ref{prop:beta_Psi}.
\end{corollary}

\begin{proof}
Let $\x$ be adapted to $(\cF_t,t\geq0)$.
Pick $\la\in\RP$
and recall that 
$L^{-1}_\la$ is an $(\cF_t)$-stopping time. 
Define 
$\tilde \x=(\tilde \x_t, t\geq0)$
by 
$\tilde \x_t:=\x_{L^{-1}_{\la}+t}$. 

\noindent \noindent \textbf{Claim 1.} The process $\tilde \x$ is independent of
$\cF_{L^{-1}_{\la}}$
and its law is equal to that of $\x$.

\noindent \textit{Proof of Claim 1.} 
Define an $(\cF_t)$-stopping time 
$\tau:=\inf\{t\geq0: L_t\geq \la\}$.
Since the local time $L$ is continuous and 
$\lim_{t\uparrow\infty}L_t=\infty$ a.s., it holds that 
$\Pr[L_\tau = \la]=\Pr[\tau<\infty]=1$.
In particular, 
$L^{-1}_{\la}=L^{-1}_{L_\tau}$
and, by Corollary~\ref{cor:almost_strong_Markov}(a), 
the claim follows. 

Define the filtration $(\cG_\la,\la\geq0)$ by
$\cG_\la:=\cF_{L^{-1}_\la}$.
Pick 
$a>0$ and a measurable set 
$\fA\in\cE_d^{(a)}$.

\noindent \textbf{Claim 2.}  The counting process $N^\fA=(N^\fA_\la,\la\geq0)$, where $N^\fA_\la$ equals the cardinality of the
set $\{s\in(0,\la]: e^\x_s\in\fA\}$, is a $(\cG_\la)$-Poisson process with intensity $\mu_r\otimes\Pr_\Psi[\Phi_a^{-1}(\fA)]$. 

Before proving the claim, note that it implies that $e^\x$ is a PPP with excursion measure $\nu$
from Proposition~\ref{prop:beta_Psi}.
Indeed, for disjoint sets 
$\fA_1,\ldots, \fA_n$ in $\cE_d^{(a)}$,
the respective counting processes 
$N^{\fA_1},\ldots, N^{\fA_n}$
are, by Claim~2, 
$(\cG_\la)$-Poisson processes that
cannot jump simultaneously. Hence they must be independent.  For any collection of disjoint sets
$\fA_1\times (s_1,t_1],\ldots, \fA_n\times (s_n,t_n]$ in $\cE_d^+\times \RP$ satisfying 
$0<\nu(\fA_j)<\infty$ for all $j\in\{1,\ldots,n\}$, 
by Proposition~\ref{prop:beta_Psi} there exists $a>0$ such that 
all the sets are contained in $\cE_d^{(a)}\times \RP$. 
Furthermore, the numbers of points of $e^\x$ in each of the sets is given by $n$
independent Poisson rvs $N^{\fA_j}_{t_j}-N^{\fA_j}_{s_j}$
with intensities $(t_j-s_j)\nu(\fA_j)$.

\smallskip

\noindent \textit{Proof of Claim 2.} 
It is clear from the definition of $N^\fA$ that it is adapted to $(\cG_\la,\la\geq0)$.
Pick $\la,\mu\in\RP$. It is sufficient to prove that $N^\fA_{\mu+\la}-N^\fA_{\la}$
is independent of $\cG_\la$ and has the same law as $N^\fA_{\mu}$. 
The number of excursions of $\x$ 
in $\fA$
completed during the time interval $(L^{-1}_\la,L^{-1}_{\la+\mu}]$
is by construction equal to the number $\tilde N^\fA_\mu$ of excursions 
in $\fA$ of $\tilde \x$ from Claim~1, completed in the time interval
$(0,\tilde L_\mu^{-1}]$. 
Recall that 
$\tilde L_\mu^{-1}=L^{-1}_{\la+\mu}-L^{-1}_\la$
is the inverse local time at the origin of $\tilde r = \|\tilde \x\|$, and hence of $\tilde \x$.
Since, by Claim~1, $\tilde \x$ is independent of $\cG_\la$, 
so is $\tilde N^\fA_\mu=N^\fA_{\mu+\la}-N^\fA_{\la}$.
Since, by Claim~1, the laws of $\x$ and $\tilde \x$ coincide, so do the laws of 
$N^\fA_\mu$ 
and
$\tilde N^\fA_\mu$. This concludes the proof of Claim 2. 
\end{proof}

\section{Invariance principle}
\label{sec:invariance}


%

\subsection{Invariance principle with discontinuous coefficients}
\label{sec:invariance_conditional}

Recall that
$\cD_d = \cD(\RP ; \R^d)$ is a space of functions $x : \RP \to \R^d$ that are right-continuous
and have left limits
(i.e. $x(t):=\lim_{s\downarrow t}x(s)$ for any $t\in\RP$,
$x(t-):=\lim_{s\uparrow t}x(s)$ exists in $\R^d$ for any
$t>0$
and, by convention, $x(0-):=x(0)$). 
We endow $\cD_d$ with the Skorohod metric (see e.g.~\cite[\S 3.5]{ek}). 
By~\cite[Prop~3.5.3, p.~119]{ek},
the induced topology on the continuous functions $\cC_d = \cC(\RP ; \R^d)$ 
coincides with the compact-open topology.
Theorem~\ref{thm:conditional_invariance} may be viewed as an extension of~\cite[Thm~7.4.1, p.~354]{ek} to 
a setting with discontinuous coefficients. It is key in establishing 
Theorem~\ref{thm:invariance}.

\begin{theorem}
\label{thm:conditional_invariance}
Let $a = (a_{ij}):\R^d\to\R^d\otimes\R^d$ be a bounded 
function that is continuous on $\R^d \setminus \{ \0 \}$,
with image contained in the set of 
symmetric, non-negative definite  matrices in 
$\R^d\otimes\R^d$.
Suppose that the $\cC_d$ martingale problem for 
$(G,v)$
is well-posed,
where
$Gf := \frac{1}{2} \sum a_{ij} \partial_i \partial_j f$ (for a smooth $f:\R^d\to\R$ with compact support)  
and a distribution $v$ on $\R^d$. 
For $n \in \N$, let $Z_n$ be a process with sample paths in $\cD_d$ and let
$A_n = (A_n^{ij} )$ be a symmetric 
$\R^d\otimes\R^d$-valued process
started at zero,
such that $A_n^{ij}$ has sample paths in $\cD_1$ and $A_n (t) - A_n (s)$
is non-negative definite for all $t > s \geq 0$. Set $\cF_t^n: = \sigma ( Z_n (s), A_n (s), s \leq t )$.
Suppose that $Z_n^i$ 
and 
$Z_n^i Z_n^j - A_n^{ij}$
are $\cF_t^n$-adapted local martingales
for each  $i, j \in \{1,\ldots, d\}$. 
Let $\tau_n^r := \inf \{ t \geq 0 : \| Z_n (t) \| \geq r \text{ or } \| Z_n (t-) \| \geq r \}$
(with convention $\inf \emptyset:=\infty$)
and 
suppose that for every $r >0$, $T>0$, and $i,j \in \{1,\ldots,d\}$,
\begin{align}
\label{ip1}
\lim_{n \to \infty} \Exp \left[ \sup_{0 \leq t \leq T \wedge \tau_n^r} \left\| Z_n (t) - Z_n (t-) \right\|^2 \right] & = 0 ;\\
\label{ip2}
\lim_{n \to \infty} \Exp \left[ \sup_{0 \leq t \leq T \wedge \tau_n^r} \left| A^{ij}_n (t) - A^{ij}_n (t-) \right| \right] & = 0 ;
\end{align}
and, as $n \to \infty$,
\begin{equation}
\label{ip3}
\sup_{0 \leq t \leq T \wedge \tau_n^r} \left| A_n^{ij}(t) - \int_0^t a_{ij} ( Z_n (s) ) \ud s \right| \toP 0,
\end{equation}
where $\toP$ denotes convergence in probability and $t\wedge s=\min\{r,s\}$ for $s,t\in[0,\infty]$. 
Assume $\sup_{n\in\N} \Exp \|Z_n(0)\|^2 <\infty$. 
Suppose that 
$Z_n (0)$ 
and 
$\| Z_n \|$
converge weakly to a probability law $v$ on $\R^d$
and the law of a Bessel process of dimension greater than one, respectively.
Then $Z_n$ converges weakly to the solution of the martingale problem for $(G,v)$. 
\end{theorem}

The underlying idea for the proof of Theorem~\ref{thm:conditional_invariance}.
is standard: show that every subsequence 
of $(Z_n)_{n\in\N}$ has a further subsequence converging weakly to the law given by the solution of the martingale
problem $(G,v)$ (cf. proof of~\cite[Thm~7.4.1, p.~354]{ek}). 
Since $a$ in Theorem~\ref{thm:conditional_invariance} is bounded, $a_i := \sup_{x\in\R^d} a_{ii} (x)$ is finite for each $i\in \{1,\ldots, d\}$.
Since 
$A_n^{ii}(t)\geq A_n^{ii}(t-)$ for all $t\geq0$ and $i\in\{1,\ldots,d\}$,
$$\eta_n := \inf \left\{ t \geq 0 : \max_{1\leq i\leq d}\{A_n^{ii}(t)-a_i t\} \geq  1 \right\}$$
is an
$(\cF_t^n)$-stopping time. 
Since 
$\eta_n\geq \inf \{ t \geq 0 :\max_{1\leq i\leq d}| A_n^{ii}(t) - \int_0^t a_{ii} ( Z_n (s) ) \ud s | \geq 1\}$
and~\eqref{ip3} holds for any $T,r>0$,
we have that 
\begin{equation}
\label{inv:eta} 
\eta_n \toP \infty\qquad\text{as $n\to\infty$.}
\end{equation}
Define for given $r>0$, $n\in\N$ and $i, j \in \{1,\ldots, d\}$ 
the processes $\tilde Z_n^r$ and 
$\tilde A_n^{ij}$
by
\begin{equation}
\label{eq:tilda_procs}
\tilde Z_n^r (t) := Z_n ( t \wedge \eta_n \wedge \tau_n^r )\qquad \text{and}\qquad
\tilde A_n^{ij} (t ) := A_n^{ij} ( t \wedge \eta_n \wedge \tau_n^r ) ,
\end{equation}
respectively ($\tilde A_n^{ij}$ depends on $r$ but this is suppressed from the notation 
as it is clear from the context). 
Observe that for any $T>0$ and 
$(\cF_t^n)$-stopping time $\tau$
less than $T$, the modulus of any component of 
$\tilde Z_n^r (\tau)-\tilde Z_n^r(0)$
is bounded above by an integrable random variable:
\begin{equation}
\label{eq:Compact_bound}
\|\tilde Z_n^r(\tau)-\tilde Z_n^r(0)\| \leq 2r+\sup_{0 \leq t \leq T \wedge \tau_n^r} \left\| Z_n (t) - Z_n (t-) \right\|.
\end{equation}
Since 
$\tilde Z_n^r (0)=Z_n(0)$
is integrable by assumption,
the local martingale $\tilde Z_n^r$ is of class (DL) and therefore a martingale~\cite[Ch.~IV, Prop.~1.7]{ry}.
An analogous argument, relying on~\eqref{ip1}--\eqref{ip2}, the inequality
$|\tilde Z_n^{r,i} \tilde Z_n^{r,j}|\leq  (\tilde Z_n^{r,i})^2 +(\tilde Z_n^{r,j})^2$ and the square integrability
of $\|Z_n(0)\|$, shows that 
$\tilde Z_n^{r,i} \tilde Z_n^{r,j} - \tilde A_n^{ij}$
is also a martingale. 
Furthermore, since $A_n^{ii}(0)=0$ for all indices $i\in \{1,\ldots, d\}$, for any $t\geq0$ 
we have
\begin{equation}
\label{eq:inv1}
\tilde A_n^{ii} ( t ) \leq a_i t + 1 + \sup_{0 \leq s \leq t \wedge \tau_n^r} \left( A_n^{ii} (s) - A_n^{ii} (s-) \right) .
\end{equation}


\begin{lemma}
\label{lem:ip1}
For each $r >0$, the sequence of the laws of processes $(\tilde Z^r_n)_{n\in\N}$ 
on $\cD_d$ 
is relatively compact
in the metric space of 
all probability measures on $\cD_d$ with the Prohorov metric.\footnote{See~\cite[\S~3.1, p.~96]{ek} for the definition 
and properties of the Prohorov metric on the set of probability measures defined on a Borel $\sigma$-algebra on a metric space. In this context we 
use the Skorohod metric $d$ on $\cD_d$, cf.~\cite[\S~3.5, p.~116]{ek}. The induced topology is the one of weak convergence of 
probability measures~\cite[Thm~3.3.1, p.~108]{ek}.}
\end{lemma}

\begin{proof}
We prove the lemma by establishing the sufficient condition 
for the relative  compactness of the sequence $(\tilde Z^r_n)_{n\in\N}$ 
given in~\cite[Thm~3.8.6, pp.~137--138]{ek}.
Fix an arbitrary $T>0$  and
let $B_K$ denote a closed ball of radius $K>2r+1$ in $\R^d$. 
Note that the bound in~\eqref{eq:Compact_bound}
and the Markov inequality 
imply
\begin{align*}
\Pr\left[\tilde Z_n^r(t)\in B_K\> \text{ for all $t\in[0,T]$} \right]
\geq &
\Pr\left[2r +\|Z_n(0)\| + \sup_{0 \leq t \leq T \wedge \tau_n^r} \left\| Z_n (t) - Z_n (t-) \right\| \leq K \right] \\ 
\geq & 1 - \frac{C_0}{K-2r}\qquad \text{for all $n\in\N$,}
\end{align*}
where $C_0>0$ 
depends on the quantities
$\sup_{n \in \N} \Exp \left[ \sup_{0 \leq t \leq T \wedge \tau_n^r} \left\| Z_n (t) - Z_n (t-) \right\|^2 \right]$ 
and
$\sup_{n\in\N} \Exp \|Z_n(0)\|^2$,
which are finite by assumption. 
As $K$ is independent of $n$ and can be arbitrarily large, the compact containment condition~\cite[Eq.~(7.9), p.~129]{ek}
holds
for $(\tilde Z_n^r)_{n\in\N}$. Hence condition~(a)
of~\cite[Thm~3.7.2]{ek}, 
also assumed in~\cite[Thm~3.8.6, pp.~137--138]{ek}, 
holds. 

Since 
$\tilde Z_n^{r,i}$
and
$(\tilde Z_n^{r,i})^2- \tilde A_n^{ii}$ are martingales for all 
$i\in \{1,\ldots, d\}$,
it holds that 
$$
\Exp \left[  \left\| \tilde Z^r_n (t+h) - \tilde Z^r_n (t) \right\|^2 \Big\vert \cF_t^n \right]  = 
\Exp \left[\sum_{i=1}^d \left( \tilde A^{ii}_n ( t + h) - \tilde A_n^{ii} (t) \right) \Big\vert \cF_t^n \right]    
$$
for any $t,h\geq0$.
With this in mind, 
define
\[ \gamma_n (\delta) := \sup_{0 \leq t \leq T\wedge \tau_n^r} \sum_{i=1}^d \left( \tilde A^{ii}_n ( t + \delta) - \tilde A_n^{ii} (t) \right) \]
for any $\delta >0$.
In order to compare $\gamma_n(\delta)$ with the corresponding quantity for the limiting process, let
$$
\Gamma_n(\delta):=\gamma_n (\delta) - \sup_{t\in[0,T\wedge \tau_n^r]} \sum_{i=1}^d \int_t^{t+\delta} a_{ii} ( \tilde Z^r_n (s) ) \ud s. 
$$
Now we have from~\eqref{ip3} that
\[ \sup_{0 \leq t \leq T\wedge \tau_n^r} \left| \tilde A_n^{ii} (t+\delta) - \int_0^{t+\delta} a_{ii} ( \tilde Z^r_n (s) ) \ud s \right| 
\text{ and }
\sup_{0 \leq t \leq T\wedge \tau_n^r} \left| \tilde A_n^{ii} (t) - \int_0^{t} a_{ii} ( \tilde Z^r_n (s) ) \ud s \right| \]
both tend to zero in probability, implying that
$\Gamma_n(\delta)$ also tends to zero in probability:
\begin{equation}
\label{eq:Gamma_prob_lim_zero}
\left|\Gamma_n(\delta)\right| 
\leq \sup_{t\in[0 ,T\wedge \tau_n^r]} \sum_{i=1}^d \left| \tilde A_n^{ii} (t+\delta) - \tilde A_n^{ii} (t) - \int_t^{t+\delta} a_{ii} ( \tilde Z^r_n (s) ) \ud s \right| \toP 0.
\end{equation}
Since the upper bound in~\eqref{eq:inv1} is non-decreasing in $t$, we get 
\begin{equation*}
\left|\Gamma_n(\delta)\right| \leq \sum_{i=1}^d \left( 3a_i (T+\delta) + 2 + 2\sup_{s \in[0, (T+\delta)  \wedge \tau_n^r]} \left(A_n^{ii} (s) - A_n^{ii} (s-)\right) \right).
\end{equation*}
By~\eqref{ip2}
the right-hand side of this inequality converges in $L^1$ as $n\to\infty$. 
Thus the sequence $(\Gamma_n(\delta))_{n\in\N}$ must be uniformly integrable 
and hence by~\eqref{eq:Gamma_prob_lim_zero} converges to zero in $L^1$.
By adding and subtracting the relevant term we find
\begin{align*}
\limsup_{n \to \infty} \Exp \gamma_n (\delta)  \leq \limsup_{n\to\infty} \Exp \left|\Gamma_n(\delta)\right| 
+ \limsup_{n\to\infty} \Exp \sup_{t\in[0 ,T\wedge \tau_n^r]}  \sum_{i=1}^d \int_t^{t+\delta} a_{ii} ( \tilde Z^r_n (s) ) \ud s 
\leq \delta \sum_{i=1}^d  a_i.
\end{align*}
Hence it clearly holds that
$\lim_{\delta \to 0}  \limsup_{n \to \infty} \Exp \gamma_n (\delta) = 0$
and the
relative  compactness of $\tilde Z^r_n$ now follows from~\cite[Thm~3.8.6, p.~137--138]{ek}
(see also~\cite[Remark~8.7(b), p.~138]{ek}).
\end{proof}

For any path $x\in\cD_d$, we define the time $\tau^r(x)$ of its first contact with the complement of the open ball of radius $r$
in $\R^d$ (centred at the origin) by
\begin{equation}
\label{eq:contact_time_def}
\tau^r(x):=\inf\{t\geq0\>:\> \|x(t)\|\geq r \quad \text{or} \quad \|x(t-)\|\geq r\},
\end{equation}
where $\inf\emptyset = \infty $.
If it is clear from the context which path $x$ we are considering,
to simplify the notation we sometimes write $\tau^r$ for $\tau^r(x)$.
Note that if $x$ is continuous, then 
$\tau^r(x)=\inf\{t\geq0: \|x(t)\|\geq r\}$.
The following lemma is important in 
the proof of Theorem~\ref{thm:conditional_invariance}.

\begin{lemma}
\label{lem:first_hitting_time}
Let $\Pr$ be a probability measure on $\cD_d$. Then 
the complement in $\RP$ of the set 
$\{r\in\RP: \Pr[\lim_{s\to r}\tau^s=\tau^r]=1\}$
is at most countable, with $\tau^r$ defined in~\eqref{eq:contact_time_def}. 
\end{lemma}

To prove Lemma~\ref{lem:first_hitting_time} 
we first need to establish properties of the function 
$r\mapsto \tau^r$.

\begin{lemma}
\label{lem:hitting_time} Fix $x\in\cD_d$. The function  
$r\mapsto \tau^r(x)$, mapping $\RP$ into $[0,\infty]$, is non-decreasing, has right limits and is left continuous. 
Put differently, for any $r\in\RP$ 
the limit $\lim_{s\downarrow r} \tau^s=:\tau^{r+}$ exists in $[0,\infty]$
and, for $r>0$, 
it holds that $\lim_{s\uparrow r} \tau^s=\tau^r$. 
Furthermore, for any 
$r\in\RP$ the following hold: 
\begin{enumerate}
\item[(i)] if $\tau^r=\infty$ then $\lim_{s\to r}\tau^s=\tau^r$;
\item[(ii)] if $\tau^r<\infty$ then for any $\varepsilon>0$ there are at most finitely many $s\in[0,r]$ such that  $\tau^{s+}>\tau^s+\varepsilon$. 
\end{enumerate}
\end{lemma}

\begin{remark}
The topology on $[0,\infty]$ is that of the one-point compactification of $\RP$.
If $\tau^r(x)=\infty$, then the function 
$s\mapsto \tau^s(x)$ defined on $[0,r]$ may have an infinite number of jumps greater than any given positive constant.
If $\tau^r(x)<\infty$, then the inequality $\tau^{r+}(x)>\tau^r(x)$ may hold invalidating the limit in Lemma~\ref{lem:hitting_time}(i).
\end{remark}

\begin{proof}[Proof of Lemma~\ref{lem:hitting_time}]
It is clear from definition~\eqref{eq:contact_time_def} that 
$\tau^s\leq\tau^r$ for any
$0\leq s\leq r$.
Hence, for any $r\in\RP$ the limit 
$\tau^{r+}$
exists in $[0,\infty]$. 
Now fix
$r>0$.
The monotonicity  implies that there exists 
$\alpha:=\lim_{s\uparrow r} \tau^s\in[0,\infty]$
satisfying
$\alpha\leq\tau^r$. 
Assuming 
$\alpha<\tau^r$, 
for any
$\beta\in(\alpha,\tau^r)$
it holds that $\tau^s<\beta$ for all $s\in[0,r)$.
For any sequence $(s_k)_{k\in\N}$ in $[0,r)$,
such that $s_k\uparrow r$, by~\eqref{eq:contact_time_def} there exists a sequence 
$(t_k)_{k\in\N}$ such that 
$t_k\in(\tau^{s_k},\beta)$ and
\begin{equation}
\label{eq:upper_cadlag_bound}
\max\{\|x(t_k)\|,\|x(t_k-)\|\}\geq s_k\qquad\text{for all $k\in\N$.}
\end{equation}
By passing to a subsequence (again denoted by 
$(t_k)_{k\in\N}$),
we may assume that the limit
$\beta':=\lim_{k\to\infty} t_k$ exists in
$[0,\beta]$.
Moreover, by passing to a further subsequence, we may assume that
$(t_k)_{k\in\N}$ is monotonic, i.e.  
either 
$t_k\uparrow\beta'$ 
or
$t_k\downarrow\beta'$. 
Since $x$ is right continuous with left limits, 
in the case 
$t_k\uparrow\beta'$ 
we find 
$\|x(\beta'-)\|=
\lim_{k\to\infty}\|x(t_k)\|=
\lim_{k\to\infty}\|x(t_k-)\|$.
Hence~\eqref{eq:upper_cadlag_bound} yields 
$$
\|x(\beta'-)\|= 
\lim_{k\to\infty}
\max\{\|x(t_k)\|,\|x(t_k-)\|\}\geq 
\lim_{k\to\infty} s_k=r.
$$
Similarly,
if 
$t_k\downarrow\beta'$ 
we get
$
\lim_{k\to\infty}\|x(t_k)\|=
\lim_{k\to\infty}\|x(t_k-)\|=\|x(\beta')\|\geq r.
$
Hence the assumption 
$\alpha<\tau^r$
implies 
$\max\{\|x(\beta')\|,\|x(\beta'-)\|\}\geq r$
for some $\beta'\leq \beta<\tau^r$,
which is a contradiction. Therefore $\alpha=\tau^r$ and the left continuity follows. 
Note that this argument does not require $\tau^r<\infty$.

It follows from the left continuity and monotonicity that $\tau^r=\infty$ implies the limit in~(i).
Assume 
$\tau^r<\infty$ and
pick $\varepsilon>0$. 
The intervals in the family 
$\{[\tau^s,\tau^{s+}):s\in[0,r]\}$
are disjoint and contained in the bounded interval
$[0,\tau^r]$.
Hence there can only be finitely many $s\in[0,r]$ satisfying 
the condition in~(ii). 
\end{proof}

\begin{proof}[Proof of Lemma~\ref{lem:first_hitting_time}]
Let $A_{\varepsilon,\delta}^r:=\{s\in[0,r]: \Pr[\tau^{s+}>\tau^s+\varepsilon]\geq \delta\}$
for arbitrary $\varepsilon,\delta>0$, $r\in\RP$. 

\noindent \textbf{Claim.} $A_{\varepsilon,\delta}^r$ is at most countable. 
\smallskip 

Note first  that the Claim implies the lemma. 
By Lemma~\ref{lem:hitting_time}, the following equivalence holds for any $r\in\RP$:  
$\lim_{s\to r}\tau^s=\tau^r \iff \tau^{r+}=\tau^r$. Hence it suffices to show the set 
$$\{r\in\RP: \Pr[\tau^{r+}>\tau^r]>0\} = \cup_{n=1}^\infty\cup_{k=1}^\infty\cup_{i=1}^\infty A_{\varepsilon_k,\delta_i}^{s_n}$$
is at most countable, which clearly holds by the claim, where 
$(\varepsilon_k)_{k\in\N}$, $(\delta_i)_{i\in\N}$ and $(s_n)_{n\in\N}$ are 
monotone sequences satisfying
$\varepsilon_k\downarrow 0$, $\delta_i\downarrow 0$ and $s_n\uparrow \infty$.

\medskip

\noindent \textit{Proof of Claim.} Assume that 
$A_{\varepsilon,\delta}^r$
is uncountable
and let 
$I$ 
be  the set of its isolated points
(i.e. $x\in I$ if and only if $x\in A_{\varepsilon,\delta}^r$ and there exists a neighbourhood $U$ of $x$ in $\RP$ such that  
$\{x\}= U\cap A_{\varepsilon,\delta}^r$). 
Then 
$I$ is at most countable. To see this, note that for each $x\in I$ there exists  a rational number $q_x\leq x$, such that
$[q_x,x)\cap A_{\varepsilon,\delta}^r=\emptyset$ (for $x\in I\cap\Q$ we may take $q_x:=x$). For any distinct points $x,y\in I$, it clearly holds 
$q_x\neq q_y$. Hence the cardinality of  $I$  is at most that of $\Q$ and the uncountable set $A_{\varepsilon,\delta}^r\setminus I$
has no isolated points. 

Consider $r_1:=\sup\{y\in A_{\varepsilon,\delta}^r\setminus I\}\leq r$. 
There exists a strictly increasing sequence $(p^1_i)_{i\in\N}$ in $A_{\varepsilon,\delta}^r\setminus I$
with limit $p^1_i\uparrow r_1$. It is also clear that any
$x\in\{\tau^{p^1_i +}>\tau^{p^1_i}+\varepsilon\}\subset \cD_d$ satisfies 
$\tau^{p^1_i}(x)<\infty$.
Hence the event 
$
B^{r_1}:=\{\tau^{p^1_i +}>\tau^{p^1_i}+\varepsilon\} \> \text{i.o.}
$
satisfies: 
$\Pr[B^{r_1}]\geq\delta$
and, for each path $x\in B^{r_1}$, the function $s\mapsto \tau^s(x)$ 
has infinitely many jumps of size at least $\varepsilon$ on the interval $[0,r_1]$. 
Furthermore, since these jumps occur along a subsequence of 
$(p^1_i)_{i\in\N}$,
Lemma~\ref{lem:hitting_time} implies 
for any $x\in B^{r_1}$
that 
$\tau^s(x)<\infty$ for all $s\in[0,r_1)$ and $\tau^{r_1}(x)=\infty$. 

Since $(A_{\varepsilon,\delta}^r\setminus I)\subseteq [0,r_1]$, 
it holds that 
$(A_{\varepsilon,\delta}^r\setminus I)\subseteq A_{\varepsilon,\delta}^{r_1}$
making 
$A_{\varepsilon,\delta}^{r_1}$
uncountable. 
Furthermore, since 
$A_{\varepsilon,\delta}^{r_1}\setminus\{r_1\}=\cup_{s<r_1}A_{\varepsilon,\delta}^{s}$, 
there exists $r'<r_1$ such that 
$A_{\varepsilon,\delta}^{r'}$
is uncountable. 
We can now repeat the construction above, 
with 
$A_{\varepsilon,\delta}^{r}$
substituted by
$A_{\varepsilon,\delta}^{r'}$,
to define the event 
$B^{r_2}$
(for some
$r_2\in (0, r']$)
with properties analogous to those of 
$B^{r_1}$.
In particular 
$\Pr[B^{r_2}]\geq\delta$
and,
since each 
$x\in B^{r_2}$
satisfies 
$\tau^{r_2}(x)=\infty$,
it must hold 
$B^{r_1}\cap B^{r_2}=\emptyset$.
As before, there exists 
$r''< r_2$
such that 
$A_{\varepsilon,\delta}^{r''}$
is uncountable. 
By the same construction there exists 
$r_3\in(0, r'']$ 
and an event 
$B^{r_3}$
satisfying 
$\Pr[B^{r_3}]\geq\delta$
and 
$B^{r_3}\cap(B^{r_1}\cup B^{r_2})=\emptyset$,
since 
$x\in B^{r_3}$
satisfies 
$\tau^{r_3}(x)=\infty$
while for any 
$x\in B^{r_1}\cup B^{r_2}$
we have 
$\tau^{r_3}(x)<\infty$.
We can thus inductively construct a sequence of pairwise disjoint events $(B^{r_n})_{n\in\N}$ in $\cD_d$
each of which has probability at least $\delta>0$. This contradicts the fact that 
the total mass of $\Pr$ is equal to one. 
\end{proof}

\begin{remark}
The proof of the Claim, contained in the proof of Lemma~\ref{lem:first_hitting_time},
shows that 
$A_{\varepsilon,\delta}^r$
is in fact locally finite.
\end{remark}

In order to apply Lemma~\ref{lem:first_hitting_time}
in the proof of Theorem~\ref{thm:conditional_invariance}, we need another fact about 
the metric space $(\cD_d,d)$, where the metric $d:\cD_d\times\cD_d\to\RP$ that induces the Skorohod topology
is defined in~\cite[Eq.~(5.2), p.~117]{ek} 
(see also~\cite[\S~3.5]{ek}). 

\begin{lemma}
\label{lem:continuity_of_mapa}
Pick $r>0$.
Assume that $x\in\cD_d$ satisfies $\lim_{s\to r}\tau^s(x)=\tau^r(x)$ 
(see~\eqref{eq:contact_time_def} for definition of $\tau^r(x)$). 
Then the function $\cD_d\to[0,\infty]$,  given by 
$y\mapsto \tau^r(y)$, is continuous at $x$. 
If in addition it holds that either 
$x(\tau^r(x)-)<r$ 
or 
$x(\tau^r(x))\leq r$, then the map
$\cD_d\to\cD_d$, given by $y\mapsto y(\cdot\wedge \tau^r(y))$, 
is continuous at $x$. 
\end{lemma}

\begin{remark}
\begin{enumerate}
\item[(a)] The lemma implies that if 
$x\in\cC_d$ 
satisfies 
$\lim_{s\to r}\tau^s(x)=\tau^r(x)$, 
the map $\cD_d\to\cD_d\times[0,\infty]$, given by $y\mapsto (y(\cdot\wedge \tau^r(y)),\tau^r(y))$, 
is continuous at $x$. 
\item[(b)] It is easy to construct $x\in\cC_d$, such that both 
$y\mapsto \tau^r(y)$ 
and
$y\mapsto y(\cdot\wedge \tau^r(y))$ are discontinuous at $x$. 
The key feature of such a function $x$ is that 
$\tau^{r+}(x)>\tau^r(x)$ (see Lemma~\ref{lem:hitting_time} for the definition of  $\tau^{r+}(x)$).
\item[(c)] If $x\in\cD_d\setminus\cC_d$, then the additional assumption in the lemma is necessary for the
continuity of $y\mapsto y(\cdot\wedge \tau^r(y))$ to hold at $x$. 
To see this, 
for any $r>0$ and $\eps\in[0,1)$, 
consider
$x_\eps(t):=(t+\eps) \2{0\leq t< r}+(r+1)\2{r\leq t<\infty}$.
Then $x_0$ clearly satisfies the first assumption in the lemma but not the
second one.  Note that for any $\eps\in(0,1)$ 
we have
$d(x_0,x_\eps)\leq \eps$ 
and
$|x_0(t\wedge\tau^r(x_0))-x_\eps(t\wedge\tau^r(x_\eps))|\geq \2{r\leq t<\infty}$.
\end{enumerate}
\end{remark}

\begin{proof}
Let $x\in\cD_d$ satisfy 
$\lim_{s\to r}\tau^s(x)=\tau^r(x)$.
We first prove that for any sequence $(x_n)_{n\in\N}$ in $\cD_d$, such that $d(x_n,x)\to0$, it holds that 
$\tau^r(x_n)\to\tau^r(x)$.
Note that $d(x_n,x)\to0$ 
and 
the definition of $d$ in~\cite[Eq.~(5.2), p.~117]{ek}
imply that there exists a sequence $(\lambda_n)_{n\in\N}$ of  strictly increasing, Lipschitz continuous, surjective functions 
$\lambda_n:\RP\to\RP$ satisfying
\begin{equation}
\label{eq:Skor_metric}
\sup\{\|x_n(\lambda_n(t))-x(t)\|, |\lambda_n(t)-t|:t\in[0,T]\}\to0
\qquad\text{for any $T>0$.}
\end{equation}
If $\tau^r(x)=\infty$, then 
for any $T>0$ $\exists\delta>0$ such that  
$\sup_{t\in[0,T]}\{\|x(t)\|,\|x(t-)\|\}<r-\delta$.
By~\eqref{eq:Skor_metric}, 
for all sufficiently large $n\in\N$ 
we have 
$\sup_{s\in[0,\lambda_n(T)]}\{\|x_n(s)\|\}<r-\delta/2$,
implying
$\tau^r(x_n)\geq T-1$. Since $T$ was arbitrary, it holds that $\tau^r(x_n)\to\infty$.

Assume now that 
$\tau^r(x)<\infty$
and that 
$(\tau^r(x_n))_{n\in\N}$ does not converge to 
$\tau^r(x)$.
By passing to a subsequence (again denoted by $(x_n)_{n\in\N}$), we may assume that
$\exists \varepsilon>0$ such that 
$|\tau^r(x_n)-\tau^r(x)|>\varepsilon$
for all 
$n\in\N$.
Pick $T>\tau^r(x)+\varepsilon$  and note that without loss of generality 
we may assume (for all $n\in\N$)
that 
either
$\tau^r(x_n)>\tau^r(x)+\varepsilon$
or
$\tau^r(x_n)<\tau^r(x)-\varepsilon$.
Consider first the former case. 
By Lemma~\ref{lem:hitting_time}, 
our assumption is equivalent to 
$\tau^{r+}(x)=\tau^r(x)$.
Hence $\exists \delta>0$ and an interval 
$[t_0,s_0]$ contained in $(\tau^r(x),\tau^r(x)+\varepsilon)$, such that 
$\inf_{t\in[t_0,s_0]}\|x(t)\|>r+\delta$.
As 
$[t_0,s_0]\subset[0,T]$, by~\eqref{eq:Skor_metric}
there exists $n\in\N$
and 
$t\in(t_0,s_0)$ such that 
$\lambda_n(t)<s_0$
and
$\|x_n(\lambda_n(t))\|\geq \|x(t)\| - \|x(t)-x_n(\lambda_n(t))\|>r+\delta/2$, 
contradicting
$\tau^r(x_n)>\tau^r(x)+\varepsilon>\lambda_n(t)$.

Consider now the case 
$\tau^r(x_n)<\tau^r(x)-\varepsilon$
for all $n\in\N$.
Then for a sequence $\delta_n\downarrow0$ we have 
$\sup_{s\in[0,\tau^r(x)-\varepsilon)}\|x_n(s)\|>r-\delta_n$.
Hence there exists a sequence $(t_n)_{n\in\N}$
in
$(0,\tau^r(x)-\varepsilon)$ 
such that 
$\|x_n(t_n)\|\to r$.
By~\eqref{eq:Skor_metric}
it holds that 
$\lambda_n^{-1}(t_n)<\tau^r(x)-\varepsilon/2$ for all sufficiently large 
(and thus wlog all) $n\in\N$.
Furthermore, the triangle inequality and~\eqref{eq:Skor_metric}
imply 
$| \|x(\lambda_n^{-1}(t_n))\|-r|\leq \|x(\lambda_n^{-1}(t_n))- x_n(t_n) \| + | \|x_n(t_n)\|- r|\to 0$,
since 
$\lambda_n^{-1}(t_n), t_n\in[0,T]$ for all $n\in\N$.
By passing to a convergent subsequence, there exists 
$\alpha\leq \tau^r(x)-\varepsilon/2$ such that
either 
$\lambda_n^{-1}(t_n)\uparrow \alpha$
or 
$\lambda_n^{-1}(t_n)\downarrow \alpha$.
Hence we either get
$\|x(\alpha-)\|=r$ 
or
$\|x(\alpha)\|=r$,
contradicting the fact that 
$\alpha<\tau^r(x)$.
This implies the continuity of the map
$y\mapsto \tau^r(y)$ at $x$. 

Consider the map
$y\mapsto y(\cdot\wedge \tau^r(y))$
in the case
$\tau^r(x)=\infty$.
Then 
$x(\cdot\wedge\tau^r(x))=x$ 
and, as we have already established, 
$\tau^r(x_n)\to\infty$.
By the definition of the metric $d$ (see~\cite[Eq.~(5.2), p.~117]{ek}), 
we have
$d(x_n(\cdot\wedge \tau^r(x_n)),x(\cdot\wedge\tau^r(x)))
\leq 
d(x_n,x)+ 
d(x_n,x_n(\cdot\wedge \tau^r(x_n)))
\leq
d(x_n,x)+ e^{-\tau^r(x_n)}\to 0$.

In the case 
$\tau^r(x)<\infty$,
we have already seen that 
$\tau^r(x_n)\to\tau^r(x)$.
By definition~\cite[Eq.~(5.2), p.~117]{ek},
for any $y\in\cD_d$, $t\in\RP$ and a sequence $(t_n)_{n\in\N}$ converging to $t$ we have
$$d(y(\cdot\wedge t_n),y(\cdot\wedge t))
\leq \|y(t)-y(t_n)\| + |t-t_n| \sup_{s\in[0,t+1]}\|y(s)\|$$
for all large $n\in\N$.
Recall that $y$ is bounded on compact intervals.
Hence 
if either $t_n\downarrow t$ or $t_n\to t$ and $y$ is continuous at $t$,
then
$d(y(\cdot\wedge t_n),y(\cdot\wedge t))\to0$.\footnote{Note that if $t_n\uparrow t$,
$d(y(\cdot\wedge t_n),y(\cdot\wedge t))$
may be bounded from below by a positive constant $\forall n\in\N$.}
Therefore the estimate 
\begin{align*}
d(x_n(\cdot\wedge \tau^r(x_n)),x(\cdot\wedge\tau^r(x)))
\leq &
d(x_n(\cdot\wedge \tau^r(x_n)),x(\cdot\wedge\tau^r(x_n))) \\
& +
d(x(\cdot\wedge \tau^r(x_n)),x(\cdot\wedge\tau^r(x))) \\
\leq &
d(x_n,x)
+
d(x(\cdot\wedge \tau^r(x_n)),x(\cdot\wedge\tau^r(x))) 
\end{align*}
implies the lemma, except when 
$\tau^r(x_n)\uparrow\tau^r(x)$
and 
$x(\tau^r(x)-)\neq x(\tau^r(x))$.

Assuming 
$\tau^r(x_n)\uparrow\tau^r(x)<\infty$
and 
$x(\tau^r(x)-)\neq x(\tau^r(x))$,
by 
$\lim_{s\to r}\tau^s(x)=\tau^r(x)$
it holds that
$x(\tau^r(x)-)<x(\tau^r(x))$. 
Furthermore, since by assumption it either holds that
$x(\tau^r(x)-)<r$ of $x(\tau^r(x))\leq r$, 
we must have 
$x(\tau^r(x)-)<r$.
Hence there exists $\delta>0$
such that 
$\sup_{t\in[0,\tau^r(x))} \|x(t)\|<r-\delta$.
Therefore by~\eqref{eq:Skor_metric} $\exists N\in\N$ such that for all $n\geq N$ and $t\in[0,\tau^r(x))$
we have 
$
\|x_n(\lambda_n(t))\|
\leq 
\|x(t)\|+ \|x_n(\lambda_n(t))-x(t)\|< r-\delta/2.
$
Thus we obtain
$\lambda_n(\tau^r(x))\leq \tau^r(x_n)$
for all $n\geq N$.
As 
$\lambda_n$ is increasing,
for every 
$t\in[0,\tau^r(x)]$
it holds that 
$
\|x_n(\lambda_n(t)\wedge\tau^r(x_n))-x(t\wedge \tau^r(x))\|
=
\|x_n(\lambda_n(t))-x(t)\|$.
Furthermore, 
since 
$\tau^r(x_n)\in[\lambda_n(\tau^r(x)), \tau^r(x)]$, 
for all $t\in(\tau^r(x),\lambda_n^{-1}(\tau^r(x_n))]$
we have
\begin{align*}
\|x_n(\lambda_n(t)\wedge\tau^r(x_n))-x(t\wedge \tau^r(x))\| 
& =  
\|x_n(\lambda_n(t))-x(\tau^r(x))\| \\
& \leq \|x(t)-x(\tau^r(x))\|  +
\|x_n(\lambda_n(t))-x(t)\|.
\end{align*}
Hence, for any $T>\tau^r(x)$, 
it holds that
\begin{multline*}
\sup_{t\in[0,T]}\|x_n(\lambda_n(t)\wedge\tau^r(x_n))-x(t\wedge \tau^r(x))\| \\
=
\sup_{t\in[0,\tau^r(x)]}\|x_n(\lambda_n(t))-x(t)\|  +
\sup_{t\in(\tau^r(x),T\wedge\lambda^{-1}_n(\tau^r(x_n))]}\|x_n(\lambda_n(t))-x(\tau^r(x))\| \\
\leq \sup_{t\in[0,T]}\|x_n(\lambda_n(t))-x(t)\|+ \sup_{t\in(\tau^r(x),\lambda_n^{-1}(\tau^r(x))]}\|x(t)-x(\tau^r(x))\|, 
\end{multline*}
where the inequality  uses 
the assumption 
$\tau^r(x_n)\leq\tau^r(x)$.
The first summand in the bound tends to zero by~\eqref{eq:Skor_metric} and the second 
by the right  continuity of $x$ and $\lambda_n^{-1}(\tau^r(x))\to\tau^r(x)$. 
Hence 
$d(x_n(\cdot\wedge\tau^r(x_n)),x(\cdot\wedge \tau^r(x)) )\to0$
by~\cite[Prop.~3.5.3, p.~119]{ek}
and the lemma follows. 
\end{proof}

The next task in the proof of Theorem~\ref{thm:conditional_invariance}
is to construct a limiting process.

\begin{lemma}
\label{lem:a_integral}
Fix $r_0 > 0$.
There exists a process $Z^{r_0}$ with paths a.s. in $\cC_d$, such that 
for all but countably many $r \in(0,r_0)$
it holds that 
\begin{equation}
\label{eqn:joint_convergence} 
( Z_{n_k} ( \, \cdot \, \wedge \tau_{n_k}^r ) ,  \tau_{n_k}^r ) \Rightarrow (Z^{r_0} ( \, \cdot \, \wedge \tau^r ) , \tau^r ), 
\end{equation}
where 
$\tau^r_n=\tau^r(Z_n)$
is given in Theorem~\ref{thm:conditional_invariance}, 
$\tau^r=\tau^r(Z^{r_0})$
is defined in~\eqref{eq:contact_time_def} and 
$\Rightarrow$ denotes the weak convergence of probability measures on $\cD_d\times[0,\infty]$.
Furthermore, 
the law of $\| Z^{r_0} ( \, \cdot \, \wedge \tau^r ) \|$ equals that of a Bessel process (of dimension greater than one) stopped at level $r$.
In particular
it holds  that
$(Z^{r_0} ( \, \cdot \, \wedge \tau^r ) , \tau^r )\in \cD_d\times\RP$ a.s.
\end{lemma}

\begin{proof}
Lemma~\ref{lem:ip1}  implies the existence of 
a convergent subsequence 
$(\tilde Z^{r_0}_{n_k})_{k\in\N}$
of the sequence $(\tilde Z^{r_0}_n)_{n\in\N}$ defined in~\eqref{eq:tilda_procs}. Denote its limit  by $Z^{r_0}$.
By~\eqref{inv:eta} and the definition of the metric $d:\cD_d\times\cD_d\to\RP$ in~\cite[Eq.~(5.2), p.~117]{ek}, which induces the Skorohod topology, 
it holds that 
$$d(\tilde Z^{r_0}_{n_k},Z_{n_k} ( \, \cdot \, \wedge \tau_{n_k}^{r_0}))\leq e^{-\eta_{n_k}}\toP0\qquad \text{as $k\to\infty$.}$$
It hence follows that the sequence
$(Z_{n_k} ( \, \cdot \, \wedge \tau_{n_k}^{r_0}))_{k\in\N}$ also converges weakly to $Z^{r_0}$. 
Furthermore, by~\cite[Thm~3.10.2, p.~148]{ek} and assumption~\eqref{ip1}, the process $Z^{r_0}$ is continuous,
i.e. the support of its law is contained in $\cC_d$.

Pick $r \in (0,r_0)$.
It follows from Lemmas~\ref{lem:first_hitting_time} and~\ref{lem:continuity_of_mapa} and the mapping theorem (see~\cite[p.~20]{bill})
that the joint convergence in~\eqref{eqn:joint_convergence}
holds for all but countably many $r < r_0$. 
Furthermore, from~\eqref{eqn:joint_convergence} we have that
$\| Z_{n_k} ( \, \cdot \, \wedge \tau_{n_k}^r ) \| \Rightarrow \| Z^{r_0} ( \, \cdot \, \wedge \tau^r ) \|$
for all but countably many $r < r_0$. By assumption in Theorem~\ref{thm:conditional_invariance}, 
the weak limit of $\| Z_{n_k} \|$ is a Bessel process. Hence, 
again by Lemmas~\ref{lem:first_hitting_time} and~\ref{lem:continuity_of_mapa}, the fact that a Bessel process
has continuous trajectories and the mapping theorem~\cite[p.~20]{bill},
the law of $\| Z^{r_0} ( \, \cdot \, \wedge \tau^r ) \|$ equals that of a Bessel process stopped at level $r$
for all but countably many $r < r_0$.
The final statement in the lemma is equivalent to saying that a Bessel process of dimension greater than one 
reaches every positive level with probability one. This is immediate in the transient case. In the recurrent
case it follows from the fact that the height of excursions away from zero is not bounded. 
\end{proof}


Define the function 
$F_{i,j}:\cD_d\times\RP\to\R$ by the formula 
$F_{i,j} (y , T) := \int_0^T a_{ij} (y (s) ) \ud s$
for any
$i,j \in \{1,\ldots, d\}$,
where
$a_{ij}$ is a coefficient in the generator $G$ in 
Theorem~\ref{thm:conditional_invariance}.

\begin{lemma}
\label{lem:a_integral_convergence}
Fix $r_0>0$. Then for all but countably many $r\in(0,r_0)$, 
the sequence of processes
$F_{i,j} ( Z_{n_k} , \, \cdot \, \wedge \tau_{n_k}^r)=(F_{i,j} ( Z_{n_k} , t \wedge \tau_{n_k}^r);t\geq 0)$ 
converges weakly to the process 
$F_{i,j} (Z^{r_0} , \, \cdot \, \wedge \tau^r)=(F_{i,j} ( Z^{r_0} , t \wedge \tau^r);t\geq 0)$ as $k \to \infty$
for any
$i,j \in \{1,\ldots, d\}$,
\end{lemma}

\begin{remark}
In the proof of~\cite[Thm~7.4.1, p.~355]{ek},
the statement of the lemma is used implicitly 
and follows directly from the continuity assumption on $a_{ij}$ in~\cite[Thm~7.4.1, p.~355]{ek}  (which implies that $F_{i,j}$ is itself continuous
at any continuous path) 
and the analogue of the the weak limit in~\eqref{eqn:joint_convergence}.
In our case the coefficient $a_{ij}$ is discontinuous at the origin and the process 
$\|Z^{r_0}\|$ may visit zero infinitely many times. Hence we must rely on the more 
detailed information about the limit law  
$\| Z^{r_0} ( \, \cdot \, \wedge \tau^r ) \|$.
In particular, we use the fact that the Bessel process of dimension greater than one is a continuous semimartingale
and apply the occupation times formula to quantify the amount of time it spends around zero. 
\end{remark}

\begin{proof}
Let $\eps >0$ and take smooth functions $\phi_1^\eps, \phi_2^\eps:\RP\to[0,1]$ satisfying 
$\phi_1^\eps (u) = 1$ for all $u \geq \eps$, 
$\phi_1^\eps (u) = 0$ for all $u \leq \eps/2$
and
$\phi_1^\eps (u) + \phi_2^\eps (u) = 1$ for all $u \in\RP$.
Let 
\[ F_{i,j}^{k,\eps} (x, T):= \int_0^T a_{ij} ( x(s) ) \phi_k^\eps ( \| x (s) \| ) \ud s,\qquad\text{where $k\in\{1,2\}$.}\] 
Then since $a_{ij}$ is continuous on $\R^d \setminus \{ \0 \}$ and $\phi_1^\eps$ is continuous and vanishes in a neighbourhood of 0,
we have that $F_{i,j}^{1,\eps}:\cD_d\times\RP\to\R$
is continuous at any point
$(x, T)\in\cC_d\times \RP$.
Hence~\eqref{eqn:joint_convergence} in Lemma~\ref{lem:a_integral} 
implies the convergence 
$F_{i,j}^{1,\eps} ( Z_{n_k} , \cdot \wedge \tau_{n_k}^r) \Rightarrow F_{i,j}^{1,\eps} (Z^{r_0} , \cdot \wedge \tau^r)$
for all but countably many $r < r_0$.

Consider now $F_{i,j}^{2,\eps} :\cD_d\times\RP\to\R$. Since $a_{ij}$ is globally bounded, there exists a constant $C>0$
such that 
\begin{equation}
\label{eq:F_2_first_bound}
|F_{i,j}^{2,\eps} ( x , T)|  \leq C \int_0^T \phi_2^\eps(\| x (s)\|) \ud s \qquad \forall (x, T)\in\cD_d\times \RP.
\end{equation}
By Lemma~\ref{lem:a_integral}, we may assume that 
$\| Z^{r_0} ( \, \cdot \wedge \tau^r ) \|$ is a Bessel process (of dimension greater than one) stopped at level $r$. 
The random field $(L_t(a))_{t,a\in\RP}$ of Bessel local times exists by~\cite[Ch.~VI, Thm~(1.7)]{ry} since the process is a continuous semimartingale
with the local martingale component equal to Brownian motion. 
Furthermore, it is well known that  $(L_t(a))_{t,a\in\RP}$ has a bi-continuous modification, i.e. the map 
$(t,a)\mapsto L_t(a)$ is a.s. continuous on $\RP^2$.
Then, by the occupation times formula~\cite[p.~224]{ry} and~\eqref{eq:F_2_first_bound} we get
\begin{equation}
\label{eq:LT_bound1}
\sup_{t\in\RP}|F_{i,j}^{2,\eps} ( Z^{r_0} , t \wedge \tau^r )| \leq C \int_0^{\tau^r} \phi_2^\eps( \| Z^{r_0} (s) \|) \ud s
= C \int_0^\eps \phi_2^\eps(a) L_{\tau^r}(a) \ud a ,
\end{equation}
since the quadratic variation of 
$\| Z^{r_0} ( \, \cdot \wedge \tau^r ) \|$
is dominated by that of
the Brownian motion and the support of 
$\phi_2^\eps$ is contained in $[0,\varepsilon]$.
Since 
$(x,t)\mapsto \int_0^t  \phi_2^\eps ( \| x (s) \| ) \ud s$ 
is continuous on 
$\cD_d\times\RP$,
Lemma~\ref{lem:a_integral}
and
the mapping theorem~\cite[p.~20]{bill} 
imply 
\begin{equation}
\label{eq:LT_bound2}
\sup_{t\in\RP}|F_{i,j}^{2,\eps}  ( Z_{n_k} , t \wedge \tau_{n_k}^r )| \leq C \int_0^{\tau_{n_k}^r} \phi_2^\eps ( \| Z_{n_k} (s)\|)  \ud s
\Rightarrow C \int_0^{\tau^r}\phi_2^\eps ( \| Z^{r_0} (s) \|)  \ud s . 
\end{equation}

If the convergence in the lemma fails, there exists a bounded uniformly continuous map
$h:\cC_1\to\R$ (with the uniform topology on $\cC_1$) and $\epsilon_0$ such that 
\begin{equation}
\label{eq:Contradictory_Statement}
\lvert \Exp h\circ F_{i,j}   (Z^{r_0} , \, \cdot \, \wedge \tau^r)-
\Exp h\circ F_{i,j}    ( Z_{n_k} , \, \cdot \, \wedge \tau_{n_k}^r)\rvert>\epsilon_0 \qquad \forall k\in\N,
\end{equation}
where we have passed to a subsequence without changing the notation. 
Then there exists
$\delta>0$
such  that if
$x,y\in\cC_1$ satisfy 
$\sup_{t\in\RP}|x(t)-y(t)|<\delta$,
then 
$|h(x)-h(y)|<\epsilon_0/6$.
Fix a monotone sequence 
$\varepsilon_n\downarrow0$
and note that we may assume that $\delta/C$ is not an atom of 
$\int_0^{\eps_n} \phi_2^{\eps_n}(a) L_{\tau^r}(a) \ud a$ 
for any $n\in\N$, where $C$ is the constant in~\eqref{eq:LT_bound1} and~\eqref{eq:LT_bound2}. 
Note that by the inequality in~\eqref{eq:LT_bound2} and the fact that $F_{i,j}   =F_{i,j}^{1,\varepsilon}+F_{i,j}^{2,\varepsilon}$ we have 
$$
\lvert \Exp h\circ F_{i,j} ( Z_{n_k} , \, \cdot \, \wedge \tau_{n_k}^r) -  \Exp h\circ F_{i,j}^{1,\eps} ( Z_{n_k} , \, \cdot \, \wedge \tau_{n_k}^r)\rvert
\leq \epsilon_0/6 + C_h \Pr\left[ \int_0^{\tau_{n_k}^r} \phi_2^\eps ( \| Z_{n_k} (s)\|)  \ud s >\delta/C \right]
$$
any $\varepsilon>0$ and some constant $C_h>0$.
By the dominated convergence theorem there exists $\eps_n$ such that 
\begin{equation}
\label{eq:OTF_1}
\Pr\left[\int_0^{\eps_n} \phi_2^{\eps_n}(a) L_{\tau^r}(a) \ud a >\delta/C\right]< \frac{\epsilon_0}{12 C_h}.
\end{equation}
By Lemma~\ref{lem:a_integral} and since $\delta/C$ is not an atom of 
$\int_0^{\eps_n} \phi_2^{\eps_n}(a) L_{\tau^r}(a) \ud a$, there exists 
$k_0\in\N$ such that 
for all $k\geq k_0$
we have
$$
\Pr\left[ \int_0^{\tau_{n_k}^r} \phi_2^{\eps_n} ( \| Z_{n_k} (s)\|)  \ud s >\delta/C \right] < 
\Pr\left[\int_0^{\eps_n} \phi_2^{\eps_n}(a) L_{\tau^r}(a) \ud a >\delta/C\right]+ \frac{\epsilon_0}{12 C_h}
<\frac{\epsilon_0}{6 C_h}.
$$  
Hence it holds that 
\begin{equation}
\label{eq:bound_A1}
\lvert \Exp h\circ F_{i,j}  ( Z_{n_k} , \, \cdot \, \wedge \tau_{n_k}^r) -  \Exp h\circ F_{i,j}^{1,\eps_n} ( Z_{n_k} , \, \cdot \, \wedge \tau_{n_k}^r)\rvert < \epsilon_0/3 \qquad \forall k\geq k_0.
\end{equation}
Since we already know 
$F_{i,j}^{1,\eps} ( Z_{n_k} , \cdot \wedge \tau_{n_k}^r) \Rightarrow F_{i,j}^{1,\eps} (Z^{r_0} , \cdot \wedge \tau^r)$,
there exists $k_1\geq k_0$, such that 
\begin{equation}
\label{eq:bound_A2}
\lvert \Exp h\circ F_{i,j}^{1,\eps_n} ( Z_{n_k} , \, \cdot \, \wedge \tau_{n_k}^r) -  \Exp h\circ F_{i,j}^{1,\eps_n} ( Z^{r_0}  , \, \cdot \, \wedge \tau^r)\rvert < \epsilon_0/3 \qquad \forall k\geq k_1.
\end{equation}
Similarly, by~\eqref{eq:LT_bound1} and~\eqref{eq:OTF_1}, we get 
\begin{equation*}
\lvert \Exp h\circ F_{i,j} ( Z^{r_0} , \, \cdot \, \wedge \tau^r) -  \Exp h\circ F_{i,j}^{1,\eps_n} ( Z^{r_0} , \, \cdot \, \wedge \tau^r)\rvert < \epsilon_0/6 + 
C_h \Pr\left[\int_0^{\eps_n} \phi_2^{\eps_n}(a) L_{\tau^r}(a) \ud a >\delta/C\right]< \frac{\epsilon_0}{3}.
\end{equation*}
This inequality, coupled with~\eqref{eq:bound_A1},~\eqref{eq:bound_A2} and the triangle inequality, contradicts
the statement in~\eqref{eq:Contradictory_Statement}, which proves the lemma. 
\end{proof}

Lemma~\ref{lem:a_integral_convergence} is key in proving that the processes in~\eqref{eq:Prod_martingale}
are true martingales, which will in turn imply that the limit $Z^{r_0}$
is a solution of the  stopped martingale problem. We establish the martingale property in
the next lemma. 

\begin{lemma}
\label{lem:ip2}
Fix $r_0>0$ and pick $r\in(0,r_0)$. 
Then the components of the process  
$Z^{r_0}( \, \cdot \, \wedge \tau^r)$ are martingales. Moreover,
for any $i,j \in \{1,\ldots, d\}$, the following process is a martingale:
\begin{equation}
\label{eq:Prod_martingale}
Z^{r_0,i} ( \, \cdot \, \wedge \tau^r) Z^{r_0,j} ( \, \cdot \, \wedge \tau^r) - \int_0^{ \cdot \, \wedge \tau^r} a_{ij} (Z^{r_0} (s) ) \ud s  
\end{equation}
\end{lemma}

\begin{proof}
Recall that 
the sequence $(\tilde Z^{r_0}_n)_{n\in\N}$, 
defined in~\eqref{eq:tilda_procs},
is relatively compact
by Lemma~\ref{lem:ip1}.  
Furthermore, the process
$Z^{r_0}$
was defined as a weak limit of a convergent subsequence 
$(\tilde Z^{r_0}_{n_k})_{k\in\N}$.
For any $i,j\in\{1,\ldots,d\}$ the processes 
$\tilde Z_{n_k}^{r_0,i}$ and $\tilde A_{n_k}^{ij}$
(see~\eqref{eq:tilda_procs} for definition)
give rise to martingales 
$\tilde Z_{n_k}^{r_0,i} \tilde Z_{n_k}^{r_0,j} - \tilde A_{n_k}^{ij}$
(see the argument following the display in~\eqref{eq:Compact_bound}).
Hence, for any index $i\in\{1,\ldots,d\}$ and $k\in\N$, we have that
\begin{align*} 
\Exp [ (\tilde Z_{n_k}^{r_0,i} (t) )^2 ]
=  \Exp [ Z_{n_k}^i(0)^2 ]  + \Exp \left[ \tilde A_{n_k}^{ii} (t) \right]\qquad \text{for all $t\geq0$.} 
\end{align*}
Thus by~\eqref{ip2},~\eqref{eq:inv1} and the assumption on the square integrability of
$Z_{n_k}(0)$ in Theorem~\ref{thm:conditional_invariance}, we have that 
$\sup_{k\in\N} \Exp [ \| \tilde Z_{n_k}^{r_0} (t) \|^2 ] < \infty$ 
and hence the family $(\|\tilde Z_{n_k}^{r_0}(t)\|)_{k\in\N}$ is uniformly integrable 
for every $t\geq0$.

To prove that the components of
$Z^{r_0}$
are martingales with respect to the natural filtration 
$(\sigma(Z^{r_0}_u:u\in[0,s]), s \in \RP)$,
note first that each $\sigma$-algebra 
$\sigma(Z^{r_0}_u:u\in[0,s])$ is generated by the $\pi$-system of events of the form 
$\{Z^{r_0}(s_1)\in A_1,\ldots,Z^{r_0}(s_p)\in A_p\}$ for any
$p\in\N$ and $s_1,\ldots,s_p\in[0,s]$, where $A_1,\ldots, A_p$ are rectangular boxes in $\R^d$.
Hence it is sufficient to show that for any $0\leq s_1 < \ldots s_p \leq s <t$ 
and a non-negative, bounded, continuous $f:\R^d\otimes\R^p \to\R$
it holds that
\begin{equation}
\label{eq:MArt_Cond}
\Exp [\left(Z^{r_0,i}(t)-Z^{r_0,i}(s)\right)f(Z^{r_0}(s_1),\ldots,Z^{r_0}(s_p)) ]=0.
\end{equation}
By the Skorohod representation theorem~\cite[Thm 3.1.8, p.~102]{ek}
we may assume that the zero mean random variables 
$\left(\tilde Z_{n_k}^{r_0,i}(t)-\tilde Z_{n_k}^{r_0,i}(s)\right)f(\tilde Z_{n_k}^{r_0}(s_1),\ldots,\tilde Z_{n_k}^{r_0}(s_p))$
converge almost surely as $k\to\infty$ to the random variable  in~\eqref{eq:MArt_Cond}.
Furthermore, since $f$ is bounded, this sequence is uniformly integrable by the argument in the 
first paragraph of this proof. This implies the convergence in $L^1$
and hence the identity in~\eqref{eq:MArt_Cond}. 
Since $Z^{r_0}$ is a martingale, so is $Z^{r_0}( \cdot \, \wedge \tau^r)$ for any $r\in(0,r_0)$.

Consider now the process in~\eqref{eq:Prod_martingale}. We start by establishing the following fact.\\
\noindent \textbf{Claim.}
For any $i,j \in \{1,\ldots, d\}$ and all but countably many $r\in(0,r_0)$ it holds that 
$$\tilde Z_{n_k}^{r_0,i}(\cdot\wedge\tau^r_{n_k}) \tilde Z_{n_k}^{r_0,j}(\cdot\wedge\tau^r_{n_k})-\tilde A_{n_k}^{ij}(\cdot\wedge\tau^r_{n_k}) 
\Rightarrow 
Z^{r_0,i} ( \cdot \wedge \tau^r) Z^{r_0,j} ( \cdot \wedge \tau^r) - \int_0^{ \cdot \, \wedge \tau^r} a_{ij} (Z^{r_0} (s) ) \ud s, $$ 
where the stopping times 
$\tau^r_n=\tau^r(Z_n)$
and
$\tau^r=\tau^r(Z^{r_0})$
are as in Lemma~\ref{lem:a_integral}.

\medskip

\noindent \textit{Proof of Claim.}
By definition it holds that 
$\tilde Z_{n_k}^{r_0}\Rightarrow Z^{r_0}$. Hence, as in the proof of Lemma~\ref{lem:a_integral},
since $Z^{r_0}$ has continuous trajectories  
it follows from Lemmas~\ref{lem:first_hitting_time} and~\ref{lem:continuity_of_mapa} 
and the mapping theorem~\cite[p.~20]{bill}
that 
$\tilde Z_{n_k}^{r_0}(\cdot\wedge\tau^r_{n_k})\Rightarrow Z^{r_0}( \cdot \wedge \tau^r)$.
Thus it holds that  
$\tilde Z_{n_k}^{r_0,i}(\cdot\wedge\tau^r_{n_k}) \tilde Z_{n_k}^{r_0,j}(\cdot\wedge\tau^r_{n_k})
\Rightarrow 
Z^{r_0,i} ( \cdot \wedge \tau^r) Z^{r_0,j} (  \cdot \wedge \tau^r)$. 

To prove the claim it therefore suffices to show that 
$\tilde A_{n_k}^{ij}(\cdot\wedge\tau^r_{n_k}) 
\Rightarrow 
\int_0^{ \cdot \, \wedge \tau^r} a_{ij} (Z^{r_0} (s) ) \ud s$.
With this in mind, we note  that   
\begin{equation}
\label{eq:decomposition}
\tilde A_{n_k}^{ij}(\cdot\wedge\tau^r_{n_k}) = U_k + V_k +F_{i,j} ( Z_{n_k} , \, \cdot \, \wedge \tau_{n_k}^r),
\end{equation}
where 
$ U_k:=\tilde A_{n_k}^{ij}(\cdot\wedge\tau^r_{n_k}) - A_{n_k}^{ij}(\cdot\wedge\tau^r_{n_k}) \toP0$
by~\eqref{inv:eta}--\eqref{eq:tilda_procs}
and
$V_k:=A_{n_k}^{ij}(\cdot\wedge\tau^r_{n_k}) - F_{i,j} ( Z_{n_k} , \, \cdot \, \wedge \tau_{n_k}^r)\toP 0$
by the assumption in~\eqref{ip3}. 
The representation of 
$\tilde A_{n_k}^{ij}(\cdot\wedge\tau^r_{n_k})$
in~\eqref{eq:decomposition},~\cite[Cor.~3.3.3, p.~110]{ek} and
Lemma~\ref{lem:a_integral_convergence} imply 
\begin{equation}
\label{eq:Int_Conv}
\tilde A_{n_k}^{ij}(\cdot\wedge\tau^r_{n_k}) 
\Rightarrow 
\int_0^{ \cdot \, \wedge \tau^r} a_{ij} (Z^{r_0} (s) ) \ud s,
\end{equation}
and the claim follows. 

\medskip

Since $\tilde Z_{n_k}^{r_0,i} \tilde Z_{n_k}^{r_0,j}-\tilde A_{n_k}^{ij}$
is a martingale by the argument following~\eqref{eq:Compact_bound},
the stopped process 
$M_{k}:=\tilde Z_{n_k}^{r_0,i}(\cdot\wedge\tau^r_{n_k}) \tilde Z_{n_k}^{r_0,j}(\cdot\wedge\tau^r_{n_k})-\tilde A_{n_k}^{ij}(\cdot\wedge\tau^r_{n_k})$
is also a martingale for every $k\in\N$. Hence the process in~\eqref{eq:Prod_martingale} will be a martingale by the analogous
argument to the one that established the martingale property of $Z^{r_0,i}$ above, if we prove that for any $t\geq0$ the family of random variables
$\{M_k(t):k\in\N\}$ is uniformly integrable. 
With this in mind, note that 
$2|\tilde A^{ij}_{n_k}|\leq \tilde A^{ii}_{n_k}+\tilde A^{jj}_{n_k}$ 
since the matrix
$\tilde A_{n_k}$  is non-negative definite.
The elementary inequality 
$2|\tilde Z_{n_k}^{r_0,i} \tilde Z_{n_k}^{r_0,j}|\leq  (\tilde Z_{n_k}^{r_0,i})^2 +(\tilde Z_{n_k}^{r_0,j})^2$ 
implies 
$$
|M_k(t)|
\leq 
\tilde Z_{n_k}^{r_0,i}(t\wedge\tau^r_{n_k})^2
+
\tilde Z_{n_k}^{r_0,j}(t\wedge\tau^r_{n_k})^2
+
\tilde A^{ii}_{n_k}(t\wedge\tau^r_{n_k})+\tilde A^{jj}_{n_k}(t\wedge\tau^r_{n_k}). 
$$
Since the sequence
$(\tilde A^{ii}_{n_k}(t\wedge\tau^r_{n_k})+\tilde A^{jj}_{n_k}(t\wedge\tau^r_{n_k}))_{k\in\N}$
is bounded in $L^1$
by~\eqref{ip2} and~\eqref{eq:inv1}, 
$\{M_k(t):k\in\N\}$ 
will be uniformly integrable 
if 
$\{\tilde Z_{n_k}^{r_0,i}(t\wedge\tau^r_{n_k})^2:k\in\N\}$
is uniformly integrable 
for all $i\in\{1,\ldots,d\}$.
Note that by~\eqref{eq:Compact_bound}, for any $r\in(0,r_0)$, we have that 
\begin{equation*}
\tilde Z_{n_k}^{r_0,i}(t\wedge \tau^r_{n_k})^2 \leq 3 \left( \sup_{n\in\N}\|Z_n(0)\|^2 +  4r_0^2 +\sup_{0 \leq s \leq t \wedge \tau_{n_k}^r} \left\| Z_{n_k} (s) - Z_{n_k} (s-) \right\|^2 \right).
\end{equation*}
The right-hand side converges in $L^1$ by~\eqref{ip1}.
Hence
$\{\tilde Z_{n_k}^{r_0,i}(t\wedge\tau^r_{n_k})^2:k\in\N\}$
is uniformly integrable and the lemma follows for all but countably many $r\in(0,r_0)$. 
Note however that there exist $r_n\uparrow r_0$ such that the martingale properties in the
lemma hold for all $r_n$. Since a stopped martingale is a martingale, the lemma follows for
all $r\in(0,r_0)$. 
\end{proof}

\begin{proof}[Proof of Theorem~\ref{thm:conditional_invariance}.]
By Lemma~\ref{lem:ip2} and It\^o's formula for continuous semimartingales,
the process $Z^{r_0}$ constructed in the proof of Lemma~\ref{lem:a_integral}
solves the stopped martingale problem (see~\cite[p.~216]{ek} for
the precise definition) $(G,v,\{x\in\R^d : \|x\|<r\})$ for any $r\in(0,r_0)$. 
Since the martingale problem $(G,v)$ is well-posed, by~\cite[Thm~4.6.1, p.~216]{ek}
there exists a unique solution to the stopped martingale problem. 
Furthermore, if 
$Z$ 
is a solution 
of the martingale
problem $(G,v)$ on $\cD_d$, then 
$Z (\cdot \wedge \tau^r(Z))$ 
must be a solution to the stopped martingale problem 
by the optional sampling theorem 
(cf.~\cite[pp.~216--217]{ek}), 
where $\tau^r(Z)$ is defined in~\eqref{eq:contact_time_def}.
In particular (since $r_0>0$ is arbitrary) for all but countably many $r>0$, any subsequence of
$Z_n ( \cdot \, \wedge \tau_n^r )$, where $\tau_n^r$ is defined in Lemma~\ref{lem:a_integral},
has by Lemma~\ref{lem:a_integral} a further subsequence 
that converges weakly to the law of the process
$Z (\cdot \wedge \tau^r(Z))$. 
It hence follows that the entire sequence must be convergent,
$Z_n ( \cdot \, \wedge \tau_n^r )\Rightarrow Z (\cdot \wedge \tau^r(Z))$,
for all but at most countably many $r>0$.

In order to prove that this implies 
$Z_n \Rightarrow Z $, 
note that 
$\tau^r(Z)\to\infty$ a.s. as $r\to\infty$, since the paths of $Z$ are in
$\cD_d$ (in fact in $\cC_d$), and 
it holds that 
$$d(Z, ,Z( \, \cdot \, \wedge \tau^r(Z))\leq e^{-\tau^r(Z)}\to0\qquad \text{a.s. as $r\to\infty$,}$$
where 
$d:\cD_d\times\cD_d\to\RP$, defined  in~\cite[Eq.~(5.2), p.~117]{ek}, is the Skorohod metric. 
Pick any uniformly continuous and bounded map
$h:\cD_d\to\R$.
This class of maps is convergence determining~\cite[Prop.~3.4.4, p.~112]{ek}.
Pick $\eps>0$ and let $\delta\in(0,1)$
satisfy: if $d(x,y)<\delta$ then $|h(x)-h(y)|<\eps/6$.
Let
$C_h>0$ satisfy
$\sup_{x\in\cD_d} |h(x)|<C_h$.
By
Lemmas~\ref{lem:first_hitting_time} and~\ref{lem:continuity_of_mapa} and the mapping theorem (see~\cite[p.~20]{bill}),
there exists $r>0$
such that 
$\tau_n^r \Rightarrow \tau^r(Z)$
and 
$\Pr[\tau^r(Z)\leq \log(1/\delta)]<\eps/(12C_h)$.
Without loss of generality we may assume that 
$\log(1/\delta)$ is not an atom of $\tau^r(Z)$.
Hence we may 
choose $N_0\in\N$ such that for all $n\geq N_0$ we have 
$\Pr[\tau^r_n\leq \log(1/\delta)]<\eps/(6C_h)$
and 
$|\Exp h(Z_n(\cdot \wedge \tau^r_n)) - \Exp h(Z (\cdot \wedge \tau^r(Z)))| <\eps/6$.
This implies the inequalities 
\begin{align*}
|\Exp h(Z_n) - \Exp h(Z)| 
\leq &
|\Exp h(Z_n) - \Exp h(Z_n(\cdot \wedge \tau^r_n))| +
|\Exp h(Z_n(\cdot \wedge \tau^r_n)) - \Exp h(Z (\cdot \wedge \tau^r(Z)))| \\
&  +
|\Exp h(Z (\cdot \wedge \tau^r(Z))) - \Exp h(Z) | \\
\leq &  \Pr[\tau^r_n>  \log(1/\delta)] \frac{\eps}{6} + \frac{\eps}{3} + 
\Pr[\tau^r(Z)>  \log(1/\delta)] \frac{\eps}{6} + \frac{\eps}{6} + \frac{\eps}{6} \leq \eps.  
\end{align*}
\end{proof}

\subsection{Proof of Theorem~\ref{thm:invariance}}
\label{subsec:Proof_of_Thms_51_and_24}

Recall the definition of the scaled process $\tX_n=(\tX_n(t),t\geq0)$ in~\eqref{eq:scaled_walk} in terms of the chain $X=(X_m,m\in\ZP)$,
$\tX_n (t) = n^{-1/2} X_{\lfloor nt \rfloor}$ for $t \in \RP$.
Theorem~\ref{thm:invariance} now follows from Theorem~\ref{thm:conditional_invariance} 
and the main result of~\cite{gmw}:

\begin{lemma}
\label{lem:radial_invariance}
Suppose that \eqref{ass:moments}--\eqref{ass:cov_form} 
hold. Without loss of generality assume that $U =1$. Then $\| \tX_n \|$ converges weakly to the $V$-dimensional Bessel process
started at $0$.
\end{lemma} 

Define
$ A_n (t) = \frac{1}{n} \sum_{m=0}^{\lfloor nt \rfloor -1} M ( X_m )$ ,
where $M(\bx)$ is the covariance matrix of the increment of the chain at $\bx\in\X$
and, as before,  we take $\sum_{m=0}^{-1}=0$. 
Define $Z_n := \tX_n$ 
and note that 
$Z^i_n  Z^j_n  - A^{ij}_n$
is a local martingale for all $i,j\in\{1,\ldots,d\}$.
By Lemma~\ref{lem:radial_invariance} 
we have
$\|Z_n\|\Rightarrow \Bes^V(0)$
as $n\to\infty$.
Let 
$a(\bx):=\sigma^2(\hat \bx)$
be a non-negative definite matrix valued function on 
$\R^d$, where 
$\sigma^2$ satisfies~\eqref{ass:cov_limit}--\eqref{ass:radial-evec}.
Let the generator 
$G$
be defined as in Theorem~\ref{thm:conditional_invariance}
for this coefficient $a$.
Then  the $\cC_d$ martingale problem 
for $(G,\delta_{\0})$ is well-posed 
by Theorem~\ref{t:well-posedness},
where
$\delta_{\0}$
denotes the delta measure on $\R^d$ concentrated at the origin. 
In order to apply 
Theorem~\ref{thm:conditional_invariance},
it remains to establish the assumptions~\eqref{ip1}, \eqref{ip2} and \eqref{ip3}
for $Z_n$ and $A_n$.
Condition~\eqref{ip1}  
follows
from~\cite[Lem.~2]{gmw}.
Since by assumption  
$|M_{ij}(\by)|\leq \sup_{\bx\in\X:\|x\|\geq r} \| M (\bx) \| <\infty$
for a sufficiently large $r>0$ and any $\by\in\X$ with $\|\by\|\geq r$, 
condition~\eqref{ip2} follows from 
$\lim_{n \to \infty} \frac{1}{n} \Exp \max_{0 \leq m \leq \lfloor n T \rfloor} \left| M_{ij} (X_m ) \right| = 0$.
Finally, 
condition~\eqref{ip3} is verified by~\cite[Lem.~5]{gmw} for the coordinate functional $\phi:\R^d\otimes\R^d\to\R$, $\phi(B)=B_{ij}$.
Thus Theorem~\ref{thm:conditional_invariance} applies, implying Theorem~\ref{thm:invariance}.

\appendix




\section{Laplace-Beltrami operator on $(\Sp{d-1},g)$}
\label{sec:manifolds}
We use the definitions and notation from Sections~\ref{subsubsec:RiemannianGeom}. 

\begin{lemma}\label{lem:BL-G-in-lc}  In the local coordinates on $H_q^\pm$, the Laplace-Beltrami operator on
  $(\Sp{d-1},g)$, equals
\[
\Delta_g f  = \sum_{i,j\in [q]} g^{ij} \Big(E_i (E_j(f))
-\sum_{k\in[q]}\Gamma_{ij}^k E_k( f)\Big).
\]
\end{lemma}

\begin{proof}
We first establish the formula for $\Delta_g$ in the local coordinates on $H_q^\pm$.
Note that $E_i$, for $i\in[q]$, defined in Section~\ref{subsubsec:RiemannianGeom}, is a vector field on
$H_q^\pm$~\cite[p.~248]{iw}. Put differently, $E_i$ is a smooth section of the (product) bundle $TH_q^\pm$.
Since the Levi-Civita connection
$\nabla$ on $\Sp{d-1}$, constructed in~\cite[Thm~4.3.1]{jost},
is a local operator,
the equality 
$\nabla_{E_i}E_j = \sum_{k\in[q]} \Gamma_{ij}^k E_k$
for all $i,j\in[q]$
follows from~\cite[Cor.~4.3.1]{jost}.
Any vector field $X$ on $\Sp{d-1}$, restricted  to 
$H_q^\pm$, takes the form
$X = \sum_{j\in[q]}X_jE_j$, where $X_j$, $j\in[q]$, are a smooth function on $H_q^\pm$.
By the product rule~\cite[Def.~4.1.1(ii)]{jost} 
we get 
$\nabla_{E_i} X = \sum_{j\in[q]} (E_i( X_j) E_j + X_j \sum_{k\in[q]} \Gamma_{ij}^k E_k)$.
By the definition of $\divg X$ given above, 
this implies 
$\divg X =  \sum_{i\in[q]} (E_i (X_i) + \sum_{j\in[q]} \Gamma_{ij}^i X_j)$.

Note that
$(\det(A+H)- \det A -\det A \trace(A^{-1}H))/\|H\|\to0$ as $\|H\|\to0$
for any invertible square matrix $A$ (here $H$ is a square matrix of the 
same dimension as $A$), 
i.e.  the derivative of the determinant at a non-singular matrix $A$
takes the form $D\det(A)H= \trace(A^{-1}H)\det A$.
It hence follows that
$\sum_{i\in[q]} \Gamma_{ij}^i = \frac{1}{2} \sum_{i,\ell\in[q]} g^{i\ell} E_j( g_{i\ell})
=\frac{1}{2} \trace (G^{-1} E_j(G)) = (1/\sqrt{\det G}) E_j( \sqrt{\det G})$,
where $G$ (resp. $E_j(G)$) denotes the matrix 
$(g_{i\ell})_{i,\ell\in[q]}$
(resp.
$(E_j(g_{i\ell}))_{i,\ell\in[q]}$),
implying 
$\divg X = (1/\sqrt{\det G}) \sum_{i\in[q]} E_i (X_i \sqrt{\det G}).$
In Section~\ref{subsubsec:RiemannianGeom} we defined 
$\grad f=\tilde g^{-1} (df)$ for any $f\in\cC^\infty(\Sp{d-1},\R)$.
Hence, in the local coordinates, we obtain
$\grad f = \sum_{i,j\in[q]} g^{ij} E_j( f) E_i$
and 
$\Delta_g f = (1/\sqrt{\det G}) \sum_{i,j\in[q]} E_i (\sqrt{\det G}g^{ij} E_j( f))$.
Since 
$(g_{ij})_{i,j\in[q]}$ 
and 
$(g^{ij})_{i,j\in[q]}$ 
are inverses,
differentiation implies
$E_k(g^{im}) =-\sum_{\ell,j\in[q]} g^{ij}g^{\ell m} E_k (g_{\ell j})
=-\sum_{\ell\in[q]} g^{\ell m} \Gamma_{\ell k}^i - \sum_{j\in[q]}g^{ij} \Gamma_{jk}^m$
for all $i,m,k\in[q]$,
where the second equality follows from the identity
$E_k (g_{j\ell}) = (E_k (g_{j\ell})+E_j (g_{k\ell}) - E_\ell (g_{kj}))/2
+ (E_k( g_{j\ell})+E_\ell (g_{kj}) - E_j (g_{k\ell}))/2$.
In particular, 
we get
$E_i (g^{ki}) = - \sum_{j\in[q]}(g^{ij}\Gamma_{ji}^k +g^{kj}\Gamma_{ji}^i)$.
By the formula above for $E_i (\sqrt{\det G})$, the following identity 
holds for all $k\in[q]$,
$\sum_{i,j\in[q]}g^{ij}\Gamma_{ij}^k = - (1/\sqrt{\det G})\sum_{i\in[q]} E_i(\sqrt{\det G} g^{ik})$,
yielding 
the formula for $\Delta_g$. 
\end{proof}

\section*{Acknowledgements}
NG and AW were supported in part by the EPSRC grant EP/J021784/1.
AM is supported by the EPSRC grant EP/P003818/1 and a Fellowship at The Alan Turing Institute, 
sponsored by the Programme on Data-Centric Engineering  funded by Lloyd's Register Foundation.

\end{document}